\documentclass{amsart}
\usepackage{xr}
\usepackage{hyperref}
\usepackage{amssymb}
\usepackage{amsmath}
\usepackage{amsfonts}
\usepackage{url}

\newtheorem{theorem}{Theorem}[section]
\theoremstyle{plain}

\newtheorem{corollary}[theorem]{Corollary}

\newtheorem{definition}[theorem]{Definition}

\newtheorem{lemma}[theorem]{Lemma}
\newtheorem{notation}{Notation}

\newtheorem{proposition}[theorem]{Proposition}
\newtheorem{remark}[theorem]{Remark}

\numberwithin{equation}{section}

\begin{document}
\title[On Kato-Sobolev type spaces]{On Kato-Sobolev type spaces}
\author{Gruia Arsu}
\address{Institute of Mathematics of The Romanian Academy\\
P.O. Box 1-174\\
RO-70700\\
Bucharest \\
Romania}
\email{Gruia.Arsu@imar.ro, agmilro@yahoo.com, agmilro@gmail.com}

\begin{abstract}
We study an increasing family of spaces $\left\{ \mathcal{B}_{k}^{p}\right\}
_{1\leq p\leq \infty }$ by adapting the techniques used in the study of
Beurling algebras by Coifman and Meyer \cite{Meyer}. A Wiener-L\'{e}vy
theorem for $\mathcal{B}_{k}^{\infty }$ algebras is proved based on an
integral representation formula belonging A. P. Calder\'{o}n. Also we study
the Schatten-von Neumann properties of pseudo-differential operators with
symbols in the spaces $\mathcal{B}_{k}^{p}$ spaces.
\end{abstract}

\maketitle
\tableofcontents

\section{Introduction}

Kato-Sobolev spaces $\mathcal{H}_{\mathtt{ul}}^{s}$ were introduced in \cite%
{Kato} by Tosio Kato and are known as uniformly local Sobolev spaces. The
uniformly local Sobolev spaces can be seen as a convenient class of
functions with the local Sobolev property and certain boundedness at
infinity. The same philosophy can explain the notion of Wiener amalgam
spaces: that is the local behavior is given by the local component, while
the global component (determines) sets out how the local pieces behave
together. In this paper we study a class of spaces which generalizes the
Kato-Sobolev spaces. As we noted in \cite{Arsu 1}, Kato-Sobolev spaces are
particular cases of Wiener amalgam spaces with local component $\mathcal{H}^{%
\mathbf{s}}$ and global component $L^{p}$. Wiener amalgam spaces were
introduced by Hans Georg Feichtinger in 1980. Allowing more general weight
functions, in this paper we consider as local component the spaces $\mathcal{%
B}_{k}=B_{2,k}$ similar to those introduced by Lars H\"{o}rmander (see \cite%
{Hormander} vol. 2) and we preserve the global component $L^{p}$. Concerning
the weight function $k$ we shall make a hypothesis which ensures that $%
\mathcal{B}_{k}$ is a module over $\mathcal{BC}^{\infty }$ (see the
notations). Most of the results proved in the case of Kato-Soblev spaces are
preserved. In this paper we study some multiplication properties of
Kato-Sobolev type spaces and we prove Schatten-von Neumann class properties
for pseudo-differential operators with symbols in the spaces $\mathcal{B}%
_{k}^{p}$. Also we develop an analytic functional calculus for Kato-Sobolev
algebras based on an integral representation formula of A. P. Calder\'{o}n.
This part corresponds to the section of \cite{Kato} where the invertible
elements of the algebra $\mathcal{H}_{\mathtt{ul}}^{s}$ are determined and
which has as main result a Wiener type lemma for $\mathcal{H}_{\mathtt{ul}%
}^{s}$. In our case, the main result is the Wiener-L\'{e}vy theorem for $%
\mathcal{B}_{k}^{\infty }$ algebras. This theorem allows a spectral analysis
of these algebras. Besides the properties of the spaces $B_{p,k}$, the main
techniques we use in the study of these spaces are inspired by those used in
the study of Beurling algebras by Coifman and Meyer \cite{Meyer}. In Section
2 we recall some properties of the spaces $B_{p,k}$ and we establish the
main technical result used in paper by adapting the techniques of Coifman
and Meyer. In Section 3 we introduce and we study an increasing family of
spaces $\left\{ \mathcal{B}_{k}^{p}\right\} _{1\leq p\leq \infty }$. Here
the Kato-Sobolev type spaces $\mathcal{B}_{k}^{p}$ are defined as Wiener
amalgam spaces with local component $\mathcal{B}_{k}$ and global component $%
L^{p}$, i.e. $\mathcal{B}_{k}^{p}=W\left( \mathcal{B}_{k},L^{p}\right) $.
The weak form of Wiener-L\'{e}vy theorem for $\mathcal{B}_{k}^{\infty }$
algebras is established in Section 4, then it is introduced a class of
weight functions for which it can be proved Wiener-L\'{e}vy theorem for $%
\mathcal{B}_{k}^{\infty }$. The Schatten-von Neumann class properties for
pseudo-differential operators with symbols in the spaces $\mathcal{B}%
_{k}^{p} $ are presented in the last section.

\textit{Some Notations}. Throughout the paper we are going to use the same
notations as in \cite{Hormander} for the usual spaces of functions and
distributions.

$\widehat{u}=\mathcal{F}(u)$ is the Fourier transform of $u$.

\textquotedblleft $Cst$\textquotedblright\ will always stand for some
positive constant which may change from one inequality to the other.

$u\approx v$ means that $u/v$ and $v/u$ are bounded.

If $m$ is an integer $\geq 0$ or $m=\infty $, then $\mathcal{BC}^{m}\left( 
\mathbb{R}^{n}\right) $ is the space of bounded functions in $\mathbb{R}^{n}$
with bounded derivatives up to the order $m$ with the (semi)norms $%
\left\Vert f\right\Vert _{\mathcal{BC}^{l}}=\max_{\left\vert \alpha
\right\vert \leq l}\sup_{x\in \mathbb{R}^{n}}\left\Vert \partial ^{\alpha
}f\left( x\right) \right\Vert <\infty $, $l<m+1$.

$\mathcal{C}_{\infty }\left( \mathbb{R}^{n}\right) $ is the space of
continuous functions vanishing at infinity and $\mathcal{C}_{pol}^{\infty
}\left( \mathbb{R}^{n}\right) $ is the space of smooth functions with
derivatives\textit{\ }of polynomial growth.

$[x]$ denotes the integral part of the real number $x$.

$\left\langle \cdot \right\rangle $ is the function $\left\langle \cdot
\right\rangle :\mathbb{R}^{n}\rightarrow \mathbb{R}$, $\left\langle \xi
\right\rangle =\left( 1+\left\vert \xi \right\vert ^{2}\right) ^{1/2}$, $\xi
\in \mathbb{R}^{n}$.

\section{The spaces $\mathcal{B}_{k}\equiv B_{2,k}$}

The spaces $B_{p,k}\left( \mathbb{R}^{n}\right) $ and are defined
essentially as inverse Fourier transforms of $L^{p}$ spaces with respect to
appropriate densities.

\begin{definition}
$\left( \mathtt{a}\right) $ A positive measurable function $k$ defined in $%
\mathbb{R}^{n}$ will be called a weight function of polynomial growth if
there are positive constants $C$ and $N$ such that 
\begin{equation}
k\left( \xi +\eta \right) \leq C\left\langle \xi \right\rangle ^{N}k\left(
\eta \right) ,\quad \xi ,\eta \in \mathbb{R}^{n}.  \label{kh2}
\end{equation}%
The set of all such functions $k$ will be denoted by $\mathcal{K}%
_{pol}\left( \mathbb{R}^{n}\right) $.

$\left( \mathtt{b}\right) $ For a weight function of polynomial growth $k$,
we shall write $M_{k}\left( \xi \right) =C\left\langle \xi \right\rangle
^{N} $, where $C$, $N$ are the positive constants that define $k$.
\end{definition}

\begin{remark}
$\left( \mathtt{a}\right) $ An immediate consequence of Peetre's inequality
is that $M_{k}$ is weakly submultiplicative, i.e.%
\begin{equation*}
M_{k}\left( \xi +\eta \right) \leq C_{k}M_{k}\left( \xi \right) M_{k}\left(
\eta \right) ,\quad \xi ,\eta \in \mathbb{R}^{n},
\end{equation*}%
where $C_{k}=2^{N/2}C^{-1}$ and that $k$ is moderate with respect to the
function $M_{k}$ or simply $M_{k}$-moderate, i.e.%
\begin{equation*}
k\left( \xi +\eta \right) \leq M_{k}\left( \xi \right) k\left( \eta \right)
,\quad \xi ,\eta \in \mathbb{R}^{n}.
\end{equation*}

$\left( \mathtt{b}\right) $ Let $k\in \mathcal{K}_{pol}\left( \mathbb{R}%
^{n}\right) $. From definition we deduce that 
\begin{equation*}
\frac{1}{M_{k}\left( \xi \right) }=C^{-1}\left\langle \xi \right\rangle
^{-N}\leq \frac{k\left( \xi +\eta \right) }{k\left( \eta \right) }\leq
C\left\langle \xi \right\rangle ^{N}=M_{k}\left( \xi \right) ,\quad \xi
,\eta \in \mathbb{R}^{n}.
\end{equation*}%
In fact, the left-hand inequality is obtained if $\xi $ is replaced by $-\xi 
$ and $\eta $ is replaced by $\xi +\eta $ in (\ref{kh2}). If we take $\eta
=0 $ we obtain the useful estimates 
\begin{equation*}
C^{-1}k\left( 0\right) \left\langle \xi \right\rangle ^{-N}\leq k\left( \xi
\right) \leq Ck\left( 0\right) \left\langle \xi \right\rangle ^{N},\quad \xi
\in \mathbb{R}^{n}.
\end{equation*}
\end{remark}

The following lemma is an easy consequence of the of the definition and the
above estimates.

\begin{lemma}
\label{kh5} Let $k$ ,$k_{1}$, $k_{2}\in \mathcal{K}_{pol}\left( \mathbb{R}%
^{n}\right) $. Then:

$\left( \mathtt{a}\right) $ $k_{1}+k_{2}$, $k_{1}\cdot k_{2}$, $\sup \left(
k_{1},k_{2}\right) $, $\inf \left( k_{1},k_{2}\right) \in \mathcal{K}%
_{pol}\left( \mathbb{R}^{n}\right) $.

$\left( \mathtt{b}\right) $ $k^{s}\in \mathcal{K}_{pol}\left( \mathbb{R}%
^{n}\right) $ for every real $s$.

$\left( \mathtt{c}\right) $ If $\check{k}\left( \xi \right) =k\left( -\xi
\right) $, $\xi \in \mathbb{R}^{n}$, then $\check{k}$ is $M_{k}$-moderate
hence $\check{k}\in \mathcal{K}_{pol}\left( \mathbb{R}^{n}\right) .$

$\left( \mathtt{d}\right) $ $0<\inf_{\xi \in K}k\left( \xi \right) \leq
\sup_{\xi \in K}k\left( \xi \right) <\infty $ for every compact subset $%
K\subset \mathbb{R}^{n}$.
\end{lemma}

\begin{definition}
If $k\in \mathcal{K}_{pol}\left( \mathbb{R}^{n}\right) $ and $1\leq p\leq
\infty $, we denote by $B_{p,k}\left( \mathbb{R}^{n}\right) $ the set of all
distributions $u\in \mathcal{S}^{\prime }$ such that $\widehat{u}$ is a
function and $k\widehat{u}\in L^{p}$. For $u\in B_{p,k}\left( \mathbb{R}%
^{n}\right) $ we define 
\begin{equation*}
\left\Vert u\right\Vert _{p,k}=\left\Vert k\widehat{u}\right\Vert
_{p}<\infty .
\end{equation*}
\end{definition}

In the following lemma we collect some properties of the the spaces $%
B_{p,k}\left( \mathbb{R}^{n}\right) $.

\begin{lemma}
\label{ks18}$\left( \mathtt{a}\right) $ $B_{p,k}\left( \mathbb{R}^{n}\right) 
$ is a Banach space with the norm $\left\Vert \cdot \right\Vert _{p,k}$. We
have 
\begin{equation*}
\mathcal{S}\left( \mathbb{R}^{n}\right) \subset B_{p,k}\left( \mathbb{R}%
^{n}\right) \subset \mathcal{S}^{\prime }\left( \mathbb{R}^{n}\right)
\end{equation*}%
continuously and densely.

$\left( \mathtt{b}\right) $ If $k_{1},k_{2}\in \mathcal{K}_{pol}\left( 
\mathbb{R}^{n}\right) $ and $k_{2}\left( \xi \right) \leq Ck_{1}\left( \xi
\right) ,\xi \in \mathbb{R}^{n}$, it follows that 
\begin{equation*}
B_{p,k_{1}}\left( \mathbb{R}^{n}\right) \subset B_{p,k_{2}}\left( \mathbb{R}%
^{n}\right) .
\end{equation*}

$\left( \mathtt{c}\right) $ The restriction of the isomorphism $\mathcal{S}%
^{\prime }\left( \mathbb{R}^{n}\right) \ni u\rightarrow \check{u}\in 
\mathcal{S}^{\prime }\left( \mathbb{R}^{n}\right) $ to the space $%
B_{p,k}\left( \mathbb{R}^{n}\right) $ induce an isometric isomorphism $%
B_{p,k}\left( \mathbb{R}^{n}\right) \ni u\rightarrow \check{u}\in B_{p,%
\check{k}}\left( \mathbb{R}^{n}\right) $. Here $\check{u}$ is of course the
composition of $u$ and $x\rightarrow -x$.

$\left( \mathtt{d}\right) $ If $L$ is a continuous linear form on $%
B_{p,k}\left( \mathbb{R}^{n}\right) $, $p<\infty $, we have for some $v\in
B_{p^{\prime },1/\check{k}}\left( \mathbb{R}^{n}\right) $, $1/p\prime +1/p=1$%
, $L\left( u\right) =\left( 2\pi \right) ^{n}\left\langle v,u\right\rangle $%
, $u\in \mathcal{S}\left( \mathbb{R}^{n}\right) $. The norm of this linear
form is $\left\Vert v\right\Vert _{p^{\prime },1/\check{k}}$. Hence $%
B_{p^{\prime },1/\check{k}}\left( \mathbb{R}^{n}\right) $ is the dual space
of $B_{p,k}\left( \mathbb{R}^{n}\right) $ and the canonical bilinear form in 
$B_{p,k}\left( \mathbb{R}^{n}\right) \times B_{p^{\prime },1/\check{k}%
}\left( \mathbb{R}^{n}\right) $ is the continuous extension of the bilinear
form $\left( 2\pi \right) ^{n}\left\langle v,u\right\rangle $, $v\in
B_{p^{\prime },1/\check{k}}\left( \mathbb{R}^{n}\right) $, $u\in \mathcal{S}%
\left( \mathbb{R}^{n}\right) $.

$\left( \mathtt{e}\right) $ If $u\in B_{p,k}\left( \mathbb{R}^{n}\right) $
and $\phi \in \mathcal{S}\left( \mathbb{R}^{n}\right) $, it follows that $%
\phi u\in B_{p,k}\left( \mathbb{R}^{n}\right) $ and that 
\begin{equation*}
\left\Vert \phi u\right\Vert _{p,k}\leq \left( 2\pi \right) ^{-n}\left\Vert
M_{k}\widehat{\phi }\right\Vert _{1}\left\Vert u\right\Vert _{p,k}.
\end{equation*}

$\left( \mathtt{f}\right) $ If $1/k\in L^{p^{\prime }}\left( \mathbb{R}%
^{n}\right) $, $1/p+1/p^{\prime }=1$, then $B_{p,k}\left( \mathbb{R}%
^{n}\right) \subset \mathcal{F}^{-1}L^{1}\left( \mathbb{R}^{n}\right)
\subset \mathcal{C}_{\infty }\left( \mathbb{R}^{n}\right) $.
\end{lemma}

\begin{proof}
$\left( \mathtt{a}\right) $ The proof is identical to the proof of the
Theorem 10.1.7 in H\"{o}rmander \cite{Hormander} vol. 2.

$\left( \mathtt{b}\right) $ and $\left( \mathtt{c}\right) $ are are easy
consequences of the definition.

$\left( \mathtt{d}\right) $ If $\mathcal{F}$ is the Fourier transformation
in $\mathcal{S}^{\prime }\left( \mathbb{R}^{n}\right) $ then%
\begin{eqnarray*}
L\left( u\right) &=&\left( 2\pi \right) ^{n}\left\langle v,u\right\rangle
=\left( 2\pi \right) ^{n}\left\langle \mathcal{FF}^{-1}v,u\right\rangle
=\left( 2\pi \right) ^{n}\left\langle \mathcal{F}^{-1}v,\mathcal{F}%
u\right\rangle \\
&=&\left\langle \mathcal{F}\check{v},\mathcal{F}u\right\rangle .
\end{eqnarray*}%
Hence the Fourier transformation reduces the theorem to the fact that a
continuous linear form on $\mathcal{S}\left( \mathbb{R}^{n}\right) $ with
respect to the norm $\left\Vert kU\right\Vert _{p}$, $p<\infty $, is a
scalar product with a function $V$ such that $V/k\in L^{p^{\prime }}$ and
that the norm of the linear form is $\left\Vert V/k\right\Vert _{p^{\prime
}} $.

$\left( \mathtt{e}\right) $ The proof is identical to the proof of the
Theorem 10.1.15 in H\"{o}rmander \cite{Hormander} vol. 2.

$\left( \mathtt{f}\right) $ Let $u\in B_{p,k}\left( \mathbb{R}^{n}\right) $.
If $1/k\in L^{p^{\prime }}\left( \mathbb{R}^{n}\right) $, then $\widehat{u}%
\in L^{1}\left( \mathbb{R}^{n}\right) $ since $k\widehat{u}\in L^{p}\left( 
\mathbb{R}^{n}\right) $ and $1/p+1/p^{\prime }=1$. Now the Riemann-Lebesgue
lemma implies the result. For $x\in \mathbb{R}^{n}$ we have%
\begin{equation*}
\left\vert u\left( x\right) \right\vert \leq \left( 2\pi \right)
^{-n}\left\Vert \widehat{u}\right\Vert _{L^{1}}\leq \left( 2\pi \right)
^{-n}\left\Vert k^{-1}\right\Vert _{L^{p^{\prime }}}\left\Vert k\widehat{u}%
\right\Vert _{L^{p}}=\left( 2\pi \right) ^{-n}\left\Vert k^{-1}\right\Vert
_{L^{p^{\prime }}}\left\Vert u\right\Vert _{p,k}.
\end{equation*}
\end{proof}

\begin{lemma}
\label{ks3}Let $k\in \mathcal{K}_{pol}\left( \mathbb{R}^{n}\right) $ and $C$%
, $N$ the positive constants that define $k$ and $1\leq p\leq \infty $. Let $%
m_{k}=\left[ N+\frac{n+1}{2}\right] +1$ and $l_{k}=\left[ N\right] +n+2$.

$\left( \mathtt{a}\right) $ If $\chi \in H^{N+\frac{n+1}{2}}\left( \mathbb{R}%
^{n}\right) $, then for every $u\in B_{p,k}\left( \mathbb{R}^{n}\right) $ we
have $\chi u\in B_{p,k}\left( \mathbb{R}^{n}\right) $ and%
\begin{eqnarray*}
\left\Vert \chi u\right\Vert _{p,k} &\leq &\left( 2\pi \right)
^{-n}\left\Vert M_{k}\widehat{\chi }\right\Vert _{1}\left\Vert u\right\Vert
_{p,k}=C\left( k,n,\chi \right) \left\Vert u\right\Vert _{p,k} \\
&\leq &C\left( C,n\right) \left\Vert \chi \right\Vert _{H^{N+\frac{n+1}{2}%
}}\left\Vert u\right\Vert _{p,k},
\end{eqnarray*}%
where 
\begin{eqnarray*}
C\left( k,n,\chi \right) &=&\left( 2\pi \right) ^{-n}\left\Vert M_{k}%
\widehat{\chi }\right\Vert _{1} \\
&\leq &C\left( C,n\right) \left\Vert \chi \right\Vert _{H^{N+\frac{n+1}{2}}},
\end{eqnarray*}%
Here $H^{m}\left( \mathbb{R}^{n}\right) $ is the usual Sobolev space, $m\in 
\mathbb{R}$. If $\chi \in H^{m_{k}}\left( \mathbb{R}^{n}\right) $, then 
\begin{equation*}
\left\Vert \chi u\right\Vert _{p,k}\leq C\left( C,n\right) \left(
\sum_{\left\vert \alpha \right\vert \leq m_{k}}\left\Vert \partial ^{\alpha
}\chi \right\Vert _{L^{2}}\right) \left\Vert u\right\Vert _{p,k}.
\end{equation*}

$\left( \mathtt{b}\right) $ If $\chi \in \mathcal{C}^{l_{k}}\left( \mathbb{R}%
^{n}\right) $ is $\mathbb{Z}^{n}$-periodic, then for every $u\in
B_{p,k}\left( \mathbb{R}^{n}\right) $ we have $\chi u\in B_{p,k}\left( 
\mathbb{R}^{n}\right) $ and%
\begin{equation*}
\left\Vert \chi u\right\Vert _{p,k}\leq Cst\left( C,N,n\right) \left\Vert
\chi \right\Vert _{\mathcal{BC}^{l_{k}}\left( \mathbb{R}^{n}\right)
}\left\Vert u\right\Vert _{p,k}.
\end{equation*}
\end{lemma}

\begin{proof}
$\left( \mathtt{a}\right) $ Since $\mathcal{S}\left( \mathbb{R}^{n}\right) $
is dense in $H^{N+\frac{n+1}{2}}\left( \mathbb{R}^{n}\right) $, we can
assume that $\chi \in \mathcal{S}\left( \mathbb{R}^{n}\right) $. We know
that 
\begin{equation*}
\left\Vert \chi u\right\Vert _{p,k}\leq \left( 2\pi \right) ^{-n}\left\Vert
M_{k}\widehat{\chi }\right\Vert _{1}\left\Vert u\right\Vert _{p,k}.
\end{equation*}%
Since $M_{k}\left( \xi \right) =C\left\langle \xi \right\rangle ^{N}$,
Schwarz inequality gives the estimate of $C\left( k,n,\chi \right) $%
\begin{eqnarray*}
C\left( k,n,\chi \right) &=&\left( 2\pi \right) ^{-n}\left\Vert M_{k}%
\widehat{\chi }\right\Vert _{1}=\left( 2\pi \right) ^{-n}C\left( \int
\left\langle \eta \right\rangle ^{N}\left\vert \widehat{\chi }\left( \eta
\right) \right\vert \mathtt{d}\eta \right) \\
&\leq &\left( 2\pi \right) ^{-n}C\left\Vert \left\langle \cdot \right\rangle
^{-n-1}\right\Vert _{L^{1}}\left\Vert \chi \right\Vert _{H^{N+\frac{n+1}{2}}}
\\
&=&C\left( C,n\right) \left\Vert \chi \right\Vert _{H^{N+\frac{n+1}{2}}},
\end{eqnarray*}%
with $\left\Vert \chi \right\Vert _{H^{N+\frac{n+1}{2}}}\leq Cst\left(
\sum_{\left\vert \alpha \right\vert \leq m_{k}}\left\Vert \partial ^{\alpha
}\chi \right\Vert _{L^{2}}\right) $ when $\chi \in H^{m_{k}}\left( \mathbb{R}%
^{n}\right) $.

$\left( \mathtt{b}\right) $ We shall use some results from \cite{Hormander}
vol. 1, pp 177-179, concerning periodic distributions. If $\chi \in \mathcal{%
C}^{l_{k}}\left( \mathbb{R}^{n}\right) $ is $\mathbb{Z}^{n}$-periodic, then 
\begin{equation*}
\chi =\sum_{\gamma \in \mathbb{Z}^{n}}\mathtt{e}^{2\pi \mathtt{i}%
\left\langle \cdot ,\gamma \right\rangle }c_{\gamma },
\end{equation*}%
with Fourier coefficients 
\begin{equation*}
c_{\gamma }=\int_{\mathtt{I}}\chi \left( x\right) \mathtt{e}^{-2\pi \mathtt{i%
}\left\langle x,\gamma \right\rangle }\mathtt{d}x,\quad \mathtt{I}=\left[
0,1\right) ^{n},\quad \gamma \in \mathbb{Z}^{n},
\end{equation*}%
satisfying%
\begin{equation*}
\left\vert c_{\gamma }\right\vert \leq Cst\left\Vert \chi \right\Vert _{%
\mathcal{BC}^{l_{k}}\left( \mathbb{R}^{n}\right) }\left\langle 2\pi \gamma
\right\rangle ^{-l_{k}},\quad \gamma \in \mathbb{Z}^{n}.
\end{equation*}%
Since $\widehat{\mathtt{e}^{\mathtt{i}\left\langle \cdot ,\eta \right\rangle
}u}=\widehat{u}\left( \cdot -\eta \right) $, multiplying by $k\left( \xi
\right) $ and noting the inequality $k\left( \xi \right) \leq C\left\langle
\eta \right\rangle ^{N}k\left( \xi -\eta \right) $, we obtain 
\begin{equation*}
\left\vert k\left( \xi \right) \widehat{\mathtt{e}^{\mathtt{i}\left\langle
\cdot ,\eta \right\rangle }u}\left( \xi \right) \right\vert \leq
C\left\langle \eta \right\rangle ^{N}\left\vert k\left( \xi -\eta \right) 
\widehat{u}\left( \xi -\eta \right) \right\vert ,
\end{equation*}%
and%
\begin{equation*}
\left\Vert \mathtt{e}^{\mathtt{i}\left\langle \cdot ,\eta \right\rangle
}u\right\Vert _{p,k}\leq C\left\langle \eta \right\rangle ^{N}\left\Vert
u\right\Vert _{p,k}.
\end{equation*}%
It follows that%
\begin{eqnarray*}
\left\Vert \chi u\right\Vert _{p,k} &\leq &Cst\cdot C\left\Vert \chi
\right\Vert _{\mathcal{BC}^{l_{k}}\left( \mathbb{R}^{n}\right) }\left(
\sum_{\gamma \in \mathbb{Z}^{n}}\left\langle 2\pi \gamma \right\rangle
^{-l_{k}}\left\langle 2\pi \gamma \right\rangle ^{N}\right) \left\Vert
u\right\Vert _{p,k} \\
&\leq &Cst\cdot C\left( \sum_{\gamma \in \mathbb{Z}^{n}}\left\langle 2\pi
\gamma \right\rangle ^{-n-1}\right) \left\Vert \chi \right\Vert _{\mathcal{BC%
}^{l_{k}}\left( \mathbb{R}^{n}\right) }\left\Vert u\right\Vert _{p,k}.
\end{eqnarray*}%
since $l_{k}=\left[ N\right] +n+2$.
\end{proof}

For any $x\in \mathbb{R}^{n}$ and for any distribution $u$ on $\mathbb{R}%
^{n} $, by $\tau _{x}u$ we shall denote the translation by $x$ of $u$, i.e. $%
\tau _{x}u=u\left( \cdot -x\right) =\delta _{x}\ast u$.

\begin{lemma}
Let $\varphi \in \mathcal{S}\left( \mathbb{R}^{n}\right) $ and $\theta \in %
\left[ 0,2\pi \right] ^{n}$. If 
\begin{equation*}
\varphi _{\theta }=\sum_{\gamma \in \mathbb{Z}^{n}}\mathtt{e}^{\mathtt{i}%
\left\langle \gamma ,\theta \right\rangle }\tau _{\gamma }\varphi
=\sum_{\gamma \in \mathbb{Z}^{n}}\mathtt{e}^{\mathtt{i}\left\langle \gamma
,\theta \right\rangle }\varphi \left( \cdot -\gamma \right) =\sum_{\gamma
\in \mathbb{Z}^{n}}\mathtt{e}^{\mathtt{i}\left\langle \gamma ,\theta
\right\rangle }\delta _{\gamma }\ast \varphi ,
\end{equation*}%
then 
\begin{equation*}
\widehat{\varphi }_{\theta }=\nu _{\theta }=\left( 2\pi \right)
^{n}\sum_{\gamma \in \mathbb{Z}^{n}}\widehat{\varphi }\left( 2\pi \gamma
+\theta \right) \delta _{2\pi \gamma +\theta }.
\end{equation*}
\end{lemma}

\begin{proof}
We have $\varphi _{\theta }=\varphi \ast \left( \mathtt{e}^{\mathtt{i}%
\left\langle \cdot ,\theta \right\rangle }S\right) $, where $S=\sum_{\gamma
\in \mathbb{Z}^{n}}\delta _{\gamma }$. We apply Poisson's summation formula, 
$\mathcal{F}\left( \sum_{\gamma \in \mathbb{Z}^{n}}\delta _{\gamma }\right)
=\left( 2\pi \right) ^{n}\sum_{\gamma \in \mathbb{Z}^{n}}\delta _{2\pi
\gamma }$, to obtain 
\begin{eqnarray*}
\widehat{\varphi }_{\theta } &=&\widehat{\varphi }\cdot \widehat{\left( 
\mathtt{e}^{\mathtt{i}\left\langle \cdot ,\theta \right\rangle }S\right) }=%
\widehat{\varphi }\cdot \tau _{\theta }\widehat{S}=\left( 2\pi \right) ^{n}%
\widehat{\varphi }\sum_{\gamma \in \mathbb{Z}^{n}}\delta _{2\pi \gamma
+\theta } \\
&=&\left( 2\pi \right) ^{n}\sum_{\gamma \in \mathbb{Z}^{n}}\widehat{\varphi }%
\left( 2\pi \gamma +\theta \right) \delta _{2\pi \gamma +\theta }.
\end{eqnarray*}
\end{proof}

\begin{notation}
For $k$ in $\mathcal{K}_{pol}\left( \mathbb{R}^{n}\right) $ we denote by $%
\mathcal{B}_{k}\left( \mathbb{R}^{n}\right) $ the Hilbert space $%
B_{2,k}\left( \mathbb{R}^{n}\right) $. We shall use $\left\Vert \cdot
\right\Vert _{\mathcal{B}_{k}}$ for the norm $\left\Vert \cdot \right\Vert
_{2,k}$.
\end{notation}

As we already said the techniques of Coifman and Meyer, used in the study of
Beurling algebras $A_{\omega }$ and $B_{\omega }$ (see \cite{Meyer} pp
7-10), can be adapted to the case spaces $\mathcal{B}_{k}\left( \mathbb{R}%
^{n}\right) =B_{2,k}\left( \mathbb{R}^{n}\right) $. An example is the
following result.

\begin{lemma}
\label{ks2}Let $k\in \mathcal{K}_{pol}\left( \mathbb{R}^{n}\right) $. Let $%
\left\{ u_{\gamma }\right\} _{\gamma \in \mathbb{Z}^{n}}$ be a a family of
elements from $\mathcal{B}_{k}\left( \mathbb{R}^{n}\right) \cap \mathcal{D}%
_{K}^{\prime }\left( \mathbb{R}^{n}\right) $, where $K\subset \mathbb{R}^{n}$
is a compact subset such that $\left( K-K\right) \cap \mathbb{Z}^{n}=\left\{
0\right\} $. Put%
\begin{equation*}
u=\sum_{\gamma \in \mathbb{Z}^{n}}\tau _{\gamma }u_{\gamma }=\sum_{\gamma
\in \mathbb{Z}^{n}}u_{\gamma }\left( \cdot -\gamma \right) =\sum_{\gamma \in 
\mathbb{Z}^{n}}\delta _{\gamma }\ast u_{\gamma }\in \mathcal{D}^{\prime
}\left( \mathbb{R}^{n}\right) .
\end{equation*}%
Then the following statements are equivalent:

$\left( \mathtt{a}\right) $ $u\in \mathcal{B}_{k}\left( \mathbb{R}%
^{n}\right) $.

$\left( \mathtt{b}\right) $ $\sum_{\gamma \in \mathbb{Z}^{n}}\left\Vert
u_{\gamma }\right\Vert _{\mathcal{B}_{k}}^{2}<\infty .$

Moreover, there is $C\geq 1$, which does not depend on the family $\left\{
u_{\gamma }\right\} _{\gamma \in \mathbb{Z}^{n}}$, such that%
\begin{equation}
C^{-1}\left\Vert u\right\Vert _{\mathcal{B}_{k}}\leq \left( \sum_{\gamma \in 
\mathbb{Z}^{n}}\left\Vert u_{\gamma }\right\Vert _{\mathcal{B}%
_{k}}^{2}\right) ^{1/2}\leq C\left\Vert u\right\Vert _{\mathcal{B}_{k}}.
\label{ks1}
\end{equation}
\end{lemma}

\begin{proof}
Let us choose $\varphi \in \mathcal{C}_{0}^{\infty }\left( \mathbb{R}%
^{n}\right) $ such that $\varphi =1$ on $K$ and \texttt{supp}$\varphi
=K^{\prime }$ satisfies the condition $\left( K^{\prime }-K^{\prime }\right)
\cap \mathbb{Z}^{n}=\left\{ 0\right\} $. For $\theta \in \left[ 0,2\pi %
\right] ^{n}$ we set 
\begin{eqnarray*}
\varphi _{\theta } &=&\sum_{\gamma \in \mathbb{Z}^{n}}\mathtt{e}^{\mathtt{i}%
\left\langle \gamma ,\theta \right\rangle }\tau _{\gamma }\varphi
=\sum_{\gamma \in \mathbb{Z}^{n}}\mathtt{e}^{\mathtt{i}\left\langle \gamma
,\theta \right\rangle }\delta _{\gamma }\ast \varphi , \\
u_{\theta } &=&\sum_{\gamma \in \mathbb{Z}^{n}}\mathtt{e}^{\mathtt{i}%
\left\langle \gamma ,\theta \right\rangle }\tau _{\gamma }u_{\gamma
}=\sum_{\gamma \in \mathbb{Z}^{n}}\mathtt{e}^{\mathtt{i}\left\langle \gamma
,\theta \right\rangle }\delta _{\gamma }\ast u_{\gamma }.
\end{eqnarray*}%
Since $\left( K^{\prime }-K^{\prime }\right) \cap \mathbb{Z}^{n}=\left\{
0\right\} $ we have 
\begin{equation*}
u_{\theta }=\varphi _{\theta }u,\quad \quad u=\varphi _{\theta }u_{-\theta }.
\end{equation*}

\textit{Step 1.} Suppose first that the family $\left\{ u_{\gamma }\right\}
_{\gamma \in \mathbb{Z}^{n}}$ has only a finite number of non-zero terms and
we shall prove in this case the estimate(\ref{ks1}). Since $u_{\theta }$, $%
u\in \mathcal{E}^{\prime }\left( \mathbb{R}^{n}\right) \subset \mathcal{S}%
^{\prime }\left( \mathbb{R}^{n}\right) $ it follows that%
\begin{equation*}
\widehat{u}_{\theta }=\left( 2\pi \right) ^{-n}\nu _{\theta }\ast \widehat{u}%
,\quad \quad \widehat{u}=\left( 2\pi \right) ^{-n}\nu _{\theta }\ast 
\widehat{u}_{-\theta },
\end{equation*}%
where $\nu _{\theta }=\widehat{\varphi }_{\theta }=\left( 2\pi \right)
^{n}\sum_{\gamma \in \mathbb{Z}^{n}}\widehat{\varphi }\left( 2\pi \gamma
+\theta \right) \delta _{2\pi \gamma +\theta }$ is a measure of rapid decay
at $\infty $. Since $\widehat{u}_{\theta }$, $\widehat{u}\in \mathcal{C}%
_{pol}^{\infty }\left( \mathbb{R}^{n}\right) $ we get the pointwise
equalities 
\begin{eqnarray*}
\widehat{u}_{\theta }\left( \xi \right) &=&\sum_{\gamma \in \mathbb{Z}^{n}}%
\widehat{\varphi }\left( 2\pi \gamma +\theta \right) \widehat{u}\left( \xi
-2\pi \gamma -\theta \right) , \\
\widehat{u}\left( \xi \right) &=&\sum_{\gamma \in \mathbb{Z}^{n}}\widehat{%
\varphi }\left( 2\pi \gamma +\theta \right) \widehat{u}_{-\theta }\left( \xi
-2\pi \gamma -\theta \right) .
\end{eqnarray*}%
Multiplying by $k\left( \xi \right) $ and using the inequality $k\left( \xi
\right) \leq C\left\langle 2\pi \gamma +\theta \right\rangle ^{N}k\left( \xi
-2\pi \gamma -\theta \right) $ we obtain%
\begin{equation*}
k\left( \xi \right) \left\vert \widehat{u}_{\theta }\left( \xi \right)
\right\vert \leq C\sum_{\gamma \in \mathbb{Z}^{n}}\left\langle 2\pi \gamma
+\theta \right\rangle ^{N}\left\vert \widehat{\varphi }\left( 2\pi \gamma
+\theta \right) \right\vert k\left( \xi -2\pi \gamma -\theta \right)
\left\vert \widehat{u}\left( \xi -2\pi \gamma -\theta \right) \right\vert ,
\end{equation*}%
and%
\begin{equation*}
k\left( \xi \right) \widehat{u}\left( \xi \right) \leq C\sum_{\gamma \in 
\mathbb{Z}^{n}}\left\langle 2\pi \gamma +\theta \right\rangle ^{N}\left\vert 
\widehat{\varphi }\left( 2\pi \gamma +\theta \right) \right\vert k\left( \xi
-2\pi \gamma -\theta \right) \left\vert \widehat{u}_{-\theta }\left( \xi
-2\pi \gamma -\theta \right) \right\vert .
\end{equation*}%
It follows that 
\begin{eqnarray*}
\left\Vert u_{\theta }\right\Vert _{\mathcal{B}_{k}} &=&\left\Vert k\widehat{%
u}_{\theta }\right\Vert _{L^{2}}\leq C\left( \sum_{\gamma \in \mathbb{Z}%
^{n}}\left\langle 2\pi \gamma +\theta \right\rangle ^{N}\left\vert \widehat{%
\varphi }\left( 2\pi \gamma +\theta \right) \right\vert \right) \left\Vert k%
\widehat{u}\right\Vert _{L^{2}} \\
&=&C\left( \sum_{\gamma \in \mathbb{Z}^{n}}\left\langle 2\pi \gamma +\theta
\right\rangle ^{N}\left\vert \widehat{\varphi }\left( 2\pi \gamma +\theta
\right) \right\vert \right) \left\Vert u\right\Vert _{\mathcal{B}%
_{k}}=C_{k,\varphi }\left\Vert u\right\Vert _{\mathcal{B}_{k}}
\end{eqnarray*}%
and%
\begin{equation*}
\left\Vert u\right\Vert _{\mathcal{B}_{k}}\leq C\left( \sum_{\gamma \in 
\mathbb{Z}^{n}}\left\langle 2\pi \gamma +\theta \right\rangle ^{N}\left\vert 
\widehat{\varphi }\left( 2\pi \gamma +\theta \right) \right\vert \right)
\left\Vert u_{-\theta }\right\Vert _{\mathcal{B}_{k}}=C_{k,\varphi
}\left\Vert u_{-\theta }\right\Vert _{\mathcal{B}_{k}},
\end{equation*}%
where $C_{k,\varphi }=C\sum_{\gamma \in \mathbb{Z}^{n}}\left\langle 2\pi
\gamma +\theta \right\rangle ^{N}\left\vert \widehat{\varphi }\left( 2\pi
\gamma +\theta \right) \right\vert <\infty $.

The above estimates can be rewritten as 
\begin{eqnarray*}
\int \left\vert k\left( \xi \right) \widehat{u}_{\theta }\left( \xi \right)
\right\vert ^{2}\mathtt{d}\xi &\leq &C_{k,\varphi }^{2}\left\Vert
u\right\Vert _{\mathcal{B}_{k}}^{2}, \\
\left\Vert u\right\Vert _{\mathcal{B}_{k}}^{2} &\leq &C_{k,\varphi }^{2}\int
\left\vert k\left( \xi \right) \widehat{u}_{-\theta }\left( \xi \right)
\right\vert ^{2}\mathtt{d}\xi .
\end{eqnarray*}%
On the other hand, the equality $u_{\theta }=\sum_{\gamma \in \mathbb{Z}^{n}}%
\mathtt{e}^{\mathtt{i}\left\langle \gamma ,\theta \right\rangle }\tau
_{\gamma }u_{\gamma }$ implies 
\begin{equation*}
\widehat{u}_{\theta }\left( \xi \right) =\sum_{\gamma \in \mathbb{Z}^{n}}%
\mathtt{e}^{\mathtt{i}\left\langle \gamma ,\theta -\xi \right\rangle }%
\widehat{u}_{\gamma }\left( \xi \right)
\end{equation*}%
with finite sum. The functions $\theta \rightarrow \widehat{u}_{\pm \theta
}\left( \xi \right) $ are in $L^{2}\left( \left[ 0,2\pi \right] ^{n}\right) $
and 
\begin{equation*}
\left( 2\pi \right) ^{-n}\int_{\left[ 0,2\pi \right] ^{n}}\left\vert 
\widehat{u}_{\pm \theta }\left( \xi \right) \right\vert ^{2}\mathtt{d}\theta
=\sum_{\gamma \in \mathbb{Z}^{n}}\left\vert \widehat{u}_{\gamma }\left( \xi
\right) \right\vert ^{2}.
\end{equation*}%
Integrating with respect $\theta $ the above inequalities we get that%
\begin{eqnarray*}
\sum_{\gamma \in \mathbb{Z}^{n}}\left\Vert u_{\gamma }\right\Vert _{\mathcal{%
B}_{k}}^{2} &\leq &C_{k,\varphi }^{2}\left\Vert u\right\Vert _{\mathcal{B}%
_{k}}^{2}, \\
\left\Vert u\right\Vert _{\mathcal{B}_{k}}^{2} &\leq &C_{k,\varphi
}^{2}\sum_{\gamma \in \mathbb{Z}^{n}}\left\Vert u_{\gamma }\right\Vert _{%
\mathcal{B}_{k}}^{2}.
\end{eqnarray*}

\textit{Step 2. }The \textit{g}eneral case is obtained by approximation.

Suppose that $u\in \mathcal{B}_{k}\left( \mathbb{R}^{n}\right) $. Let $\psi
\in \mathcal{C}_{0}^{\infty }\left( \mathbb{R}^{n}\right) $ be such that $%
\psi =1$ on $B\left( 0,1\right) $. Then $\psi ^{\varepsilon }u\rightarrow u$
in $\mathcal{B}_{k}\left( \mathbb{R}^{n}\right) $ where $\psi ^{\varepsilon
}\left( x\right) =\psi \left( \varepsilon x\right) $, $0<\varepsilon \leq 1$%
, $x\in \mathbb{R}^{n}$. Also we have 
\begin{equation*}
\left\Vert \psi ^{\varepsilon }u\right\Vert _{\mathcal{B}_{k}}\leq C\left(
k,\psi \right) \left\Vert u\right\Vert _{\mathcal{B}_{k}},\quad
0<\varepsilon \leq 1,
\end{equation*}%
where 
\begin{eqnarray*}
C\left( k,\psi \right) &=&\left( 2\pi \right) ^{-n}C\sup_{0<\varepsilon \leq
1}\left( \int \left\langle \eta \right\rangle ^{N}\varepsilon
^{-n}\left\vert \widehat{\psi }\left( \eta /\varepsilon \right) \right\vert 
\mathtt{d}\eta \right) \\
&=&\left( 2\pi \right) ^{-n}C\sup_{0<\varepsilon \leq 1}\left( \int
\left\langle \varepsilon \eta \right\rangle ^{N}\left\vert \widehat{\psi }%
\left( \eta \right) \right\vert \mathtt{d}\eta \right) \\
&=&\left( 2\pi \right) ^{-n}C\left( \int \left\langle \eta \right\rangle
^{N}\left\vert \widehat{\psi }\left( \eta \right) \right\vert \mathtt{d}\eta
\right) .
\end{eqnarray*}%
Let $m\in \mathbb{N},m\geq 1$. Then there is $\varepsilon _{m}$ such that
for any $\varepsilon \in \left( 0,\varepsilon _{m}\right] $ we have 
\begin{equation*}
\psi ^{\varepsilon }u=\sum_{\left\vert \gamma \right\vert \leq m}\tau
_{\gamma }u_{\gamma }+\sum_{finite}\tau _{\gamma }\left( \left( \tau
_{-\gamma }\psi ^{\varepsilon }\right) u_{\gamma }\right) .
\end{equation*}%
By the first part we get that 
\begin{equation*}
\sum_{\left\vert \gamma \right\vert \leq m}\left\Vert u_{\gamma }\right\Vert
_{\mathcal{B}_{k}}^{2}\leq C_{k,\varphi }^{2}\left\Vert \psi ^{\varepsilon
}u\right\Vert _{\mathcal{B}_{k}}^{2}\leq C_{k,\varphi }^{2}C\left( k,\psi
\right) ^{2}\left\Vert u\right\Vert _{\mathcal{B}_{k}}^{2}.
\end{equation*}%
Since $m$ is arbitrary, it follows that $\sum_{\gamma \in \mathbb{Z}%
^{n}}\left\Vert u_{\gamma }\right\Vert _{\mathcal{B}_{k}}^{2}<\infty $.
Further from 
\begin{equation*}
\sum_{\left\vert \gamma \right\vert \leq m}\left\Vert u_{\gamma }\right\Vert
_{\mathcal{B}_{k}}^{2}\leq C_{k,\varphi }^{2}\left\Vert \psi ^{\varepsilon
}u\right\Vert _{\mathcal{B}_{k}}^{2},\quad 0<\varepsilon \leq \varepsilon
_{m},
\end{equation*}%
we obtain that 
\begin{equation*}
\sum_{\left\vert \gamma \right\vert \leq m}\left\Vert u_{\gamma }\right\Vert
_{\mathcal{B}_{k}}^{2}\leq C_{k,\varphi }^{2}\left\Vert u\right\Vert _{%
\mathcal{B}_{k}}^{2},\quad m\in \mathbb{N}.
\end{equation*}%
Hence 
\begin{equation*}
\sum_{\gamma \in \mathbb{Z}^{n}}\left\Vert u_{\gamma }\right\Vert _{\mathcal{%
B}_{k}}^{2}\leq C_{k,\varphi }^{2}\left\Vert u\right\Vert _{\mathcal{B}%
_{k}}^{2}.
\end{equation*}

Now suppose that $\sum_{\gamma \in \mathbb{Z}^{n}}\left\Vert u_{\gamma
}\right\Vert _{\mathcal{B}_{k}}^{2}<\infty $. For $m\in \mathbb{N}$, $m\geq
1 $ we put $u\left( m\right) =\sum_{\left\vert \gamma \right\vert \leq
m}\tau _{\gamma }u_{\gamma }$. Then 
\begin{equation*}
\left\Vert u\left( m+p\right) -u\left( m\right) \right\Vert _{\mathcal{B}%
_{k}}^{2}\leq C_{k,\varphi }^{2}\sum_{m\leq \left\vert \gamma \right\vert
\leq m+p}\left\Vert u_{\gamma }\right\Vert _{\mathcal{B}_{k}}^{2}
\end{equation*}%
It follows that $\left\{ u\left( m\right) \right\} _{m\geq 1}$ is a Cauchy
sequence in $\mathcal{B}_{k}\left( \mathbb{R}^{n}\right) $. Let $v\in 
\mathcal{B}_{k}\left( \mathbb{R}^{n}\right) $ be such that $u\left( m\right)
\rightarrow v$ in $\mathcal{B}_{k}\left( \mathbb{R}^{n}\right) $. Since $%
u\left( m\right) \rightarrow u$ in $\mathcal{D}^{\prime }\left( \mathbb{R}%
^{n}\right) $, it follows that $u=v$. Hence $u\left( m\right) \rightarrow u$
in $\mathcal{B}_{k}\left( \mathbb{R}^{n}\right) $. Since we have 
\begin{equation*}
\left\Vert u\left( m\right) \right\Vert _{\mathcal{B}_{k}}^{2}\leq
C_{k,\varphi }^{2}\sum_{\left\vert \gamma \right\vert \leq m}\left\Vert
u_{\gamma }\right\Vert _{\mathcal{B}_{k}}^{2}\leq C_{k,\varphi
}^{2}\sum_{\gamma \in \mathbb{Z}^{n}}\left\Vert u_{\gamma }\right\Vert _{%
\mathcal{B}_{k}}^{2},\quad m\in \mathbb{N},
\end{equation*}%
we obtain that 
\begin{equation*}
\left\Vert u\right\Vert _{\mathcal{B}_{k}}^{2}\leq C_{k,\varphi
}^{2}\sum_{\gamma \in \mathbb{Z}^{n}}\left\Vert u_{\gamma }\right\Vert _{%
\mathcal{B}_{k}}^{2}.
\end{equation*}
\end{proof}

To use the previous result we need a convenient partition of unity. Let $%
m\in \mathbb{N}$ and $\left\{ x_{1},...,x_{m}\right\} \subset \mathbb{R}^{n}$
be such that 
\begin{equation*}
\left[ 0,1\right] ^{n}\subset \left( x_{1}+\left[ \frac{1}{3},\frac{2}{3}%
\right] ^{n}\right) \cup ...\cup \left( x_{m}+\left[ \frac{1}{3},\frac{2}{3}%
\right] ^{n}\right) .
\end{equation*}%
Let $\widetilde{h}\in \mathcal{C}_{0}^{\infty }\left( \mathbb{R}^{n}\right) $%
, $\widetilde{h}\geq 0$, be such that $\widetilde{h}=1$ on $\left[ \frac{1}{3%
},\frac{2}{3}\right] ^{n}$ and \texttt{supp}$\widetilde{h}\subset \left[ 
\frac{1}{4},\frac{3}{4}\right] ^{n}$. Then

\begin{enumerate}
\item[$\left( \mathtt{a}\right) $] $\widetilde{H}=\sum_{i=1}^{m}\sum_{\gamma
\in \mathbb{Z}^{n}}\tau _{\gamma +x_{i}}\widetilde{h}\in \mathcal{BC}%
^{\infty }\left( \mathbb{R}^{n}\right) $ is $\mathbb{Z}^{n}$-periodic and $%
\widetilde{H}\geq 1$.

\item[$\left( \mathtt{b}\right) $] $h_{i}=\frac{\tau _{x_{i}}\widetilde{h}}{%
\widetilde{H}}\in \mathcal{C}_{0}^{\infty }\left( \mathbb{R}^{n}\right) $, $%
h_{i}\geq 0$, \texttt{supp}$h_{i}\subset x_{i}+\left[ \frac{1}{4},\frac{3}{4}%
\right] ^{n}=K_{i}$, $\left( K_{i}-K_{i}\right) \cap \mathbb{Z}^{n}=\left\{
0\right\} $, $i=1,...,m$.

\item[$\left( \mathtt{c}\right) $] $\chi _{i}=\sum_{\gamma \in \mathbb{Z}%
^{n}}\tau _{\gamma }h_{i}\in \mathcal{BC}^{\infty }\left( \mathbb{R}%
^{n}\right) $ is $\mathbb{Z}^{n}$-periodic, $i=1,...,m$ and $%
\sum_{i=1}^{m}\chi _{i}=1.$

\item[$\left( \mathtt{d}\right) $] $h=\sum_{i=1}^{m}h_{i}\in \mathcal{C}%
_{0}^{\infty }\left( \mathbb{R}^{n}\right) $, $h\geq 0$, $\sum_{\gamma \in 
\mathbb{Z}^{n}}\tau _{\gamma }h=$ $1$.
\end{enumerate}

\noindent A first consequence of previous results is the next proposition.

\begin{proposition}
\label{ks4}Let $k\in \mathcal{K}_{pol}\left( \mathbb{R}^{n}\right) $ and $C$%
, $N$ the positive constants that define $k$. Let $m_{k}=\left[ N+\frac{n+1}{%
2}\right] +1$. Then $\mathcal{BC}^{m_{k}}\left( \mathbb{R}^{n}\right) \cdot 
\mathcal{B}_{k}\left( \mathbb{R}^{n}\right) \subset \mathcal{B}_{k}\left( 
\mathbb{R}^{n}\right) $. In particular $\mathcal{BC}^{\infty }\left( \mathbb{%
R}^{n}\right) \cdot \mathcal{B}_{k}\left( \mathbb{R}^{n}\right) \subset 
\mathcal{B}_{k}\left( \mathbb{R}^{n}\right) $.
\end{proposition}

\begin{proof}
Let $u\in \mathcal{B}_{k}\left( \mathbb{R}^{n}\right) $. We use the
partition of unity constructed above to obtain a decomposition of $u$
satisfying the conditions of Lemma \ref{ks2}. We have $u=\sum_{i=1}^{m}\chi
_{i}u$ with $\chi _{i}u\in \mathcal{B}_{k}\left( \mathbb{R}^{n}\right) $ by
Lemma \ref{ks3} $\left( \mathtt{b}\right) $ and 
\begin{gather*}
\chi _{i}u=\sum_{\gamma \in \mathbb{Z}^{n}}\tau _{\gamma }\left( h_{i}\tau
_{-\gamma }u\right) ,\quad h_{i}\tau _{-\gamma }u\in \mathcal{B}_{k}\left( 
\mathbb{R}^{n}\right) \cap \mathcal{D}_{K_{i}}^{\prime }\left( \mathbb{R}%
^{n}\right) , \\
\left( K_{i}-K_{i}\right) \cap \mathbb{Z}^{n}=\left\{ 0\right\} ,\quad
i=1,...,m.
\end{gather*}%
So we can assume that $u\in \mathcal{B}_{k}\left( \mathbb{R}^{n}\right) $ is
of the form described in Lemma \ref{ks2}.

Let $\psi \in \mathcal{BC}^{m_{k}}\left( \mathbb{R}^{n}\right) $. Then $\psi
u=\sum_{\gamma \in \mathbb{Z}^{n}}\psi \tau _{\gamma }u_{\gamma
}=\sum_{\gamma \in \mathbb{Z}^{n}}\tau _{\gamma }\left( \psi _{\gamma
}u_{\gamma }\right) $ with $\psi _{\gamma }=\varphi \left( \tau _{-\gamma
}\psi \right) $, where $\varphi \in \mathcal{C}_{0}^{\infty }\left( \mathbb{R%
}^{n}\right) $ is the function considered in the proof of Lemma \ref{ks2}.
We apply Lemma \ref{ks2} and Lemma \ref{ks3} $\left( \mathtt{a}\right) $ to
obtain 
\begin{equation*}
\left\Vert \psi u\right\Vert _{\mathcal{B}_{k}}^{2}\leq C_{k,\varphi
}^{2}\sum_{\gamma \in \mathbb{Z}^{n}}\left\Vert \psi _{\gamma }u_{\gamma
}\right\Vert _{\mathcal{B}_{k}}^{2},
\end{equation*}%
\begin{eqnarray*}
\left\Vert \psi _{\gamma }u_{\gamma }\right\Vert _{\mathcal{B}_{k}} &\leq
&Cst\left( \sum_{\left\vert \alpha \right\vert \leq m_{k}}\left\Vert
\partial ^{\alpha }\left( \varphi \left( \tau _{-\gamma }\psi \right)
\right) \right\Vert _{L^{2}}\right) \left\Vert u_{\gamma }\right\Vert _{%
\mathcal{B}_{k}} \\
&\leq &Cst\left\Vert \varphi \right\Vert _{H^{m_{k}}}\left\Vert \psi
\right\Vert _{\mathcal{BC}^{m_{k}}}\left\Vert u_{\gamma }\right\Vert _{%
\mathcal{B}_{k}},\quad \gamma \in \mathbb{Z}^{n}.
\end{eqnarray*}%
Hence another application of Lemma \ref{ks2} gives 
\begin{eqnarray*}
\left\Vert \psi u\right\Vert _{\mathcal{B}_{k}}^{2} &\leq &Cst\left\Vert
\varphi \right\Vert _{H^{m_{k}}}^{2}\left\Vert \psi \right\Vert _{\mathcal{BC%
}^{m_{k}}}^{2}\sum_{\gamma \in \mathbb{Z}^{n}}\left\Vert u_{\gamma
}\right\Vert _{\mathcal{B}_{k}}^{2} \\
&\leq &Cst\left\Vert \varphi \right\Vert _{H^{m_{k}}}^{2}\left\Vert \psi
\right\Vert _{\mathcal{BC}^{m_{k}}}^{2}\left\Vert u\right\Vert _{\mathcal{B}%
_{k}}^{2}.
\end{eqnarray*}
\end{proof}

\section{The spaces $\mathcal{B}_{k}^{p}$}

We begin by proving some results that will be useful later. Let $\varphi
,\psi \in \mathcal{C}_{0}^{\infty }\left( \mathbb{R}^{n}\right) $ (or $%
\varphi ,\psi \in \mathcal{S}\left( \mathbb{R}^{n}\right) $). Then the maps 
\begin{eqnarray*}
\mathbb{R}^{n}\times \mathbb{R}^{n} &\ni &\left( x,y\right) \overset{f}{%
\longrightarrow }\varphi \left( x\right) \psi \left( x-y\right) =\left(
\varphi \tau _{y}\psi \right) \left( x\right) \in \mathbb{C}, \\
\mathbb{R}^{n}\times \mathbb{R}^{n} &\ni &\left( x,y\right) \overset{g}{%
\longrightarrow }\varphi \left( y\right) \psi \left( x-y\right) =\varphi
\left( y\right) \left( \tau _{y}\psi \right) \left( x\right) \in \mathbb{C},
\end{eqnarray*}%
are in $\mathcal{C}_{0}^{\infty }\left( \mathbb{R}^{n}\times \mathbb{R}%
^{n}\right) $ (respectively in $\mathcal{S}\left( \mathbb{R}^{n}\times 
\mathbb{R}^{n}\right) $). Let $u\in \mathcal{D}^{\prime }\left( \mathbb{R}%
^{n}\right) $ (or $u\in \mathcal{S}^{\prime }\left( \mathbb{R}^{n}\right) $%
). Then using Fubini theorem for distributions we get 
\begin{eqnarray*}
\left\langle u\otimes 1,f\right\rangle &=&\left\langle u\left( x\right)
,\left\langle 1\left( y\right) ,\varphi \left( x\right) \psi \left(
x-y\right) \right\rangle \right\rangle \\
&=&\left( \int \psi \right) \left\langle u,\varphi \right\rangle , \\
\left\langle u\otimes 1,f\right\rangle &=&\left\langle 1\left( y\right)
,\left\langle u\left( x\right) ,\varphi \left( x\right) \psi \left(
x-y\right) \right\rangle \right\rangle \\
&=&\int \left\langle u,\varphi \tau _{y}\psi \right\rangle \mathtt{d}y.
\end{eqnarray*}%
It follows that%
\begin{equation*}
\left( \int \psi \right) \left\langle u,\varphi \right\rangle =\int
\left\langle u,\varphi \tau _{y}\psi \right\rangle \mathtt{d}y
\end{equation*}%
valid for $\varphi ,\psi \in \mathcal{C}_{0}^{\infty }\left( \mathbb{R}%
^{n}\right) $ and $u\in \mathcal{D}^{\prime }\left( \mathbb{R}^{n}\right) $
(or $\varphi ,\psi \in \mathcal{S}\left( \mathbb{R}^{n}\right) $ and $u\in 
\mathcal{S}^{\prime }\left( \mathbb{R}^{n}\right) $).

We also have%
\begin{eqnarray*}
\left\langle u\otimes 1,g\right\rangle &=&\left\langle u\left( x\right)
,\left\langle 1\left( y\right) ,\varphi \left( y\right) \psi \left(
x-y\right) \right\rangle \right\rangle =\left\langle u\left( x\right)
,\left( \varphi \ast \psi \right) \left( x\right) \right\rangle \\
&=&\left\langle u,\varphi \ast \psi \right\rangle , \\
\left\langle u\otimes 1,g\right\rangle &=&\left\langle 1\left( y\right)
,\left\langle u\left( x\right) ,\varphi \left( y\right) \psi \left(
x-y\right) \right\rangle \right\rangle \\
&=&\int \varphi \left( y\right) \left\langle u,\tau _{y}\psi \right\rangle 
\mathtt{d}y.
\end{eqnarray*}%
Hence%
\begin{equation*}
\left\langle u,\varphi \ast \psi \right\rangle =\int \varphi \left( y\right)
\left\langle u,\tau _{y}\psi \right\rangle \mathtt{d}y
\end{equation*}%
true for $\varphi ,\psi \in \mathcal{C}_{0}^{\infty }\left( \mathbb{R}%
^{n}\right) $ and $u\in \mathcal{D}^{\prime }\left( \mathbb{R}^{n}\right) $
(or $\varphi ,\psi \in \mathcal{S}\left( \mathbb{R}^{n}\right) $ and $u\in 
\mathcal{S}^{\prime }\left( \mathbb{R}^{n}\right) $).

\begin{lemma}
Let $\varphi ,\psi \in \mathcal{C}_{0}^{\infty }\left( \mathbb{R}^{n}\right) 
$ and $u\in \mathcal{D}^{\prime }\left( \mathbb{R}^{n}\right) $ $($or $%
\varphi ,\psi \in \mathcal{S}\left( \mathbb{R}^{n}\right) $ and $u\in 
\mathcal{S}^{\prime }\left( \mathbb{R}^{n}\right) )$. Then 
\begin{equation}
\left( \int \psi \right) \left\langle u,\varphi \right\rangle =\int
\left\langle u,\varphi \tau _{y}\psi \right\rangle \mathtt{d}y  \label{ks5}
\end{equation}%
\begin{equation}
\left\langle u,\varphi \ast \psi \right\rangle =\int \varphi \left( y\right)
\left\langle u,\tau _{y}\psi \right\rangle \mathtt{d}y  \label{ks9}
\end{equation}
\end{lemma}

If $\varepsilon _{1},...,\varepsilon _{n}$ is a basis in $\mathbb{R}^{n}$,
we say that $\Gamma =\oplus _{j=1}^{n}\mathbb{Z}\varepsilon _{j}$ is a
lattice.

Let $\Gamma \subset \mathbb{R}^{n}$ be a lattice. Let $\psi \in \mathcal{S}%
\left( \mathbb{R}^{n}\right) $. Then $\sum_{\gamma \in \Gamma }\tau _{\gamma
}\psi =\sum_{\gamma \in \Gamma }\psi \left( \cdot -\gamma \right) $ is
uniformly convergent on compact subsets of $\mathbb{R}^{n}$. Since $\partial
^{\alpha }\psi \in \mathcal{S}\left( \mathbb{R}^{n}\right) $, it follows
that there is $\Psi \in \mathcal{C}^{\infty }\left( \mathbb{R}^{n}\right) $
such that $\Psi =\sum_{\gamma \in \Gamma }\tau _{\gamma }\psi $ in $\mathcal{%
C}^{\infty }\left( \mathbb{R}^{n}\right) $. Moreover, we have $\tau _{\gamma
}\Psi =\Psi \left( \cdot -\gamma \right) =\Psi $ for any $\gamma \in \Gamma $%
. Consequently we have $\Psi \in \mathcal{BC}^{\infty }\left( \mathbb{R}%
^{n}\right) $ and if $\Psi \left( y\right) \neq 0$ for any $y\in \mathbb{R}%
^{n}$, then $\frac{1}{\Psi }\in \mathcal{BC}^{\infty }\left( \mathbb{R}%
^{n}\right) $.

Let $\varphi \in \mathcal{S}\left( \mathbb{R}^{n}\right) $. Then $\varphi
\Psi =\sum_{\gamma \in \Gamma }\varphi \left( \tau _{\gamma }\psi \right) $
with the series convergent in $\mathcal{S}\left( \mathbb{R}^{n}\right) $.
Indeed, by applying Peetre's inequality $\left\langle \gamma \right\rangle
^{n+1}\leq 2^{\frac{n+1}{2}}\left\langle x\right\rangle ^{n+1}\left\langle
x-\gamma \right\rangle ^{n+1}$ one obtains%
\begin{multline*}
\sum_{\gamma \in \Gamma }\left\langle x\right\rangle ^{k}\left\vert \partial
^{\alpha }\varphi \left( x\right) \partial ^{\beta }\psi \left( x-\gamma
\right) \right\vert \\
\leq 2^{\frac{n+1}{2}}\left( \sum_{\gamma \in \Gamma }\left\langle \gamma
\right\rangle ^{-n-1}\right) \cdot \sup \left\langle \cdot \right\rangle
^{k+n+1}\left\vert \partial ^{\alpha }\varphi \right\vert \cdot \sup
\left\langle \cdot \right\rangle ^{n+1}\left\vert \partial ^{\beta }\psi
\right\vert
\end{multline*}%
and this estimate proves the convergence of the series in $\mathcal{S}\left( 
\mathbb{R}^{n}\right) $. Let $\chi $ be the sum of the series $\sum_{\gamma
\in \Gamma }\varphi \left( \tau _{\gamma }\psi \right) $ in $\mathcal{S}%
\left( \mathbb{R}^{n}\right) $. Then for any $y\in \mathbb{R}^{n}$ we have%
\begin{eqnarray*}
\chi \left( y\right) &=&\left\langle \delta _{y},\chi \right\rangle
=\left\langle \delta _{y},\sum_{\gamma \in \Gamma }\varphi \left( \tau
_{\gamma }\psi \right) \right\rangle =\sum_{\gamma \in \Gamma }\left\langle
\delta _{y},\varphi \left( \tau _{\gamma }\psi \right) \right\rangle \\
&=&\sum_{\gamma \in \Gamma }\varphi \left( y\right) \psi \left( y-\gamma
\right) =\varphi \left( y\right) \Psi \left( y\right) .
\end{eqnarray*}%
So $\varphi \Psi =\sum_{\gamma \in \Gamma }\varphi \left( \tau _{\gamma
}\psi \right) $ in $\mathcal{S}\left( \mathbb{R}^{n}\right) $.

If $\psi ,\varphi \in \mathcal{C}_{0}^{\infty }\left( \mathbb{R}^{n}\right) $
and $\mathcal{S}\left( \mathbb{R}^{n}\right) $ is replaced by $\mathcal{C}%
_{0}^{\infty }\left( \mathbb{R}^{n}\right) $, then the previous remarks are
trivial.

\begin{lemma}
Let $\psi ,\varphi \in \mathcal{S}\left( \mathbb{R}^{n}\right) $ and $u\in 
\mathcal{S}^{\prime }\left( \mathbb{R}^{n}\right) $ $($or $\psi ,\varphi \in 
\mathcal{C}_{0}^{\infty }\left( \mathbb{R}^{n}\right) $ and $u\in \mathcal{D}%
^{\prime }\left( \mathbb{R}^{n}\right) )$. Then $\Psi =\sum_{\gamma \in
\Gamma }\tau _{\gamma }\psi \in \mathcal{BC}^{\infty }\left( \mathbb{R}%
^{n}\right) $ is $\Gamma $-periodic and%
\begin{equation}
\left\langle u,\Psi \varphi \right\rangle =\sum_{\gamma \in \Gamma
}\left\langle u,\left( \tau _{\gamma }\psi \right) \varphi \right\rangle .
\label{ks6}
\end{equation}
\end{lemma}

\begin{lemma}
$(\mathtt{a})$ Let $\chi \in \mathcal{S}\left( \mathbb{R}^{n}\right) $ and $%
u\in \mathcal{S}^{\prime }\left( \mathbb{R}^{n}\right) $. Then $\widehat{%
\chi u}\in \mathcal{S}^{\prime }\left( \mathbb{R}^{n}\right) \cap \mathcal{C}%
_{pol}^{\infty }\left( \mathbb{R}^{n}\right) $. In fact we have 
\begin{equation*}
\widehat{\chi u}\left( \xi \right) =\left\langle \mathtt{e}^{-\mathtt{i}%
\left\langle \cdot ,\xi \right\rangle }u,\chi \right\rangle =\left\langle u,%
\mathtt{e}^{-\mathtt{i}\left\langle \cdot ,\xi \right\rangle }\chi
\right\rangle ,\quad \xi \in \mathbb{R}^{n}.
\end{equation*}

$(\mathtt{b})$ Let $u\in \mathcal{D}^{\prime }\left( \mathbb{R}^{n}\right) $ 
$($or $u\in \mathcal{S}^{\prime }\left( \mathbb{R}^{n}\right) )$ and $\chi
\in \mathcal{C}_{0}^{\infty }\left( \mathbb{R}^{n}\right) $ $($or $\chi \in 
\mathcal{S}\left( \mathbb{R}^{n}\right) )$. Then 
\begin{equation*}
\mathbb{R}^{n}\times \mathbb{R}^{n}\ni \left( y,\xi \right) \rightarrow 
\widehat{u\tau _{y}\chi }\left( \xi \right) =\left\langle u,\mathtt{e}^{-%
\mathtt{i}\left\langle \cdot ,\xi \right\rangle }\chi \left( \cdot -y\right)
\right\rangle \in \mathbb{C}
\end{equation*}%
is a $\mathcal{C}^{\infty }$-function.
\end{lemma}

\begin{proof}
Let $q:\mathbb{R}_{x}^{n}\times \mathbb{R}_{\xi }^{n}\rightarrow \mathbb{R}$%
, $q\left( x,\xi \right) =\left\langle x,\xi \right\rangle $. Then $\mathtt{e%
}^{-\mathtt{i}q}\left( u\otimes 1\right) \in \mathcal{S}^{\prime }\left( 
\mathbb{R}_{x}^{n}\times \mathbb{R}_{\xi }^{n}\right) $. If $\varphi \in 
\mathcal{S}\left( \mathbb{R}_{\xi }^{n}\right) $, then we have%
\begin{eqnarray*}
\left\langle \widehat{\chi u},\varphi \right\rangle &=&\left\langle u,\chi 
\widehat{\varphi }\right\rangle =\left\langle u\left( x\right) ,\left\langle
1\left( \xi \right) ,\mathtt{e}^{-\mathtt{i}q\left( x,\xi \right) }\chi
\left( x\right) \varphi \left( \xi \right) \right\rangle \right\rangle \\
&=&\left\langle u\otimes 1,\mathtt{e}^{-\mathtt{i}q}\left( \chi \otimes
\varphi \right) \right\rangle =\left\langle 1\left( \xi \right)
,\left\langle u\left( x\right) ,\mathtt{e}^{-\mathtt{i}\left\langle x,\xi
\right\rangle }\chi \left( x\right) \varphi \left( \xi \right) \right\rangle
\right\rangle \\
&=&\left\langle 1\left( \xi \right) ,\varphi \left( \xi \right) \left\langle
u,\mathtt{e}^{-\mathtt{i}\left\langle \cdot ,\xi \right\rangle }\chi
\right\rangle \right\rangle =\left\langle 1\left( \xi \right) ,\varphi
\left( \xi \right) \left\langle \mathtt{e}^{-\mathtt{i}\left\langle \cdot
,\xi \right\rangle }u,\chi \right\rangle \right\rangle \\
&=&\int \varphi \left( \xi \right) \left\langle \mathtt{e}^{-\mathtt{i}%
\left\langle \cdot ,\xi \right\rangle }u,\chi \right\rangle \mathtt{d}\xi
\end{eqnarray*}%
This proves that $\widehat{\chi u}\left( \xi \right) =\left\langle \mathtt{e}%
^{-\mathtt{i}\left\langle \cdot ,\xi \right\rangle }u,\chi \right\rangle $, $%
\xi \in \mathbb{R}^{n}$.
\end{proof}

\begin{corollary}
Let $k\in \mathcal{K}_{pol}\left( \mathbb{R}^{n}\right) $, $u\in \mathcal{D}%
^{\prime }\left( \mathbb{R}^{n}\right) $ $($or $u\in \mathcal{S}^{\prime
}\left( \mathbb{R}^{n}\right) )$ and $\chi \in \mathcal{C}_{0}^{\infty
}\left( \mathbb{R}^{n}\right) $ $($or $\chi \in \mathcal{S}\left( \mathbb{R}%
^{n}\right) )$. Then 
\begin{equation*}
\mathbb{R}^{n}\ni y\rightarrow f\left( y\right) =\left\Vert u\tau _{y}\chi
\right\Vert _{\mathcal{B}_{k}}\in \left[ 0,\infty \right]
\end{equation*}%
is a measurable function.
\end{corollary}

\begin{proof}
Let $\left( K_{m}\right) _{m\in \mathbb{N}}$ be a secuence of compact
subsets of $\mathbb{R}^{n}$ such that $K_{m}\subset \mathring{K}_{m+1}$ and $%
\bigcup K_{m}=\mathbb{R}^{n}$. Then $f_{m}\nearrow f$, where $f_{m}$ is the
continuous function on $\mathbb{R}^{n}$ defined by 
\begin{equation*}
f_{m}\left( y\right) =\left( \int_{K_{m}}\left\vert \widehat{u\tau _{y}\chi }%
\left( \xi \right) k\left( \xi \right) \right\vert ^{p}\mathtt{d}\xi \right)
^{1/p}.
\end{equation*}
\end{proof}

Let $u\in \mathcal{D}^{\prime }\left( \mathbb{R}^{n}\right) $ $($or $u\in 
\mathcal{S}^{\prime }\left( \mathbb{R}^{n}\right) )$ and $\chi \in \mathcal{C%
}_{0}^{\infty }\left( \mathbb{R}^{n}\right) \smallsetminus 0$ $($or $\chi
\in \mathcal{S}\left( \mathbb{R}^{n}\right) \smallsetminus 0)$. Let $%
\widetilde{\chi }\in \mathcal{C}_{0}^{\infty }\left( \mathbb{R}^{n}\right) $ 
$($or $\widetilde{\chi }\in \mathcal{S}\left( \mathbb{R}^{n}\right) )$and $%
\varphi \in \mathcal{C}_{0}^{\infty }\left( \mathbb{R}^{n}\right) $. By
using (\ref{ks5}) and Lemma \ref{ks18} $\left( \mathtt{d}\right) $ we get 
\begin{eqnarray*}
\left\langle u\tau _{z}\widetilde{\chi },\varphi \right\rangle &=&\frac{1}{%
\left\Vert \chi \right\Vert _{L^{2}}^{2}}\int \left\langle u\tau _{z}%
\widetilde{\chi },\left( \tau _{y}\chi \right) \left( \tau _{y}\overline{%
\chi }\right) \varphi \right\rangle \mathtt{d}y \\
&=&\frac{1}{\left\Vert \chi \right\Vert _{L^{2}}^{2}}\int \left\langle u\tau
_{y}\chi ,\left( \tau _{z}\widetilde{\chi }\right) \left( \tau _{y}\overline{%
\chi }\right) \varphi \right\rangle \mathtt{d}y,
\end{eqnarray*}%
\begin{equation*}
\left\vert \left\langle u\tau _{z}\widetilde{\chi },\varphi \right\rangle
\right\vert \leq \frac{\left( 2\pi \right) ^{-n}}{\left\Vert \chi
\right\Vert _{L^{2}}^{2}}\int \left\Vert u\tau _{y}\chi \right\Vert _{%
\mathcal{B}_{k}}\left\Vert \left( \tau _{z}\widetilde{\chi }\right) \left(
\tau _{y}\overline{\chi }\right) \varphi \right\Vert _{\mathcal{B}_{1/\check{%
k}}}\mathtt{d}y.
\end{equation*}

Let $\Gamma \subset \mathbb{R}^{n}$ be a lattice. Let $u\in \mathcal{D}%
^{\prime }\left( \mathbb{R}^{n}\right) $ $($or $u\in \mathcal{S}^{\prime
}\left( \mathbb{R}^{n}\right) )$ and let $\chi \in \mathcal{C}_{0}^{\infty
}\left( \mathbb{R}^{n}\right) $ $($or $\chi \in \mathcal{S}\left( \mathbb{R}%
^{n}\right) )$ be such that $\Psi =\Psi _{\Gamma ,\chi }=\sum_{\gamma \in
\Gamma }\left\vert \tau _{\gamma }\chi \right\vert ^{2}>0$. Then $\Psi ,%
\frac{1}{\Psi }\in \mathcal{BC}^{\infty }\left( \mathbb{R}^{n}\right) $ and
both are $\Gamma $-periodic. Let $\widetilde{\chi }\in \mathcal{C}%
_{0}^{\infty }\left( \mathbb{R}^{n}\right) $ $($or $\widetilde{\chi }\in 
\mathcal{S}\left( \mathbb{R}^{n}\right) )$. Using (\ref{ks6}) and Lemma \ref%
{ks18} $\left( \mathtt{d}\right) $ we obtain that 
\begin{eqnarray*}
\left\langle u\tau _{z}\widetilde{\chi },\varphi \right\rangle
&=&\sum_{\gamma \in \Gamma }\left\langle u\tau _{\gamma }\chi ,\frac{1}{\Psi 
}\left( \tau _{\gamma }\overline{\chi }\right) \left( \tau _{z}\widetilde{%
\chi }\right) \varphi \right\rangle , \\
\left\vert \left\langle u\tau _{z}\widetilde{\chi },\varphi \right\rangle
\right\vert &\leq &\left( 2\pi \right) ^{-n}\sum_{\gamma \in \Gamma
}\left\Vert u\tau _{\gamma }\chi \right\Vert _{\mathcal{B}_{k}}\left\Vert 
\frac{1}{\Psi }\left( \tau _{\gamma }\overline{\chi }\right) \left( \tau _{z}%
\widetilde{\chi }\right) \varphi \right\Vert _{\mathcal{B}_{1/\check{k}}} \\
&\leq &C_{\Psi }\sum_{\gamma \in \Gamma }\left\Vert u\tau _{\gamma }\chi
\right\Vert _{\mathcal{B}_{k}}\left\Vert \left( \tau _{\gamma }\overline{%
\chi }\right) \left( \tau _{z}\widetilde{\chi }\right) \varphi \right\Vert _{%
\mathcal{B}_{1/\check{k}}}.
\end{eqnarray*}%
In the last inequality we used the Proposition \ref{ks4} and the fact that $%
\frac{1}{\Psi }\in \mathcal{BC}^{\infty }\left( \mathbb{R}^{n}\right) $.

If $\left( Y,\mathtt{\mu }\right) $ is either $\mathbb{R}^{n}$ with Lebesgue
measure or $\Gamma $ with the counting measure, then the previous estimates
can be written as:%
\begin{equation*}
\left\vert \left\langle u\tau _{z}\widetilde{\chi },\varphi \right\rangle
\right\vert \leq Cst\int_{Y}\left\Vert u\tau _{y}\chi \right\Vert _{\mathcal{%
B}_{k}}\left\Vert \left( \tau _{z}\widetilde{\chi }\right) \left( \tau _{y}%
\overline{\chi }\right) \varphi \right\Vert _{\mathcal{B}_{1/\check{k}}}%
\mathtt{d\mu }\left( y\right)
\end{equation*}

We shall use Proposition \ref{ks4} to estimate $\left\Vert \left( \tau _{z}%
\widetilde{\chi }\right) \left( \tau _{y}\overline{\chi }\right) \varphi
\right\Vert _{\mathcal{B}_{1/\check{k}}}$. Let us note that $m_{k}=m_{\check{%
k}}=m_{1/\check{k}}=\left[ N+\frac{n+1}{2}\right] +1$. Then we have%
\begin{equation*}
\left\Vert \left( \tau _{z}\widetilde{\chi }\right) \left( \tau _{y}%
\overline{\chi }\right) \varphi \right\Vert _{\mathcal{B}_{1/\check{k}}}\leq
Cst\sup_{\left\vert \alpha +\beta \right\vert \leq m_{k}}\left\vert \left(
\left( \tau _{z}\partial ^{\alpha }\widetilde{\chi }\right) \left( \tau
_{y}\partial ^{\beta }\overline{\chi }\right) \right) \right\vert \left\Vert
\varphi \right\Vert _{\mathcal{B}_{1/\check{k}}}.
\end{equation*}%
There is a continuous seminorm $p_{n,k}$ on $\mathcal{S}\left( \mathbb{R}%
^{n}\right) $ so that 
\begin{eqnarray*}
\left\vert \left( \tau _{z}\partial ^{\alpha }\widetilde{\chi }\right)
\left( \tau _{y}\partial ^{\beta }\overline{\chi }\right) \left( x\right)
\right\vert &\leq &p_{n,k}\left( \widetilde{\chi }\right) p_{n,k}\left( \chi
\right) \left\langle x-z\right\rangle ^{-2\left( n+1\right) }\left\langle
x-y\right\rangle ^{-2\left( n+1\right) } \\
&\leq &2^{n+1}p_{n,k}\left( \widetilde{\chi }\right) p_{n,k}\left( \chi
\right) \left\langle 2x-z-y\right\rangle ^{-n-1}\left\langle
z-y\right\rangle ^{-n-1} \\
&\leq &2^{n+1}p_{n,k}\left( \widetilde{\chi }\right) p_{n,k}\left( \chi
\right) \left\langle z-y\right\rangle ^{-n-1},\quad \left\vert \alpha +\beta
\right\vert \leq m_{k}.
\end{eqnarray*}%
Here we used the inequality 
\begin{equation*}
\left\langle X\right\rangle ^{-2\left( n+1\right) }\left\langle
Y\right\rangle ^{-2\left( n+1\right) }\leq 2^{n+1}\left\langle
X+Y\right\rangle ^{-n-1}\left\langle X-Y\right\rangle ^{-n-1},\quad X,Y\in 
\mathbb{R}^{m}
\end{equation*}%
which is a consequence of Peetre's inequality $\left\langle X\pm
Y\right\rangle ^{n+1}\leq 2^{\frac{n+1}{2}}\left\langle X\right\rangle
^{n+1}\left\langle Y\right\rangle ^{n+1}$. Hence 
\begin{gather*}
\sup_{\left\vert \alpha +\beta \right\vert \leq m_{k}}\left\vert \left(
\left( \tau _{z}\partial ^{\alpha }\widetilde{\chi }\right) \left( \tau
_{y}\partial ^{\beta }\overline{\chi }\right) \right) \right\vert \leq
2^{n+1}p_{n,k}\left( \widetilde{\chi }\right) p_{n,k}\left( \chi \right)
\left\langle z-y\right\rangle ^{-n-1}, \\
\left\Vert \left( \tau _{z}\widetilde{\chi }\right) \left( \tau _{y}%
\overline{\chi }\right) \varphi \right\Vert _{\mathcal{B}_{1/\check{k}}}\leq
C^{\prime }\left( n,k,\chi \mathbf{,}\widetilde{\chi }\right) \left\langle
z-y\right\rangle ^{-n-1}\left\Vert \varphi \right\Vert _{\mathcal{B}_{1/%
\check{k}}}, \\
\left\vert \left\langle u\tau _{z}\widetilde{\chi },\varphi \right\rangle
\right\vert \leq C^{\prime }\left( n,k,\chi \mathbf{,}\widetilde{\chi }%
\right) \left( \int_{Y}\left\Vert u\tau _{y}\chi \right\Vert _{\mathcal{B}%
_{k}}\left\langle z-y\right\rangle ^{-n-1}\mathtt{d\mu }\left( y\right)
\right) \left\Vert \varphi \right\Vert _{\mathcal{B}_{1/\check{k}}}.
\end{gather*}%
The last estimate implies that 
\begin{equation}
\left\Vert u\tau _{z}\widetilde{\chi }\right\Vert _{\mathcal{B}_{k}}\leq
C\left( n,k,\chi \mathbf{,}\widetilde{\chi }\right) \left(
\int_{Y}\left\Vert u\tau _{y}\chi \right\Vert _{\mathcal{B}_{k}}\left\langle
z-y\right\rangle ^{-n-1}\mathtt{d\mu }\left( y\right) \right) .  \label{kh13}
\end{equation}%
Let $1\leq p<\infty $. If $\left( Z,\mathtt{\upsilon }\right) $ is either $%
\mathbb{R}^{n}$ with Lebesgue measure or a lattice with the counting
measure, then Schur's lemma implies%
\begin{equation*}
\left( \int_{Z}\left\Vert u\tau _{z}\widetilde{\chi }\right\Vert _{\mathcal{B%
}_{k}}^{p}\mathtt{d\upsilon }\left( z\right) \right) ^{\frac{1}{p}}\leq
C\left( n,k,\chi \mathbf{,}\widetilde{\chi }\right) \left\Vert \left\langle
\cdot \right\rangle ^{-n-1}\right\Vert _{L^{1}}\left( \int_{Y}\left\Vert
u\tau _{y}\chi \right\Vert _{\mathcal{B}_{k}}^{p}\mathtt{d\mu }\left(
y\right) \right) ^{\frac{1}{p}}.
\end{equation*}%
For $p=\infty $ we have 
\begin{equation*}
\sup_{z}\left\Vert u\tau _{z}\widetilde{\chi }\right\Vert _{\mathcal{B}%
_{k}}\leq C\left( n,k,\chi \mathbf{,}\widetilde{\chi }\right) \left\Vert
\left\langle \cdot \right\rangle ^{-n-1}\right\Vert
_{L^{1}}\sup_{y}\left\Vert u\tau _{y}\chi \right\Vert _{\mathcal{B}_{k}}.
\end{equation*}%
By taking different combinations of $\left( Y,\mathtt{\mu }\right) $ and $%
\left( Z,\mathtt{\upsilon }\right) $ we obtain the following result.

\begin{proposition}
\label{ks8}Let $k\in \mathcal{K}_{pol}\left( \mathbb{R}^{n}\right) $ and $C$%
, $N$ the positive constants that define $k$ and $1\leq p<\infty $. Let $%
u\in \mathcal{D}^{\prime }\left( \mathbb{R}^{n}\right) $ $($or $u\in 
\mathcal{S}^{\prime }\left( \mathbb{R}^{n}\right) )$ and $\chi \in \mathcal{C%
}_{0}^{\infty }\left( \mathbb{R}^{n}\right) \smallsetminus 0$ $($or $\chi
\in \mathcal{S}\left( \mathbb{R}^{n}\right) \smallsetminus 0)$.

$\left( \mathtt{a}\right) $ If $\widetilde{\chi }\in \mathcal{C}_{0}^{\infty
}\left( \mathbb{R}^{n}\right) $ $($or $\widetilde{\chi }\in \mathcal{S}%
\left( \mathbb{R}^{n}\right) )$, then there is $C\left( n,k,\chi \mathbf{,}%
\widetilde{\chi }\right) >0$ such that%
\begin{eqnarray*}
\left( \int \left\Vert u\tau _{\widetilde{y}}\widetilde{\chi }\right\Vert _{%
\mathcal{B}_{k}}^{p}\mathtt{d}\widetilde{y}\right) ^{\frac{1}{p}} &\leq
&C\left( n,k,\chi \mathbf{,}\widetilde{\chi }\right) \left( \int \left\Vert
u\tau _{y}\chi \right\Vert _{\mathcal{B}_{k}}^{p}\mathtt{d}y\right) ^{\frac{1%
}{p}}, \\
\sup_{\widetilde{y}}\left\Vert u\tau _{\widetilde{y}}\widetilde{\chi }%
\right\Vert _{\mathcal{B}_{k}} &\leq &C\left( n,k,\chi \mathbf{,}\widetilde{%
\chi }\right) \sup_{y}\left\Vert u\tau _{y}\chi \right\Vert _{\mathcal{B}%
_{k}}.
\end{eqnarray*}

$\left( \mathtt{b}\right) $ If $\Gamma \subset \mathbb{R}^{n}$ is a lattice
such that $\Psi =\Psi _{\Gamma ,\chi }=\sum_{\gamma \in \Gamma }\left\vert
\tau _{\gamma }\chi \right\vert ^{2}>0$ and $\widetilde{\chi }\in \mathcal{C}%
_{0}^{\infty }\left( \mathbb{R}^{n}\right) $ $($or $\widetilde{\chi }\in 
\mathcal{S}\left( \mathbb{R}^{n}\right) )$, then there is $C\left(
n,k,\Gamma ,\chi \mathbf{,}\widetilde{\chi }\right) >0$ such that%
\begin{eqnarray*}
\left( \int \left\Vert u\tau _{\widetilde{y}}\widetilde{\chi }\right\Vert _{%
\mathcal{B}_{k}}^{p}\mathtt{d}\widetilde{y}\right) ^{\frac{1}{p}} &\leq
&C\left( n,k,\Gamma ,\chi \mathbf{,}\widetilde{\chi }\right) \left(
\sum_{\gamma \in \Gamma }\left\Vert u\tau _{\gamma }\chi \right\Vert _{%
\mathcal{B}_{k}}^{p}\right) ^{\frac{1}{p}}, \\
\sup_{\widetilde{y}}\left\Vert u\tau _{\widetilde{y}}\widetilde{\chi }%
\right\Vert _{\mathcal{B}_{k}} &\leq &C\left( n,k,\Gamma ,\chi \mathbf{,}%
\widetilde{\chi }\right) \sup_{\gamma }\left\Vert u\tau _{\gamma }\chi
\right\Vert _{\mathcal{B}_{k}}.
\end{eqnarray*}

$\left( \mathtt{c}\right) $ If $\widetilde{\Gamma }\subset \mathbb{R}^{n}$
is a lattice and $\widetilde{\chi }\in \mathcal{C}_{0}^{\infty }\left( 
\mathbb{R}^{n}\right) $ $($or $\widetilde{\chi }\in \mathcal{S}\left( 
\mathbb{R}^{n}\right) )$, then there is $C\left( n,k,\widetilde{\Gamma }%
,\chi \mathbf{,}\widetilde{\chi }\right) >0$ such that%
\begin{eqnarray*}
\left( \sum_{\widetilde{\gamma }\in \widetilde{\Gamma }}\left\Vert u\tau _{%
\widetilde{\gamma }}\widetilde{\chi }\right\Vert _{\mathcal{B}%
_{k}}^{p}\right) ^{\frac{1}{p}} &\leq &C\left( n,k,\widetilde{\Gamma },\chi 
\mathbf{,}\widetilde{\chi }\right) \left( \int \left\Vert u\tau _{y}\chi
\right\Vert _{\mathcal{B}_{k}}^{p}\mathtt{d}y\right) ^{\frac{1}{p}}, \\
\sup_{\widetilde{\gamma }}\left\Vert u\tau _{\widetilde{\gamma }}\widetilde{%
\chi }\right\Vert _{\mathcal{B}_{k}} &\leq &C\left( n,k,\widetilde{\Gamma }%
,\chi \mathbf{,}\widetilde{\chi }\right) \sup_{y}\left\Vert u\tau _{y}\chi
\right\Vert _{\mathcal{B}_{k}}.
\end{eqnarray*}

$\left( \mathtt{d}\right) $ If $\Gamma ,\widetilde{\Gamma }\subset \mathbb{R}%
^{n}$ are lattices such that $\Psi =\Psi _{\Gamma ,\chi }=\sum_{\gamma \in
\Gamma }\left\vert \tau _{\gamma }\chi \right\vert ^{2}>0$ and $\widetilde{%
\chi }\in \mathcal{C}_{0}^{\infty }\left( \mathbb{R}^{n}\right) $ $($or $%
\widetilde{\chi }\in \mathcal{S}\left( \mathbb{R}^{n}\right) )$, then there
is $C\left( n,k\mathbf{,}\Gamma ,\widetilde{\Gamma },\chi \mathbf{,}%
\widetilde{\chi }\right) >0$ such that 
\begin{eqnarray*}
\left( \sum_{\widetilde{\gamma }\in \widetilde{\Gamma }}\left\Vert u\tau _{%
\widetilde{\gamma }}\chi \right\Vert _{\mathcal{B}_{k}}^{p}\right) ^{\frac{1%
}{p}} &\leq &C\left( n,k\mathbf{,}\Gamma ,\widetilde{\Gamma },\chi \mathbf{,}%
\widetilde{\chi }\right) \left( \sum_{\gamma \in \Gamma }\left\Vert u\tau
_{\gamma }\chi \right\Vert _{\mathcal{B}_{k}}^{p}\right) ^{\frac{1}{p}}, \\
\sup_{\widetilde{\gamma }}\left\Vert u\tau _{\widetilde{\gamma }}\widetilde{%
\chi }\right\Vert _{\mathcal{B}_{k}} &\leq &C\left( n,k\mathbf{,}\Gamma ,%
\widetilde{\Gamma },\chi \mathbf{,}\widetilde{\chi }\right) \sup_{\gamma
}\left\Vert u\tau _{\gamma }\chi \right\Vert _{\mathcal{B}_{k}}.
\end{eqnarray*}
\end{proposition}

Let us introduce the space%
\begin{equation*}
\mathcal{B}_{k}^{loc}\left( \mathbb{R}^{n}\right) =\left\{ u;u\in \mathcal{D}%
^{\prime }\left( \mathbb{R}^{n}\right) ,\phi u\in \mathcal{B}_{k}\left( 
\mathbb{R}^{n}\right) \text{ \textit{for every }}\phi \in \mathcal{C}%
_{0}^{\infty }\left( \mathbb{R}^{n}\right) \right\} .
\end{equation*}%
The next result concerns the regularity of the map $z\rightarrow u\tau _{z}%
\widetilde{\chi }$.

\begin{lemma}
If $u\in \mathcal{B}_{k}^{loc}\left( \mathbb{R}^{n}\right) $ and $\widetilde{%
\chi }\in \mathcal{C}_{0}^{\infty }\left( \mathbb{R}^{n}\right) $, then the
function 
\begin{equation*}
\mathbb{R}^{n}\ni z\rightarrow u\tau _{z}\widetilde{\chi }\in \mathcal{B}%
_{k}\left( \mathbb{R}^{n}\right)
\end{equation*}%
is locally Lipschitz.
\end{lemma}

\begin{proof}
Since local Lipschitz continuity is a local property and $u\in \mathcal{B}%
_{k}^{loc}\left( \mathbb{R}^{n}\right) $, we can assume that $u\in \mathcal{B%
}_{k}\left( \mathbb{R}^{n}\right) $. Let $\chi \in \mathcal{C}_{0}^{\infty
}\left( \mathbb{R}^{n}\right) \smallsetminus 0$. During the proof of the
last proposition we proved that there is $C=C\left( n,k\right) $ such that 
\begin{equation*}
\left\Vert u\tau _{z}\widetilde{\chi }\right\Vert _{\mathcal{B}_{k}}\leq
Cp_{n,k}\left( \widetilde{\chi }\right) p_{n,k}\left( \chi \right) \left(
\int \left\Vert u\tau _{y}\chi \right\Vert _{\mathcal{B}_{k}}\left\langle
z-y\right\rangle ^{-n-1}\mathtt{d}y\right) .
\end{equation*}%
Now if we replace $\widetilde{\chi }$ with $\tau _{h}\widetilde{\chi }-%
\widetilde{\chi }$, $\left\vert h\right\vert \leq 1$, we can find a seminorm 
$q_{n,k}$ on $\mathcal{S}\left( \mathbb{R}^{n}\right) $ such that $%
p_{n,k}\left( \tau _{h}\widetilde{\chi }-\widetilde{\chi }\right) \leq
\left\vert h\right\vert q_{n,k}\left( \widetilde{\chi }\right) $ and%
\begin{eqnarray*}
\left\Vert u\tau _{z+h}\widetilde{\chi }-u\tau _{z}\widetilde{\chi }%
\right\Vert _{\mathcal{B}_{k}} &\leq &Cq_{n,k}\left( \widetilde{\chi }%
\right) p_{n,k}\left( \chi \right) \left( \int \left\Vert u\tau _{y}\chi
\right\Vert _{\mathcal{B}_{k}}\left\langle z-y\right\rangle ^{-n-1}\mathtt{d}%
y\right) \left\vert h\right\vert \\
&\leq &C^{\prime }q_{n,k}\left( \widetilde{\chi }\right) p_{n,k}\left( \chi
\right) \left\Vert u\right\Vert _{\mathcal{B}_{k}}\left\Vert \chi
\right\Vert _{\mathcal{BC}^{m_{k}}}\left\Vert \left\langle \cdot
\right\rangle ^{-n-1}\right\Vert _{L^{1}}\left\vert h\right\vert .
\end{eqnarray*}
\end{proof}

\begin{definition}
Let $1\leq p\leq \infty $, $k\in \mathcal{K}_{pol}\left( \mathbb{R}%
^{n}\right) $ and $u\in \mathcal{D}^{\prime }\left( \mathbb{R}^{n}\right) $.
We say that $u$ belongs to $\mathcal{B}_{k}^{p}\left( \mathbb{R}^{n}\right) $
if there is $\chi \in \mathcal{C}_{0}^{\infty }\left( \mathbb{R}^{n}\right)
\smallsetminus 0$ such that the measurable function $\mathbb{R}^{n}\ni
y\rightarrow \left\Vert u\tau _{y}\chi \right\Vert _{\mathcal{B}_{k}}\in 
\mathbb{R}$ belongs to $L^{p}$. We put%
\begin{eqnarray*}
\left\Vert u\right\Vert _{k,p,\chi } &=&\left( \int \left\Vert u\tau
_{y}\chi \right\Vert _{\mathcal{B}_{k}}^{p}\mathtt{d}y\right) ^{\frac{1}{p}%
},\quad 1\leq p<\infty , \\
\left\Vert u\right\Vert _{k,\infty ,\chi } &\equiv &\left\Vert u\right\Vert
_{k,\mathtt{ul},\chi }=\sup_{y}\left\Vert u\tau _{y}\chi \right\Vert _{%
\mathcal{B}_{k}}.
\end{eqnarray*}
\end{definition}

\begin{proposition}
\label{ks15}$\left( \mathtt{a}\right) $ The above definition does not depend
on the choice of the function $\chi \in \mathcal{C}_{0}^{\infty }\left( 
\mathbb{R}^{n}\right) \smallsetminus 0$.

$\left( \mathtt{b}\right) $ If $\chi \in \mathcal{C}_{0}^{\infty }\left( 
\mathbb{R}^{n}\right) \smallsetminus 0$, then $\left\Vert \cdot \right\Vert
_{k,p,\chi }$ is a norm on $\mathcal{B}_{k}^{p}\left( \mathbb{R}^{n}\right) $
and the topology that defines does not depend on the function $\chi $.

$\left( \mathtt{c}\right) $ Let $\Gamma \subset \mathbb{R}^{n}$ be a lattice
and $\chi \in \mathcal{C}_{0}^{\infty }\left( \mathbb{R}^{n}\right) $ be a
function with the property that $\Psi =\Psi _{\Gamma ,\chi }=\sum_{\gamma
\in \Gamma }\left\vert \tau _{\gamma }\chi \right\vert ^{2}>0$. Then 
\begin{equation*}
\mathcal{B}_{k}^{p}\left( \mathbb{R}^{n}\right) \ni u\rightarrow \left\{ 
\begin{array}{cc}
\left( \sum_{\gamma \in \Gamma }\left\Vert u\tau _{\gamma }\chi \right\Vert
_{\mathcal{B}_{k}}^{p}\right) ^{\frac{1}{p}} & 1\leq p<\infty \\ 
\sup_{\gamma }\left\Vert u\tau _{\gamma }\chi \right\Vert _{\mathcal{B}_{k}}
& p=\infty%
\end{array}%
\right.
\end{equation*}%
is a norm on $\mathcal{B}_{k}^{p}\left( \mathbb{R}^{n}\right) $ and the
topology that defines is the topology of $\mathcal{B}_{k}^{p}\left( \mathbb{R%
}^{n}\right) $. We shall use the notation 
\begin{equation*}
\left\Vert u\right\Vert _{k,p,\Gamma ,\chi }=\left\{ 
\begin{array}{cc}
\left( \sum_{\gamma \in \Gamma }\left\Vert u\tau _{\gamma }\chi \right\Vert
_{\mathcal{B}_{k}}^{p}\right) ^{\frac{1}{p}} & 1\leq p<\infty \\ 
\sup_{\gamma }\left\Vert u\tau _{\gamma }\chi \right\Vert _{\mathcal{B}_{k}}
& p=\infty%
\end{array}%
\right. .
\end{equation*}

$\left( \mathtt{d}\right) $ If $1\leq p\leq q\leq \infty $, Then%
\begin{equation*}
\mathcal{S}\left( \mathbb{R}^{n}\right) \subset \mathcal{B}_{k}^{1}\left( 
\mathbb{R}^{n}\right) \subset \mathcal{B}_{k}^{p}\left( \mathbb{R}%
^{n}\right) \subset \mathcal{B}_{k}^{q}\left( \mathbb{R}^{n}\right) \subset 
\mathcal{B}_{k}^{\infty }\left( \mathbb{R}^{n}\right) \subset \mathcal{S}%
^{\prime }\left( \mathbb{R}^{n}\right) .
\end{equation*}

$\left( \mathtt{e}\right) $ If $k^{\prime },k\in k\in \mathcal{K}%
_{pol}\left( \mathbb{R}^{n}\right) $ and $k^{\prime }\leq Cst\cdot k,$ then $%
\mathcal{B}_{k}^{p}\left( \mathbb{R}^{n}\right) \subset \mathcal{B}%
_{k^{\prime }}^{p}\left( \mathbb{R}^{n}\right) $.

$\left( \mathtt{f}\right) $ $\left( \mathcal{B}_{k}^{p}\left( \mathbb{R}%
^{n}\right) ,\left\Vert \cdot \right\Vert _{k,p,\chi }\right) $ is a Banach
space.

$\left( \mathtt{g}\right) $ If $1/k\in L^{2}\left( \mathbb{R}^{n}\right) $,
then $\mathcal{B}_{k}^{\infty }\left( \mathbb{R}^{n}\right) \subset \mathcal{%
B}\mathcal{C}\left( \mathbb{R}^{n}\right) $.
\end{proposition}

\begin{proof}
$\left( \mathtt{a}\right) $ $\left( \mathtt{b}\right) $ $\left( \mathtt{c}%
\right) $ are immediate consequences of the previous proposition.

$\left( \mathtt{d}\right) $ The inclusions $\mathcal{B}_{k}^{1}\left( 
\mathbb{R}^{n}\right) \subset \mathcal{B}_{k}^{p}\left( \mathbb{R}%
^{n}\right) \subset \mathcal{B}_{k}^{q}\left( \mathbb{R}^{n}\right) \subset 
\mathcal{B}_{k}^{\infty }\left( \mathbb{R}^{n}\right) $ are consequences of
the elementary inclusions $l^{1}\subset l^{p}\subset l^{q}\subset l^{\infty
} $. What remain to be shown are the inclusions $\mathcal{S}\left( \mathbb{R}%
^{n}\right) \subset \mathcal{B}_{k}^{1}\left( \mathbb{R}^{n}\right) $, $%
\mathcal{B}_{k}^{\infty }\left( \mathbb{R}^{n}\right) \subset \mathcal{S}%
^{\prime }\left( \mathbb{R}^{n}\right) $.

Let $u\in \mathcal{B}_{k}^{\infty }\left( \mathbb{R}^{n}\right) $, $\chi \in 
\mathcal{C}_{0}^{\infty }\left( \mathbb{R}^{n}\right) \smallsetminus 0$ and $%
\varphi \in \mathcal{C}_{0}^{\infty }\left( \mathbb{R}^{n}\right) $. We have%
\begin{equation*}
\left\langle u,\varphi \right\rangle =\frac{1}{\left\Vert \chi \right\Vert
_{L^{2}}^{2}}\int \left\langle u\tau _{y}\chi ,\left( \tau _{y}\overline{%
\chi }\right) \varphi \right\rangle \mathtt{d}y,
\end{equation*}%
\begin{eqnarray*}
\left\vert \left\langle u,\varphi \right\rangle \right\vert &\leq &\frac{1}{%
\left\Vert \chi \right\Vert _{L^{2}}^{2}}\int \left\vert \left\langle u\tau
_{y}\chi ,\left( \tau _{y}\overline{\chi }\right) \varphi \right\rangle
\right\vert \mathtt{d}y \\
&\leq &\frac{\left( 2\pi \right) ^{-n}}{\left\Vert \chi \right\Vert
_{L^{2}}^{2}}\int \left\Vert u\tau _{y}\chi \right\Vert _{\mathcal{B}%
_{k}}\left\Vert \left( \tau _{y}\overline{\chi }\right) \varphi \right\Vert
_{\mathcal{B}_{1/\check{k}}}\mathtt{d}y \\
&\leq &\frac{\left( 2\pi \right) ^{-n}}{\left\Vert \chi \right\Vert
_{L^{2}}^{2}}\left\Vert u\right\Vert _{k,\infty ,\chi }\int \left\Vert
\left( \tau _{y}\overline{\chi }\right) \varphi \right\Vert _{\mathcal{B}_{1/%
\check{k}}}\mathtt{d}y.
\end{eqnarray*}%
We shall use Proposition \ref{ks4} to estimate $\left\Vert \left( \tau _{y}%
\overline{\chi }\right) \varphi \right\Vert _{\mathcal{B}_{1/\check{k}}}$.
Let $\widetilde{\chi }\in \mathcal{C}_{0}^{\infty }\left( \mathbb{R}%
^{n}\right) $, $\widetilde{\chi }=1$ on $\mathtt{supp}\chi $. If $m_{k}=m_{1/%
\check{k}}=\left[ N+\frac{n+1}{2}\right] +1$, then we obtain that 
\begin{eqnarray*}
\left\Vert \left( \tau _{y}\overline{\chi }\right) \varphi \right\Vert _{%
\mathcal{B}_{1/\check{k}}} &\leq &C\sup_{\left\vert \alpha +\beta
\right\vert \leq m_{k}}\left\vert \left( \partial ^{\alpha }\varphi \right)
\left( \tau _{y}\partial ^{\beta }\overline{\chi }\right) \right\vert
\left\Vert \tau _{y}\widetilde{\chi }\right\Vert _{\mathcal{B}_{1/\check{k}}}
\\
&=&C\sup_{\left\vert \alpha +\beta \right\vert \leq m_{k}}\left\vert \left(
\partial ^{\alpha }\varphi \right) \left( \tau _{y}\partial ^{\beta }%
\overline{\chi }\right) \right\vert \left\Vert \widetilde{\chi }\right\Vert
_{\mathcal{B}_{1/\check{k}}}.
\end{eqnarray*}%
Since $\chi $, $\varphi \in \mathcal{S}\left( \mathbb{R}^{n}\right) $ it
follows that there is a continuous seminorm $p_{n,k}$ on $\mathcal{S}\left( 
\mathbb{R}^{n}\right) $ so that 
\begin{eqnarray*}
\left\vert \left( \partial ^{\alpha }\varphi \right) \left( \tau
_{y}\partial ^{\beta }\overline{\chi }\right) \left( x\right) \right\vert
&\leq &p_{n,k}\left( \varphi \right) p_{n,k}\left( \chi \right) \left\langle
x-y\right\rangle ^{-2\left( n+1\right) }\left\langle x\right\rangle
^{-2\left( n+1\right) } \\
&\leq &2^{n+1}p_{n,k}\left( \varphi \right) p_{n,k}\left( \chi \right)
\left\langle 2x-y\right\rangle ^{-\left( n+1\right) }\left\langle
y\right\rangle ^{-\left( n+1\right) } \\
&\leq &2^{n+1}p_{n,k}\left( \varphi \right) p_{n,k}\left( \chi \right)
\left\langle y\right\rangle ^{-\left( n+1\right) },\quad \left\vert \alpha
+\beta \right\vert \leq m_{k}.
\end{eqnarray*}%
Hence 
\begin{equation*}
\left\vert \left\langle u,\varphi \right\rangle \right\vert \leq 2^{n+1}C%
\frac{\left( 2\pi \right) ^{-n}}{\left\Vert \chi \right\Vert _{L^{2}}^{2}}%
\left\Vert u\right\Vert _{k,\infty ,\chi }\left\Vert \left\langle \cdot
\right\rangle ^{-\left( n+1\right) }\right\Vert _{L^{1}}\left\Vert 
\widetilde{\chi }\right\Vert _{\mathcal{B}_{1/\check{k}}}p_{n,k}\left( \chi
\right) p_{n,k}\left( \varphi \right) .
\end{equation*}

If $u\in \mathcal{S}\left( \mathbb{R}^{n}\right) $, $\chi \in \mathcal{C}%
_{0}^{\infty }\left( \mathbb{R}^{n}\right) \smallsetminus 0$ and $\widetilde{%
\chi }\in \mathcal{C}_{0}^{\infty }\left( \mathbb{R}^{n}\right) $, $%
\widetilde{\chi }=1$ on $\mathtt{supp}\chi $, then Proposition \ref{ks4} and
the above arguments imply that there is a continuous seminorm $p_{n,k}$ on $%
\mathcal{S}\left( \mathbb{R}^{n}\right) $ so that%
\begin{eqnarray*}
\left\Vert \left( \tau _{y}\chi \right) u\right\Vert _{\mathcal{B}_{k}}
&\leq &C\sup_{\left\vert \alpha +\beta \right\vert \leq m_{k}}\left\vert
\left( \partial ^{\alpha }u\right) \left( \tau _{y}\partial ^{\beta }\chi
\right) \right\vert \left\Vert \tau _{y}\widetilde{\chi }\right\Vert _{%
\mathcal{B}_{k}} \\
&=&C\sup_{\left\vert \alpha +\beta \right\vert \leq m_{k}}\left\vert \left(
\partial ^{\alpha }u\right) \left( \tau _{y}\partial ^{\beta }\chi \right)
\right\vert \left\Vert \widetilde{\chi }\right\Vert _{\mathcal{B}_{k}} \\
&\leq &2^{n+1}Cp_{n,k}\left( u\right) p_{n,k}\left( \chi \right)
\left\langle y\right\rangle ^{-\left( n+1\right) }\left\Vert \widetilde{\chi 
}\right\Vert _{\mathcal{B}_{k}}
\end{eqnarray*}%
Hence $u\in \mathcal{B}_{k}^{1}\left( \mathbb{R}^{n}\right) $ and%
\begin{equation*}
\left\Vert u\right\Vert _{k,1,\chi }\leq Cst\left\Vert \left\langle \cdot
\right\rangle ^{-\left( n+1\right) }\right\Vert _{L^{1}}\left\Vert 
\widetilde{\chi }\right\Vert _{\mathcal{B}_{k}}p_{n,k}\left( u\right)
p_{n,k}\left( \chi \right) .
\end{equation*}

$\left( \mathtt{e}\right) $ is trivial.

$\left( \mathtt{f}\right) $ Let $\chi \in \mathcal{C}_{0}^{\infty }\left( 
\mathbb{R}^{n}\right) \smallsetminus 0$. Then a consequence of Holder's
inequality and (\ref{kh13}) is that there is $C=C\left( n,k\right) $ and a
continuous seminorm $p=p_{n,k}$ on $\mathcal{S}\left( \mathbb{R}^{n}\right) $
such that 
\begin{equation*}
\left\Vert u\tau _{z}\chi \right\Vert _{\mathcal{B}_{k}}\leq Cp_{n,k}\left(
\chi \right) ^{2}\max \left\{ \left\Vert \left\langle \cdot \right\rangle
^{-\left( n+1\right) }\right\Vert _{L^{1}},1\right\} \left\Vert u\right\Vert
_{k,p,\chi },\quad z\in \mathbb{R}^{n}.
\end{equation*}%
Let $\left\{ u_{n}\right\} $ be a Cauchy sequence in $\mathcal{B}%
_{k}^{p}\left( \mathbb{R}^{n}\right) $. Then for any $z\in \mathbb{R}^{n}$
there is $u_{z}\in \mathcal{B}_{k}\left( \mathbb{R}^{n}\right) $ such that $%
u_{n}\tau _{z}\chi \rightarrow u_{z}$ in $\mathcal{B}_{k}\left( \mathbb{R}%
^{n}\right) $. Since $\mathcal{B}_{k}^{p}\left( \mathbb{R}^{n}\right)
\subset \mathcal{S}^{\prime }\left( \mathbb{R}^{n}\right) $ and $\mathcal{S}%
^{\prime }\left( \mathbb{R}^{n}\right) $ is sequentially complete, there is $%
u\in \mathcal{S}^{\prime }\left( \mathbb{R}^{n}\right) $ such that $%
u_{n}\rightarrow u$ in $\mathcal{S}^{\prime }\left( \mathbb{R}^{n}\right) $
and this implies that $u_{z}=u\tau _{z}\chi $ for any $z\in \mathbb{R}^{n}$.
Then $\left\Vert u_{n}-u\right\Vert _{k,p,\chi }\rightarrow 0$ by Fatou's
lemma.

$\left( \mathtt{g}\right) $ If $1/k\in L^{2}$, then $\mathcal{B}_{k}^{\infty
}\left( \mathbb{R}^{n}\right) \subset \mathcal{B}\mathcal{C}\left( \mathbb{R}%
^{n}\right) $. Let $\chi \in \mathcal{C}_{0}^{\infty }\left( \mathbb{R}%
^{n}\right) $ be such that $\chi \left( 0\right) =1$. Then for $x\in \mathbb{%
R}^{n}$ 
\begin{eqnarray*}
\left\vert u\left( x\right) \right\vert &=&\left\vert u\tau _{x}\chi \left(
x\right) \right\vert \leq \left( 2\pi \right) ^{-n}\left\Vert \widehat{u\tau
_{x}\chi }\right\Vert _{L^{1}} \\
&\leq &\left( 2\pi \right) ^{-n}\left\Vert 1/k\right\Vert _{L^{2}}\left\Vert
u\tau _{x}\chi u\right\Vert _{\mathcal{B}_{k}} \\
&=&\left( 2\pi \right) ^{-n}\left\Vert 1/k\right\Vert _{L^{2}}\left\Vert
u\right\Vert _{k,\infty ,\chi }.
\end{eqnarray*}
\end{proof}

\begin{remark}
The spaces $\mathcal{B}_{k}^{p}\left( \mathbb{R}^{n}\right) $ are particular
cases of Wiener amalgam spaces. More precisely we have%
\begin{equation*}
\mathcal{B}_{k}^{p}\left( \mathbb{R}^{n}\right) =W\left( \mathcal{B}%
_{k},L^{p}\right)
\end{equation*}%
with local component $\mathcal{B}_{k}\left( \mathbb{R}^{n}\right) $ and
global component $L^{p}\left( \mathbb{R}^{n}\right) $. Wiener amalgam spaces
were introduced by Hans Georg Feichtinger in 1980.
\end{remark}

Now using the techniques of Coifman and Meyer, developed for the study of
Beurling algebras $A_{\omega }$ and $B_{\omega }$ (see \cite{Meyer} pp
7-10), we shall prove an interesting result.

\begin{theorem}[localization principle]
$\mathcal{B}_{k}\left( \mathbb{R}^{n}\right) =\mathcal{B}_{k}^{2}\left( 
\mathbb{R}^{n}\right) =W\left( \mathcal{B}_{k},L^{2}\right) $.
\end{theorem}

To prove the result, we shall use the partition of unity built in the
previous section.

\begin{lemma}
$\mathcal{B}_{k}^{2}\left( \mathbb{R}^{n}\right) \subset \mathcal{B}%
_{k}\left( \mathbb{R}^{n}\right) $.
\end{lemma}

\begin{proof}
Let $\left\Vert \cdot \right\Vert _{k,2}$ is a fixed norm on $\mathcal{B}%
_{k}^{2}\left( \mathbb{R}^{n}\right) $. Let $u\in \mathcal{B}_{k}^{2}\left( 
\mathbb{R}^{n}\right) $. We have $u=\sum_{j=1}^{m}\chi _{j}u$ with $\chi
_{j}u=\sum_{\gamma \in \mathbb{Z}^{n}}\left( \tau _{\gamma }h_{j}\right) u$.
Since $u\in \mathcal{B}_{k}^{2}\left( \mathbb{R}^{n}\right) $, Proposition %
\ref{ks8} implies that 
\begin{equation*}
\sum_{\gamma \in \mathbb{Z}^{n}}\left\Vert \left( \tau _{\gamma
}h_{j}\right) u\right\Vert _{\mathcal{B}_{k}}^{2}\leq C_{j}^{2}\left\Vert
u\right\Vert _{k,2}^{2}<\infty
\end{equation*}%
Using Lemma \ref{ks2} it follows that $\chi _{j}u\in \mathcal{B}_{k}\left( 
\mathbb{R}^{n}\right) $ and 
\begin{equation*}
\left\Vert \chi _{j}u\right\Vert _{\mathcal{B}_{k}}^{2}\approx \sum_{\gamma
\in \mathbb{Z}^{n}}\left\Vert \left( \tau _{\gamma }h_{j}\right)
u\right\Vert _{\mathcal{B}_{k}}^{2}\leq C_{j}^{2}\left\Vert u\right\Vert
_{k,2}^{2}.
\end{equation*}
This proves that $u=\sum_{j=1}^{m}\chi _{j}u\in \mathcal{B}_{k}\left( 
\mathbb{R}^{n}\right) $ and that%
\begin{equation*}
\left\Vert u\right\Vert _{\mathcal{B}_{k}}\leq \sum_{j=1}^{m}\left\Vert \chi
_{j}u\right\Vert _{\mathcal{B}_{k}}\leq \left( \sum_{j=1}^{m}C_{j}\right)
\left\Vert u\right\Vert _{k,2}.
\end{equation*}
\end{proof}

\begin{lemma}
$\mathcal{B}_{k}\left( \mathbb{R}^{n}\right) \subset \mathcal{B}%
_{k}^{2}\left( \mathbb{R}^{n}\right) $.
\end{lemma}

\begin{proof}
Then the following statements are equivalent:

\texttt{(i)} $u\in \mathcal{B}_{k}\left( \mathbb{R}^{n}\right) $.

\texttt{(ii)} $\chi _{j}u\in \mathcal{B}_{k}\left( \mathbb{R}^{n}\right) $, $%
j=1,...,m$. (Here we use Lemma \ref{ks3} $\left( \mathtt{b}\right) $)

\texttt{(iii)} $\left\{ \left\Vert \left( \tau _{\gamma }h_{j}\right)
u\right\Vert _{\mathcal{B}_{k}}\right\} _{\gamma \in \mathbb{Z}^{n}}\in
l^{2}\left( \mathbb{Z}^{n}\right) $, $j=1,...,m$. (Here we use Lemma \ref%
{ks2})

Now $h=\sum_{j=1}^{m}h_{j}$ and $\left\Vert \left( \tau _{\gamma }h\right)
u\right\Vert _{\mathcal{B}_{k}}\leq \sum_{j=1}^{m}\left\Vert \left( \tau
_{\gamma }h_{j}\right) u\right\Vert _{\mathcal{B}_{k}}$, $\gamma \in \mathbb{%
Z}^{n}$ imply that $\left\{ \left\Vert \left( \tau _{\gamma }h\right)
u\right\Vert _{\mathcal{B}_{k}}\right\} _{\gamma \in \mathbb{Z}^{n}}\in
l^{2}\left( \mathbb{Z}^{n}\right) $. Since $h\in \mathcal{C}_{0}^{\infty
}\left( \mathbb{R}^{n}\right) $, $h\geq 0$ and $\sum_{\gamma \in \mathbb{Z}%
^{n}}\tau _{\gamma }h=1$ it follows that $u\in \mathcal{B}_{k}^{2}\left( 
\mathbb{R}^{n}\right) $ and 
\begin{eqnarray*}
\left\Vert u\right\Vert _{k,2,h} &\approx &\left\Vert \left\{ \left\Vert
\left( \tau _{\gamma }h\right) u\right\Vert _{\mathcal{B}_{k}}\right\}
_{\gamma \in \mathbb{Z}^{n}}\right\Vert _{l^{2}\left( \mathbb{Z}^{n}\right)
}\leq \sum_{j=1}^{m}\left\Vert \left\{ \left\Vert \left( \tau _{\gamma
}h_{j}\right) u\right\Vert _{\mathcal{B}_{k}}\right\} _{\gamma \in \mathbb{Z}%
^{n}}\right\Vert _{l^{2}\left( \mathbb{Z}^{n}\right) } \\
&\approx &\sum_{j=1}^{m}\left\Vert \chi _{j}u\right\Vert _{\mathcal{B}%
_{k}}\leq Cst\left\Vert u\right\Vert _{\mathcal{B}_{k}}.
\end{eqnarray*}
\end{proof}

\begin{lemma}
If $1\leq p<\infty $, then $\mathcal{S}\left( \mathbb{R}^{n}\right) $ is
dense in $\mathcal{B}_{k}^{p}\left( \mathbb{R}^{n}\right) $.
\end{lemma}

\begin{proof}
Let $\chi \in \mathcal{C}_{0}^{\infty }\left( \mathbb{R}^{n}\right) $ be
such that $\Psi =\Psi _{\mathbb{Z}^{n},\chi }=\sum_{\gamma \in \mathbb{Z}%
^{n}}\left\vert \tau _{\gamma }\chi \right\vert ^{2}>0$ and let $\left\Vert
\cdot \right\Vert _{k,p,\mathbb{Z}^{n},\chi }$ be the associated norm in $%
\mathcal{B}_{k}^{p}\left( \mathbb{R}^{n}\right) $.

\textit{Step 1.} Let $\psi \in \mathcal{C}_{0}^{\infty }\left( \mathbb{R}%
^{n}\right) $ be such that $\psi =1$ on $B\left( 0,1\right) $, $\psi
^{\varepsilon }\left( x\right) =\psi \left( \varepsilon x\right) $, $%
0<\varepsilon \leq 1$, $x\in \mathbb{R}^{n}$. If $u\in \mathcal{B}_{k}\left( 
\mathbb{R}^{n}\right) $, then $\psi ^{\varepsilon }u\rightarrow u$ in $%
\mathcal{B}_{k}\left( \mathbb{R}^{n}\right) $. Moreover we have 
\begin{equation*}
\left\Vert \psi ^{\varepsilon }u\right\Vert _{\mathcal{B}_{k}}\leq C\left(
k,\psi \right) \left\Vert u\right\Vert _{\mathcal{B}_{k}},\quad
0<\varepsilon \leq 1,
\end{equation*}
where $C\left( k,\psi \right) =\left( 2\pi \right) ^{-n}C\left( \int
\left\langle \eta \right\rangle ^{N}\left\vert \widehat{\psi }\left( \eta
\right) \right\vert \mathtt{d}\eta \right) $.

\textit{Step 2.} Suppose that $u\in \mathcal{B}_{k}^{p}\left( \mathbb{R}%
^{n}\right) $. Let $F\subset \mathbb{Z}^{n}$ be an arbitrary finite subset.
Then the subadditivity property of the norm $\left\Vert \cdot \right\Vert
_{l^{p}}$ implies that:%
\begin{multline*}
\left\Vert \psi ^{\varepsilon }u-u\right\Vert _{k,p,\mathbb{Z}^{n},\chi
}\leq \left( \sum_{\gamma \in F}\left\Vert \psi ^{\varepsilon }u\tau
_{\gamma }\chi -u\tau _{\gamma }\chi \right\Vert _{\mathcal{B}%
_{k}}^{p}\right) ^{\frac{1}{p}}+\left( \sum_{\gamma \in \mathbb{Z}%
^{n}\smallsetminus F}\left\Vert \psi ^{\varepsilon }u\tau _{\gamma }\chi
\right\Vert _{\mathcal{B}_{k}}^{p}\right) ^{\frac{1}{p}} \\
+\left( \sum_{\gamma \in \mathbb{Z}^{n}\smallsetminus F}\left\Vert u\tau
_{\gamma }\chi \right\Vert _{\mathcal{B}_{k}}^{p}\right) ^{\frac{1}{p}} \\
\leq \left( \sum_{\gamma \in F}\left\Vert \psi ^{\varepsilon }u\tau _{\gamma
}\chi -u\tau _{\gamma }\chi \right\Vert _{\mathcal{B}_{k}}^{p}\right) ^{%
\frac{1}{p}}+\left( C\left( k,\psi \right) +1\right) \left( \sum_{\gamma \in 
\mathbb{Z}^{n}\smallsetminus F}\left\Vert u\tau _{\gamma }\chi \right\Vert _{%
\mathcal{B}_{k}}^{p}\right) ^{\frac{1}{p}}.
\end{multline*}%
By making $\varepsilon \rightarrow 0$ we deduce that 
\begin{equation*}
\limsup_{\varepsilon \rightarrow 0}\left\Vert \psi ^{\varepsilon
}u-u\right\Vert _{k,p,\mathbb{Z}^{n},\chi }\leq \left( C\left( k,\psi
\right) +1\right) \left( \sum_{\gamma \in \mathbb{Z}^{n}\smallsetminus
F}\left\Vert u\tau _{\gamma }\chi \right\Vert _{\mathcal{B}_{k}}^{p}\right)
^{\frac{1}{p}}
\end{equation*}%
for any $F\subset \mathbb{Z}^{n}$ finite subset. Hence $\lim_{\varepsilon
\rightarrow 0}\psi ^{\varepsilon }u=u$ in $\mathcal{B}_{k}^{p}\left( \mathbb{%
R}^{n}\right) $. The immediate consequence is that $\mathcal{E}^{\prime
}\left( \mathbb{R}^{n}\right) \cap \mathcal{B}_{k}^{p}\left( \mathbb{R}%
^{n}\right) $ is dense in $\mathcal{B}_{k}^{p}\left( \mathbb{R}^{n}\right) $.

\textit{Step 3.} Suppose now that $u\in \mathcal{E}^{\prime }\left( \mathbb{R%
}^{n}\right) \cap \mathcal{B}_{k}^{p}\left( \mathbb{R}^{n}\right) $. Let $%
\varphi \in \mathcal{C}_{0}^{\infty }\left( \mathbb{R}^{n}\right) $ be such
that \texttt{supp}$\varphi \subset B\left( 0;1\right) $, $\int \varphi
\left( x\right) \mathtt{d}x=1$. For $\varepsilon \in \left( 0,1\right] $, we
set $\varphi _{\varepsilon }=\varepsilon ^{-n}\varphi \left( \cdot
/\varepsilon \right) $. Let $K=$\texttt{supp}$u+\overline{B\left( 0;1\right) 
}$. Then there is a finite set $F=F_{K,\chi }\subset \mathbb{Z}^{n}$ such
that $\left( \tau _{\gamma }\chi \right) \left( \varphi _{\varepsilon }\ast
u-u\right) =0$ for any $\gamma \in \mathbb{Z}^{n}\smallsetminus F$. It
follows that 
\begin{eqnarray*}
\left\Vert \varphi _{\varepsilon }\ast u-u\right\Vert _{k,p,\mathbb{Z}%
^{n},\chi } &=&\left( \sum_{\gamma \in F}\left\Vert \left( \tau _{\gamma
}\chi \right) \left( \varphi _{\varepsilon }\ast u-u\right) \right\Vert _{%
\mathcal{B}_{k}}^{p}\right) ^{\frac{1}{p}} \\
&\approx &\left( \sum_{\gamma \in F}\left\Vert \left( \tau _{\gamma }\chi
\right) \left( \varphi _{\varepsilon }\ast u-u\right) \right\Vert _{\mathcal{%
B}_{k}}^{2}\right) ^{\frac{1}{2}} \\
&\approx &\left\Vert \varphi _{\varepsilon }\ast u-u\right\Vert _{\mathcal{B}%
_{k}}\rightarrow 0,\quad \text{\textit{as }}\varepsilon \rightarrow 0.
\end{eqnarray*}
\end{proof}

\begin{proposition}
Let $1\leq p\leq \infty $, $k\in \mathcal{K}_{pol}\left( \mathbb{R}%
^{n}\right) $ and $C$, $N$ the positive constants that define $k$. Let $%
m_{k}=\left[ N+\frac{n+1}{2}\right] +1$. Then

$\left( \mathtt{a}\right) $ $\mathcal{BC}^{m_{k}}\left( \mathbb{R}%
^{n}\right) \cdot \mathcal{B}_{k}^{p}\left( \mathbb{R}^{n}\right) \subset 
\mathcal{B}_{k}^{p}\left( \mathbb{R}^{n}\right) $. In particular $\mathcal{BC%
}^{\infty }\left( \mathbb{R}^{n}\right) \cdot \mathcal{B}_{k}^{p}\left( 
\mathbb{R}^{n}\right) \subset \mathcal{B}_{k}^{p}\left( \mathbb{R}%
^{n}\right) $.

$\left( \mathtt{b}\right) $ $\mathcal{BC}^{m_{k}}\left( \mathbb{R}%
^{n}\right) \subset \mathcal{B}_{k}^{\infty }\left( \mathbb{R}^{n}\right) $.
\end{proposition}

\begin{proof}
$\left( \mathtt{a}\right) $ Let $u\in \mathcal{B}_{k}^{p}\left( \mathbb{R}%
^{n}\right) $ and $\psi \in \mathcal{BC}^{m_{k}}\left( \mathbb{R}^{n}\right) 
$. Let $\chi \in \mathcal{C}_{0}^{\infty }\left( \mathbb{R}^{n}\right)
\smallsetminus 0$. By using Proposition \ref{ks4} we obtain that $\psi u\tau
_{y}\chi \in \mathcal{B}_{k}\left( \mathbb{R}^{n}\right) $ and%
\begin{equation*}
\left\Vert \psi u\tau _{y}\chi \right\Vert _{\mathcal{B}_{k}}\leq
C_{k}\left\Vert \psi \right\Vert _{\mathcal{BC}^{m_{k}}}\left\Vert u\tau
_{y}\chi \right\Vert _{\mathcal{\mathcal{\mathcal{B}}}_{k}}.
\end{equation*}%
This inequality implies that 
\begin{equation*}
\left\Vert \psi u\right\Vert _{k,p,\chi }\leq C_{k}\left\Vert \psi
\right\Vert _{\mathcal{BC}^{m_{k}}}\left\Vert u\right\Vert _{k,p,\chi }
\end{equation*}%
$\left( \mathtt{b}\right) $ Since $1\in \mathcal{B}_{k}^{\infty }\left( 
\mathbb{R}^{n}\right) $ it follows that 
\begin{equation*}
\mathcal{BC}^{m_{k}}\left( \mathbb{R}^{n}\right) =\mathcal{BC}^{m_{k}}\left( 
\mathbb{R}^{n}\right) \cdot 1\subset \mathcal{BC}^{m_{k}}\left( \mathbb{R}%
^{n}\right) \cdot \mathcal{B}_{k}^{\infty }\left( \mathbb{R}^{n}\right)
\subset \mathcal{B}_{k}^{\infty }\left( \mathbb{R}^{n}\right) .
\end{equation*}
\end{proof}

\section{Wiener-L\'{e}vy theorem for $\mathcal{B}_{k}^{\infty }$ algebras}

In order to have algebra structure on $\mathcal{B}_{k}^{\infty }$, the
weight function $k$ should verify an additional condition.

\begin{lemma}
\label{kh11}Let $k,k_{1},k_{2}\in \mathcal{K}_{pol}\left( \mathbb{R}%
^{n}\right) $. Suppose that there is $C>0$ such that 
\begin{equation*}
\frac{1}{k_{1}^{2}}\ast \frac{1}{k_{2}^{2}}\leq \frac{C^{2}}{k^{2}}.
\end{equation*}%
Then the bilinear map 
\begin{equation*}
\mathcal{S}\left( \mathbb{R}^{n}\right) \mathcal{\mathbf{\times }S}\left( 
\mathbb{R}^{n}\right) \mathcal{\ni }\left( u_{1},u_{2}\right) \rightarrow
u_{1}u_{2}\in \mathcal{S}\left( \mathbb{R}^{n}\right)
\end{equation*}%
has a bounded extension%
\begin{eqnarray*}
\mathcal{\mathcal{B}}_{k_{1}}\left( \mathbb{R}^{n}\right) \times \mathcal{%
\mathcal{\mathcal{B}}}_{k_{2}}\left( \mathbb{R}^{n}\right) &\mathcal{\ni }%
&\left( u_{1},u_{2}\right) \rightarrow u_{1}u_{2}\in \mathcal{B}_{k}\left( 
\mathbb{R}^{n}\right) , \\
\left\Vert u_{1}u_{2}\right\Vert _{\mathcal{B}_{k}} &\leq &\left( 2\pi
\right) ^{-n}C\left\Vert u_{1}\right\Vert _{\mathcal{\mathcal{B}}%
_{k_{1}}}\left\Vert u_{2}\right\Vert _{\mathcal{\mathcal{\mathcal{B}}}%
_{k_{2}}}.
\end{eqnarray*}
\end{lemma}

\begin{proof}
Let $\left( u_{1},u_{2}\right) \in \mathcal{S}\left( \mathbb{R}^{n}\right)
\times \mathcal{S}\left( \mathbb{R}^{n}\right) $. Then 
\begin{eqnarray*}
\left\Vert u_{1}u_{2}\right\Vert _{\mathcal{B}_{k}}^{2} &=&\left\Vert k%
\widehat{u_{1}u_{2}}\right\Vert _{L^{2}}^{2}=\left( 2\pi \right)
^{-2n}\left\Vert k\left( \widehat{u_{1}}\ast \widehat{u_{2}}\right)
\right\Vert _{L^{2}}^{2} \\
&=&\left( 2\pi \right) ^{-2n}\int \left\vert k\left( \xi \right) \left( 
\widehat{u_{1}}\ast \widehat{u_{2}}\right) \left( \xi \right) \right\vert
^{2}\mathtt{d}\xi .
\end{eqnarray*}%
By using Schwarz's inequality, we can estimate the integrand as follows%
\begin{align*}
\left\vert k\left( \xi \right) \left( \widehat{u_{1}}\ast \widehat{u_{2}}%
\right) \left( \xi \right) \right\vert ^{2}& \leq \left( \int \left\vert
k_{1}\left( \eta \right) \widehat{u_{1}}\left( \eta \right) \right\vert
\left\vert k_{2}\left( \xi -\eta \right) \widehat{u_{2}}\left( \xi -\eta
\right) \right\vert \frac{k\left( \xi \right) }{k_{1}\left( \eta \right)
k_{2}\left( \xi -\eta \right) }\mathtt{d}\eta \right) ^{2} \\
& \leq \left( \int \left\vert k_{1}\left( \eta \right) \widehat{u_{1}}\left(
\eta \right) \right\vert ^{2}\left\vert k_{2}\left( \xi -\eta \right) 
\widehat{u_{2}}\left( \xi -\eta \right) \right\vert ^{2}\mathtt{d}\eta
\right) \\
& \qquad \qquad \qquad \qquad \cdot \left( \int \frac{k^{2}\left( \xi
\right) }{k_{1}^{2}\left( \eta \right) k_{2}^{2}\left( \xi -\eta \right) }%
\mathtt{d}\eta \right) \\
& \leq C^{2}\left( \int \left\vert k_{1}\left( \eta \right) \widehat{u_{1}}%
\left( \eta \right) \right\vert ^{2}\left\vert k_{2}\left( \xi -\eta \right) 
\widehat{u_{2}}\left( \xi -\eta \right) \right\vert ^{2}\mathtt{d}\eta
\right) .
\end{align*}%
Hence 
\begin{eqnarray*}
\left\Vert k\left( \widehat{u_{1}}\ast \widehat{u_{2}}\right) \right\Vert
_{L^{2}}^{2} &\leq &C^{2}\int \left( \int \left\vert k_{1}\left( \eta
\right) \widehat{u_{1}}\left( \eta \right) \right\vert ^{2}\left\vert
k_{2}\left( \xi -\eta \right) \widehat{u_{2}}\left( \xi -\eta \right)
\right\vert ^{2}\mathtt{d}\eta \right) \mathtt{d}\xi \\
&=&C^{2}\left\Vert u_{1}\right\Vert _{\mathcal{\mathcal{B}}%
_{k_{1}}}^{2}\left\Vert u_{2}\right\Vert _{\mathcal{\mathcal{\mathcal{B}}}%
_{k_{2}}}^{2}
\end{eqnarray*}%
and%
\begin{equation*}
\left\Vert u_{1}u_{2}\right\Vert _{\mathcal{B}_{k}}^{2}\leq \left( 2\pi
\right) ^{-2n}C^{2}\left\Vert u_{1}\right\Vert _{\mathcal{\mathcal{B}}%
_{k_{1}}}^{2}\left\Vert u_{2}\right\Vert _{\mathcal{\mathcal{\mathcal{B}}}%
_{k_{2}}}^{2}.
\end{equation*}
\end{proof}

\begin{corollary}
Let $k,k_{1},k_{2}\in \mathcal{K}_{pol}\left( \mathbb{R}^{n}\right) $.
Suppose that there is $C>0$ such that 
\begin{equation*}
\frac{1}{k_{1}^{2}}\ast \frac{1}{k_{2}^{2}}\leq \frac{C^{2}}{k^{2}}.
\end{equation*}%
If $\frac{1}{p_{1}}+\frac{1}{p_{2}}=\frac{1}{p}$, then%
\begin{equation*}
\mathcal{\mathcal{B}}_{k_{1}}^{p_{1}}\left( \mathbb{R}^{n}\right) \cdot 
\mathcal{\mathcal{B}}_{k_{2}}^{p_{2}}\left( \mathbb{R}^{n}\right) \subset 
\mathcal{B}_{k}^{p}\left( \mathbb{R}^{n}\right) .
\end{equation*}
\end{corollary}

\begin{proof}
Let $\chi \in \mathcal{C}_{0}^{\infty }\left( \mathbb{R}^{n}\right)
\smallsetminus 0$, $u_{1}\in \mathcal{\mathcal{B}}_{k_{1}}^{p_{1}}\left( 
\mathbb{R}^{n}\right) $ and $u_{2}\in \mathcal{\mathcal{B}}%
_{k_{2}}^{p_{2}}\left( \mathbb{R}^{n}\right) $. By using the previous lemma
we obtain that $u_{1}u_{2}\tau _{y}\chi ^{2}\in \mathcal{B}_{k}\left( 
\mathbb{R}^{n}\right) $ and%
\begin{equation*}
\left\Vert u_{1}u_{2}\tau _{y}\chi ^{2}\right\Vert _{\mathcal{B}_{k}}\leq
\left( 2\pi \right) ^{-n}C\left\Vert u_{1}\tau _{y}\chi \right\Vert _{%
\mathcal{\mathcal{B}}_{k_{1}}}\left\Vert u_{2}\tau _{y}\chi \right\Vert _{%
\mathcal{\mathcal{\mathcal{B}}}_{k_{2}}}.
\end{equation*}%
Finally, H\"{o}lder's inequality implies that 
\begin{equation*}
\left\Vert u_{1}u_{2}\right\Vert _{k,p,\chi ^{2}}\leq \left( 2\pi \right)
^{-n}C\left\Vert u_{1}\right\Vert _{k_{1},p_{1},\chi }\left\Vert
u_{2}\right\Vert _{s,p_{2},\chi }.
\end{equation*}
\end{proof}

\begin{corollary}
Let $k\in \mathcal{K}_{pol}\left( \mathbb{R}^{n}\right) $ and $1\leq p\leq
\infty $. Suppose that there is $C>0$ such that 
\begin{equation*}
\frac{1}{k^{2}}\ast \frac{1}{k^{2}}\leq \frac{C^{2}}{k^{2}}.
\end{equation*}%
Then $\mathcal{B}_{k}^{\infty }\left( \mathbb{R}^{n}\right) $ is a Banach
algebra with respect to the usual product and $\mathcal{B}_{k}^{p}\left( 
\mathbb{R}^{n}\right) $ is an ideal in $\mathcal{B}_{k}^{\infty }\left( 
\mathbb{R}^{n}\right) $.
\end{corollary}

\begin{definition}
The set of all weights $k\in \mathcal{K}_{pol}\left( \mathbb{R}^{n}\right) $
satisfying $\frac{1}{k^{2}}\ast \frac{1}{k^{2}}\leq \frac{C_{k}^{2}}{k^{2}}$
will be denoted by $\mathcal{K}_{a}\left( \mathbb{R}^{n}\right) $. Then for
any $k\in \mathcal{K}_{a}\left( \mathbb{R}^{n}\right) $, $\mathcal{B}%
_{k}^{\infty }\left( \mathbb{R}^{n}\right) $ is a Banach algebra with
respect to the usual product.
\end{definition}

\begin{lemma}
\label{kh8}Let $k\in \mathcal{K}_{pol}\left( \mathbb{R}^{n}\right) $. The
map 
\begin{equation*}
\mathcal{C}_{0}^{\infty }\left( \mathbb{R}^{n}\right) \times \mathcal{B}%
_{k}^{\infty }\left( \mathbb{R}^{n}\right) \ni \left( \varphi ,u\right)
\rightarrow \varphi \ast u\in \mathcal{B}_{k}^{\infty }\left( \mathbb{R}%
^{n}\right)
\end{equation*}%
is well defined and for any $\chi \in \mathcal{S}\left( \mathbb{R}%
^{n}\right) \smallsetminus 0$ we have the estimate%
\begin{equation*}
\left\Vert \varphi \ast u\right\Vert _{k,\infty ,\chi }\leq \left\Vert
\varphi \right\Vert _{L^{1}}\left\Vert u\right\Vert _{k,\infty ,\chi },\quad
\left( \varphi ,u\right) \in \mathcal{C}_{0}^{\infty }\left( \mathbb{R}%
^{n}\right) \times \mathcal{B}_{k}^{\infty }\left( \mathbb{R}^{n}\right) .
\end{equation*}
\end{lemma}

\begin{proof}
Let $\left( \varphi ,u\right) \in \mathcal{C}_{0}^{\infty }\left( \mathbb{R}%
^{n}\right) \times \mathcal{B}_{k}^{\infty }\left( \mathbb{R}^{n}\right) $, $%
\chi \in \mathcal{S}\left( \mathbb{R}^{n}\right) \smallsetminus 0$ and $\psi
\in \mathcal{S}\left( \mathbb{R}^{n}\right) $. Then using (\ref{ks9}) we
obtain%
\begin{eqnarray*}
\left\langle \tau _{z}\chi \left( \varphi \ast u\right) ,\psi \right\rangle
&=&\left\langle u,\check{\varphi}\ast \left( \left( \tau _{z}\chi \right)
\psi \right) \right\rangle =\int \check{\varphi}\left( y\right) \left\langle
u,\tau _{y}\left( \left( \tau _{z}\chi \right) \psi \right) \right\rangle 
\mathtt{d}y \\
&=&\int \varphi \left( y\right) \left\langle u,\tau _{-y}\left( \left( \tau
_{z}\chi \right) \psi \right) \right\rangle \mathtt{d}y=\int \varphi \left(
y\right) \left\langle \left( \tau _{z-y}\chi \right) u,\tau _{-y}\psi
\right\rangle \mathtt{d}y,
\end{eqnarray*}%
where $\check{\varphi}\left( y\right) =\varphi \left( -y\right) $. Since 
\begin{equation*}
\left\vert \left\langle \left( \tau _{z-y}\chi \right) u,\tau _{-y}\psi
\right\rangle \right\vert \leq \left( 2\pi \right) ^{-n}\left\Vert \left(
\tau _{z-y}\chi \right) u\right\Vert _{\mathcal{B}_{k}}\left\Vert \tau
_{-y}\psi \right\Vert _{\mathcal{B}_{1/\check{k}}}\leq \left( 2\pi \right)
^{-n}\left\Vert u\right\Vert _{k,\infty ,\chi }\left\Vert \psi \right\Vert _{%
\mathcal{B}_{1/\check{k}}}
\end{equation*}%
it follows that 
\begin{equation*}
\left\vert \left\langle \tau _{z}\chi \left( \varphi \ast u\right) ,\psi
\right\rangle \right\vert \leq \left( 2\pi \right) ^{-n}\left\Vert \varphi
\right\Vert _{L^{1}}\left\Vert u\right\Vert _{k,\infty ,\chi }\left\Vert
\psi \right\Vert _{\mathcal{B}_{1/\check{k}}}.
\end{equation*}%
Hence $\tau _{z}\chi \left( \varphi \ast u\right) \in \mathcal{B}_{k}\left( 
\mathbb{R}^{n}\right) $ and $\left\Vert \tau _{z}\chi \left( \varphi \ast
u\right) \right\Vert _{\mathcal{B}_{k}}\leq \left\Vert \varphi \right\Vert
_{L^{1}}\left\Vert u\right\Vert _{k,\infty ,\chi }$ for every $z\in \mathbb{R%
}^{n}$, i.e. $\varphi \ast u\in \mathcal{B}_{k}^{\infty }\left( \mathbb{R}%
^{n}\right) $ and%
\begin{equation*}
\left\Vert \varphi \ast u\right\Vert _{k,\infty ,\chi }\leq \left\Vert
\varphi \right\Vert _{L^{1}}\left\Vert u\right\Vert _{k,\infty ,\chi }.
\end{equation*}
\end{proof}

Let $\varphi \in \mathcal{C}_{0}^{\infty }\left( \mathbb{R}^{n}\right) $, $%
\varphi \geq 0$ be such that \texttt{supp}$\varphi \subset B\left(
0;1\right) $, $\int \varphi \left( x\right) \mathtt{d}x=1$. For $\varepsilon
\in \left( 0,1\right] $, we set $\varphi _{\varepsilon }=\varepsilon
^{-n}\varphi \left( \cdot /\varepsilon \right) $.

\begin{lemma}
If $k,k^{\prime }\in \mathcal{K}_{pol}\left( \mathbb{R}^{n}\right) $ and 
\begin{equation*}
\frac{k^{\prime }\left( \xi \right) }{k\left( \xi \right) }\rightarrow
0,\quad \xi \rightarrow \infty ,
\end{equation*}%
it follows that $\mathcal{B}_{k}^{\infty }\left( \mathbb{R}^{n}\right)
\subset \mathcal{B}_{k^{\prime }}^{\infty }\left( \mathbb{R}^{n}\right) $
and for any $u\in \mathcal{B}_{k}^{\infty }\left( \mathbb{R}^{n}\right) $ 
\begin{equation*}
\varphi _{\varepsilon }\ast u\rightarrow u\text{\quad }in\text{ }\mathcal{B}%
_{k^{\prime }}^{\infty }\left( \mathbb{R}^{n}\right) ,\quad \varepsilon
\rightarrow 0.
\end{equation*}
\end{lemma}

\begin{proof}
First let us note that the hypothesis and Lemma \ref{kh5} imply that $%
k^{\prime }/k$ is a bounded function and this implies that $\mathcal{B}%
_{k}^{\infty }\left( \mathbb{R}^{n}\right) \subset \mathcal{B}_{k^{\prime
}}^{\infty }\left( \mathbb{R}^{n}\right) $. Let $\chi ,\chi _{0}\in \mathcal{%
C}_{0}^{\infty }\left( \mathbb{R}^{n}\right) \smallsetminus 0$ be such that $%
\chi _{0}=1$ on \texttt{supp}$\chi +B\left( 0;1\right) $ and let $u\in 
\mathcal{B}_{k}^{\infty }\left( \mathbb{R}^{n}\right) $. Then for $%
0<\varepsilon \leq 1$ we have%
\begin{equation*}
\left( \varphi _{\varepsilon }\ast u-u\right) \tau _{y}\chi =\left( \varphi
_{\varepsilon }\ast \left( u\tau _{y}\chi _{0}\right) -u\tau _{y}\chi
_{0}\right) \tau _{y}\chi
\end{equation*}%
and by Proposition \ref{ks4}%
\begin{equation*}
\left\Vert \left( \varphi _{\varepsilon }\ast u-u\right) \tau _{y}\chi
\right\Vert _{\mathcal{B}_{k^{\prime }}}\leq C\left( k^{\prime },\chi
\right) \left\Vert \varphi _{\varepsilon }\ast \left( u\tau _{y}\chi
_{0}\right) -u\tau _{y}\chi _{0}\right\Vert _{\mathcal{B}_{k^{\prime }}}.
\end{equation*}%
We have 
\begin{eqnarray*}
\left\Vert \varphi _{\varepsilon }\ast \left( u\tau _{y}\chi _{0}\right)
-u\tau _{y}\chi _{0}\right\Vert _{\mathcal{B}_{k^{\prime }}}^{2}
&=&\left\Vert k^{\prime }\mathcal{F}\left( \varphi _{\varepsilon }\ast
\left( u\tau _{y}\chi _{0}\right) -u\tau _{y}\chi _{0}\right) \right\Vert
_{L^{2}}^{2} \\
&=&\int \left\vert \widehat{\varphi }\left( \varepsilon \xi \right)
-1\right\vert ^{2}\left( \frac{k^{\prime }\left( \xi \right) }{k\left( \xi
\right) }\right) ^{2}\left\vert k\left( \xi \right) \widehat{u\tau _{y}\chi
_{0}}\left( \xi \right) \right\vert ^{2}\mathtt{d}\xi .
\end{eqnarray*}%
Given any $\delta >0$ we now choose a ball $S=S_{\delta }$ so large that 
\begin{equation*}
\left\vert \widehat{\varphi }\left( \varepsilon \xi \right) -1\right\vert 
\frac{k^{\prime }\left( \xi \right) }{k\left( \xi \right) }\leq 2\frac{%
k^{\prime }\left( \xi \right) }{k\left( \xi \right) }<\delta ,\quad \xi \in 
\mathbb{R}^{n}\smallsetminus S,\text{ }\varepsilon \in \left( 0,1\right] .
\end{equation*}%
So for $S=S_{\delta }$ we can choose $\varepsilon _{\delta }$ so small that $%
\varepsilon \in \left( 0,\varepsilon _{\delta }\right] $ implies%
\begin{equation*}
\sup_{\xi \in S}\left\vert \widehat{\varphi }\left( \varepsilon \xi \right)
-1\right\vert \left( \frac{k^{\prime }\left( \xi \right) }{k\left( \xi
\right) }\right) <\delta .
\end{equation*}%
By writing $\int =\int_{S}+\int_{\mathbb{R}^{n}\smallsetminus S}$ we obtain 
\begin{multline*}
\left\Vert \varphi _{\varepsilon }\ast \left( u\tau _{y}\chi _{0}\right)
-u\tau _{y}\chi _{0}\right\Vert _{\mathcal{B}_{k^{\prime }}}^{2}\leq \delta
^{2}\int_{S}\left\vert k\left( \xi \right) \widehat{u\tau _{y}\chi _{0}}%
\left( \xi \right) \right\vert ^{2}\mathtt{d}\xi \\
+\delta ^{2}\int_{\mathbb{R}^{n}\smallsetminus S}\left\vert k\left( \xi
\right) \widehat{u\tau _{y}\chi _{0}}\left( \xi \right) \right\vert ^{2}%
\mathtt{d}\xi ,\quad \varepsilon \in \left( 0,\varepsilon _{\delta }\right] ,
\end{multline*}%
i.e.%
\begin{equation*}
\left\Vert \varphi _{\varepsilon }\ast \left( u\tau _{y}\chi _{0}\right)
-u\tau _{y}\chi _{0}\right\Vert _{\mathcal{B}_{k^{\prime }}}\leq \delta
\left\Vert k\widehat{u\tau _{y}\chi _{0}}\right\Vert _{L^{2}}=\delta
\left\Vert u\tau _{y}\chi _{0}\right\Vert _{\mathcal{B}_{k}},\quad
\varepsilon \in \left( 0,\varepsilon _{\delta }\right] .
\end{equation*}%
It follows that%
\begin{equation*}
\left\Vert \varphi _{\varepsilon }\ast u-u\right\Vert _{k^{\prime },\infty
,\chi }\leq \delta C\left( k^{\prime },\chi \right) \left\Vert u\right\Vert
_{k,\infty ,\chi _{0}}.
\end{equation*}%
The proof is complete.
\end{proof}

\begin{definition}
Let $k,k^{\prime }\in \mathcal{K}_{pol}\left( \mathbb{R}^{n}\right) $ be
such that 
\begin{equation*}
k^{\prime }\left( \xi \right) \leq Ck\left( \xi \right) ,\quad \xi \in 
\mathbb{R}^{n}.
\end{equation*}%
We set $\mathcal{B}_{k\left( k^{\prime }\right) }^{\infty }\left( \mathbb{R}%
^{n}\right) \equiv \left( \mathcal{B}_{k}^{\infty }\left( \mathbb{R}%
^{n}\right) ,\left\Vert \cdot \right\Vert _{k^{\prime },\infty }\right) .$
\end{definition}

\begin{corollary}
\label{ks10}Let $k,k^{\prime }\in \mathcal{K}_{pol}\left( \mathbb{R}%
^{n}\right) $ be such that 
\begin{equation*}
\frac{k^{\prime }\left( \xi \right) }{k\left( \xi \right) }\rightarrow
0,\quad \xi \rightarrow \infty .
\end{equation*}%
Then

$\left( \mathtt{a}\right) $ $\mathcal{B}_{k}^{\infty }\left( \mathbb{R}%
^{n}\right) \cap \mathcal{C}^{\infty }\left( \mathbb{R}^{n}\right) $ is
dense in $\mathcal{B}_{k\left( k^{\prime }\right) }^{\infty }\left( \mathbb{R%
}^{n}\right) $.

$\left( \mathtt{b}\right) $ If $1/k\in L^{2}\left( \mathbb{R}^{n}\right) $,
then $\mathcal{BC}^{\infty }\left( \mathbb{R}^{n}\right) $ is dense in $%
\mathcal{B}_{k\left( k^{\prime }\right) }^{\infty }\left( \mathbb{R}%
^{n}\right) $.
\end{corollary}

\begin{proof}
$\left( \mathtt{b}\right) $ If $1/k\in L^{2}$, then $\mathcal{B}_{k}^{\infty
}\left( \mathbb{R}^{n}\right) \subset \mathcal{BC}\left( \mathbb{R}%
^{n}\right) $. Therefore, $\varphi _{\varepsilon }\ast \mathcal{B}%
_{k}^{\infty }\left( \mathbb{R}^{n}\right) \subset \varphi _{\varepsilon
}\ast \mathcal{BC}\left( \mathbb{R}^{n}\right) \subset \mathcal{BC}^{\infty
}\left( \mathbb{R}^{n}\right) $.
\end{proof}

\begin{theorem}[Wiener-L\'{e}vy for $\mathcal{B}_{k}^{\infty }$, weak form]
\label{ks11}Let $\mathit{\Omega =\mathring{\Omega}}\subset \mathbb{C}^{d}$
and $\mathit{\Phi }:\mathit{\Omega }\rightarrow \mathbb{C}$ a holomorphic
function. Let $\left( k,k^{\prime }\right) \in \mathcal{K}_{pol}\left( 
\mathbb{R}^{n}\right) \times \mathcal{K}_{a}\left( \mathbb{R}^{n}\right) $
be such that 
\begin{equation*}
\frac{k^{\prime }\left( \xi \right) }{k\left( \xi \right) }\rightarrow
0,\quad \xi \rightarrow \infty
\end{equation*}%
and $1/k^{\prime }\in L^{2}\left( \mathbb{R}^{n}\right) $.

$\left( \mathtt{a}\right) $ If $u=\left( u_{1},...,u_{d}\right) \in \mathcal{%
B}_{k}^{\infty }\left( \mathbb{R}^{n}\right) ^{d}$ satisfies the condition $%
\overline{u\left( \mathbb{R}^{n}\right) }\subset \mathit{\Omega }$, then 
\begin{equation*}
\mathit{\Phi }\circ u\equiv \mathit{\Phi }\left( u\right) \in \mathcal{B}%
_{k^{\prime }}^{\infty }\left( \mathbb{R}^{n}\right) .
\end{equation*}

$\left( \mathtt{b}\right) $ If $u$, $u_{\varepsilon }\in \mathcal{B}%
_{k}^{\infty }\left( \mathbb{R}^{n}\right) ^{d}$, $0<\varepsilon \leq 1$, $%
\overline{u\left( \mathbb{R}^{n}\right) }\subset \mathit{\Omega }$ and $%
u_{\varepsilon }\rightarrow u$ in $\mathcal{B}_{k^{\prime }}^{\infty }\left( 
\mathbb{R}^{n}\right) ^{d}$ as $\varepsilon \rightarrow 0$, then there is $%
\varepsilon _{0}\in \left( 0,1\right] $ such that $\overline{u_{\varepsilon
}\left( \mathbb{R}^{n}\right) }\subset \mathit{\Omega }$ for every $%
0<\varepsilon \leq \varepsilon _{0}$ and $\mathit{\Phi }\left(
u_{\varepsilon }\right) \rightarrow \mathit{\Phi }\left( u\right) $ in $%
\mathcal{B}_{k^{\prime }}^{\infty }\left( \mathbb{R}^{n}\right) $ as $%
\varepsilon \rightarrow 0$.
\end{theorem}

\begin{proof}
On $\mathbb{C}^{d}$ we shall consider the distance given by the norm 
\begin{equation*}
\left\vert z\right\vert _{\infty }=\max \left\{ \left\vert z_{1}\right\vert
,...,\left\vert z_{d}\right\vert \right\} ,\quad z\in \mathbb{C}^{d}.
\end{equation*}%
Let $r=\mathtt{dist}\left( \overline{u\left( \mathbb{R}^{n}\right) },\mathbb{%
C}^{d}\smallsetminus \mathit{\Omega }\right) /8$. Since $\overline{u\left( 
\mathbb{R}^{n}\right) }\subset \mathit{\Omega }$ it follows that $r>0$ and 
\begin{equation*}
\bigcup_{y\in \overline{u\left( \mathbb{R}^{n}\right) }}\overline{B\left(
y;4r\right) }\subset \mathit{\Omega }.
\end{equation*}%
On $\mathcal{B}_{k^{\prime }}^{\infty }\left( \mathbb{R}^{n}\right) ^{d}$ we
shall consider the norm 
\begin{equation*}
\left\vert \left\vert \left\vert u\right\vert \right\vert \right\vert
_{k^{\prime },\infty }=\max \left\{ \left\Vert u_{1}\right\Vert _{k^{\prime
},\infty },...,\left\Vert u_{d}\right\Vert _{k^{\prime },\infty }\right\}
,\quad u\in \mathcal{B}_{k^{\prime }}^{\infty }\left( \mathbb{R}^{n}\right)
^{d},
\end{equation*}%
where $\left\Vert \cdot \right\Vert _{k^{\prime },\infty }$ is a fixed
Banach algebra norm on $\mathcal{B}_{k^{\prime }}^{\infty }\left( \mathbb{R}%
^{n}\right) $, and on $\mathcal{BC}\left( \mathbb{R}^{n}\right) ^{d}$ we
shall consider the norm 
\begin{equation*}
\left\vert \left\vert \left\vert u\right\vert \right\vert \right\vert
_{\infty }=\max \left\{ \left\Vert u_{1}\right\Vert _{\infty
},...,\left\Vert u_{d}\right\Vert _{\infty }\right\} ,\quad u\in \mathcal{BC}%
\left( \mathbb{R}^{n}\right) ^{d}.
\end{equation*}%
Since $\mathcal{B}_{k^{\prime }}^{\infty }\left( \mathbb{R}^{n}\right)
\subset \mathcal{BC}\left( \mathbb{R}^{n}\right) $ there is $C\geq 1$ so
that 
\begin{equation*}
\left\Vert \cdot \right\Vert _{\infty }\leq C\left\Vert \cdot \right\Vert
_{k^{\prime },\infty }.
\end{equation*}%
According to Corollary \ref{ks10} $\mathcal{BC}^{\infty }\left( \mathbb{R}%
^{n}\right) $ is dense in $\mathcal{B}_{k\left( k^{\prime }\right) }^{\infty
}\left( \mathbb{R}^{n}\right) $. Therefore we find $v=\left(
v_{1},...,v_{d}\right) \in \mathcal{BC}^{\infty }\left( \mathbb{R}%
^{n}\right) ^{d}$ so that $\left\vert \left\vert \left\vert u-v\right\vert
\right\vert \right\vert _{k^{\prime },\infty }<r/C$. Then 
\begin{equation*}
\left\vert \left\vert \left\vert u-v\right\vert \right\vert \right\vert
_{\infty }\leq C\left\vert \left\vert \left\vert u-v\right\vert \right\vert
\right\vert _{k^{\prime },\infty }<r.
\end{equation*}%
Using the last estimate we show that $\overline{v\left( \mathbb{R}%
^{n}\right) }\subset \bigcup_{x\in \mathbb{R}^{n}}B\left( u\left( x\right)
;r\right) $. Indeed, if $z\in \overline{v\left( \mathbb{R}^{n}\right) }$,
then there is $x\in \mathbb{R}^{n}$ such that 
\begin{equation*}
\left\vert z-v\left( x\right) \right\vert _{\infty }<r-\left\vert \left\vert
\left\vert v-u\right\vert \right\vert \right\vert _{\infty }.
\end{equation*}%
It follows that 
\begin{eqnarray*}
\left\vert z-u\left( x\right) \right\vert _{\infty } &\leq &\left\vert
z-v\left( x\right) \right\vert _{\infty }+\left\vert v\left( x\right)
-u\left( x\right) \right\vert _{\infty } \\
&\leq &\left\vert z-v\left( x\right) \right\vert _{\infty }+\left\vert
\left\vert \left\vert v-u\right\vert \right\vert \right\vert _{\infty } \\
&<&r-\left\vert \left\vert \left\vert v-u\right\vert \right\vert \right\vert
_{\infty }+\left\vert \left\vert \left\vert v-u\right\vert \right\vert
\right\vert _{\infty }=r
\end{eqnarray*}%
so $z\in B\left( u\left( x\right) ;r\right) $.

From $\overline{v\left( \mathbb{R}^{n}\right) }\subset \bigcup_{x\in \mathbb{%
R}^{n}}B\left( u\left( x\right) ;r\right) $ we get 
\begin{equation*}
\overline{v\left( \mathbb{R}^{n}\right) }+\overline{B\left( 0;3r\right) }%
\subset \bigcup_{x\in \mathbb{R}^{n}}B\left( u\left( x\right) ;4r\right)
\subset \mathit{\Omega },
\end{equation*}%
hence the map%
\begin{equation*}
\mathbb{R}^{n}\times \overline{B\left( 0;3r\right) }\ni \left( x,\zeta
\right) \rightarrow \mathit{\Phi }\left( v\left( x\right) +\zeta \right) \in 
\mathbb{C}
\end{equation*}%
is well defined. Let $\Gamma \left( r\right) $ denote the polydisc $\left(
\partial \mathbb{D}\left( 0,3r\right) \right) ^{d}$. Since $\overline{%
v\left( \mathbb{R}^{n}\right) }+\Gamma \left( r\right) \subset \mathit{%
\Omega }$ is a compact subset, the map 
\begin{equation*}
\Gamma \left( r\right) \ni \zeta \rightarrow \mathit{\Phi }\left( \zeta
+v\right) \in \mathcal{BC}^{m_{k^{\prime }}}\left( \mathbb{R}^{n}\right)
\subset \mathcal{B}_{k^{\prime }}^{\infty }\left( \mathbb{R}^{n}\right)
\end{equation*}%
is continuous. On the other hand we have 
\begin{equation*}
\left( \zeta _{1}+v_{1}-u_{1}\right) ^{-1},...,\left( \zeta
_{d}+v_{d}-u_{d}\right) ^{-1}\in \mathcal{B}_{k^{\prime }}^{\infty }\left( 
\mathbb{R}^{n}\right)
\end{equation*}%
because $\left\Vert u_{1}-v_{1}\right\Vert _{k^{\prime },\infty
},...,\left\Vert u_{d}-v_{d}\right\Vert _{k^{\prime },\infty }<r/C\leq r$
and $\left\vert \zeta _{1}\right\vert =...=\left\vert \zeta _{d}\right\vert
=3r$. It follows that the integral 
\begin{equation}
h=\frac{1}{\left( 2\pi \mathtt{i}\right) ^{d}}\int_{\Gamma \left( r\right) }%
\frac{\mathit{\Phi }\left( \zeta +v\right) }{\left( \zeta
_{1}+v_{1}-u_{1}\right) ...\left( \zeta _{d}+v_{d}-u_{d}\right) }\mathtt{d}%
\zeta  \label{ksc}
\end{equation}%
defines an element $h\in \mathcal{B}_{k^{\prime }}^{\infty }\left( \mathbb{R}%
^{n}\right) $.

Let 
\begin{equation*}
\delta _{x}:\mathcal{B}_{k^{\prime }}^{\infty }\left( \mathbb{R}^{n}\right)
\subset \mathcal{BC}\left( \mathbb{R}^{n}\right) \rightarrow \mathbb{C}%
,\quad w\rightarrow w\left( x\right) ,
\end{equation*}%
be the evaluation functional at $x\in \mathbb{R}^{n}$. Then 
\begin{eqnarray*}
h\left( x\right) &=&\frac{1}{\left( 2\pi \mathtt{i}\right) ^{d}}\int_{\Gamma
\left( r\right) }\frac{\mathit{\Phi }\left( \zeta +v\left( x\right) \right) 
}{\left( \zeta _{1}-\left( u_{1}\left( x\right) -v_{1}\left( x\right)
\right) \right) ...\left( \zeta _{d}-\left( u_{d}\left( x\right)
-v_{d}\left( x\right) \right) \right) }\mathtt{d}\zeta \\
&=&\mathit{\Phi }\left( \zeta +v\left( x\right) \right) |_{\zeta =u\left(
x\right) -v\left( x\right) }=\mathit{\Phi }\left( u\left( x\right) \right)
\end{eqnarray*}%
because $\left\vert u\left( x\right) -v\left( x\right) \right\vert _{\infty
}\leq \left\vert \left\vert \left\vert u-v\right\vert \right\vert
\right\vert _{\infty }<r$, so $u\left( x\right) -v\left( x\right) $ is
within polydisc $\Gamma \left( r\right) $. Hence $h=\mathit{\Phi }\circ
u\equiv \mathit{\Phi }\left( u\right) \in \mathcal{B}_{k^{\prime }}^{\infty
}\left( \mathbb{R}^{n}\right) $.

$\left( \mathtt{b}\right) $ Let $\varepsilon _{0}\in \left( 0,1\right] $ be
such that for any $0<\varepsilon \leq \varepsilon _{0}$ we have%
\begin{equation*}
\left\vert \left\vert \left\vert u-u_{\varepsilon }\right\vert \right\vert
\right\vert _{k^{\prime },\infty }<r/C.
\end{equation*}%
Then $\left\vert \left\vert \left\vert u-u_{\varepsilon }\right\vert
\right\vert \right\vert _{\infty }\leq C\left\vert \left\vert \left\vert
u-u_{\varepsilon }\right\vert \right\vert \right\vert _{k^{\prime },\infty
}<r$ and $\overline{u_{\varepsilon }\left( \mathbb{R}^{n}\right) }\subset
\bigcup_{x\in \mathbb{R}^{n}}B\left( u\left( x\right) ;r\right) \subset 
\mathit{\Omega }$ for every $0<\varepsilon \leq \varepsilon _{0}$. On the
other hand we have $\left\vert \left\vert \left\vert v-u_{\varepsilon
}\right\vert \right\vert \right\vert _{k^{\prime },\infty }\leq \left\vert
\left\vert \left\vert v-u\right\vert \right\vert \right\vert _{k^{\prime
},\infty }+\left\vert \left\vert \left\vert u-u_{\varepsilon }\right\vert
\right\vert \right\vert _{k^{\prime },\infty }<r/C+r/C\leq 2r$. It follows
that 
\begin{equation*}
\left( \zeta _{1}+v_{1}-u_{\varepsilon 1}\right) ^{-1},...,\left( \zeta
_{d}+v_{d}-u_{\varepsilon d}\right) ^{-1}\in \mathcal{B}_{k^{\prime
}}^{\infty }\left( \mathbb{R}^{n}\right)
\end{equation*}%
because $\left\Vert u_{\varepsilon 1}-v_{1}\right\Vert _{k^{\prime },\infty
},...,\left\Vert u_{\varepsilon d}-v_{d}\right\Vert _{k^{\prime },\infty
}<2r $ and $\left\vert \zeta _{1}\right\vert =...=\left\vert \zeta
_{d}\right\vert =3r$. We obtain that 
\begin{eqnarray*}
\mathit{\Phi }\left( u_{\varepsilon }\right) &=&\frac{1}{\left( 2\pi \mathtt{%
i}\right) ^{d}}\int_{\Gamma \left( r\right) }\frac{\mathit{\Phi }\left(
\zeta +v\right) }{\left( \zeta _{1}+v_{1}-u_{\varepsilon 1}\right) ...\left(
\zeta _{d}+v_{d}-u_{\varepsilon d}\right) }\mathtt{d}\zeta \\
&\rightarrow &\frac{1}{\left( 2\pi \mathtt{i}\right) ^{d}}\int_{\Gamma
\left( r\right) }\frac{\mathit{\Phi }\left( \zeta +v\right) }{\left( \zeta
_{1}+v_{1}-u_{1}\right) ...\left( \zeta _{d}+v_{d}-u_{d}\right) }\mathtt{d}%
\zeta =\mathit{\Phi }\left( u\right) ,\quad \text{\textit{as} }\varepsilon
\rightarrow 0.
\end{eqnarray*}
\end{proof}

\begin{remark}
According to Coquand and Stolzenberg \cite{Coquand}, this type of
representation formula, $($\ref{ksc}$)$, was introduced more than 60 years
ago by A. P. Calder\'{o}n.
\end{remark}

In order to reduce the number of weight functions occurring in Theorem \ref%
{ks11} from two to one, $k$ should verify some additional conditions. To
find them we need some auxiliary results that we shall present below.

For $k\in \mathcal{K}_{pol}\left( \mathbb{R}^{n}\right) $ we shall denote by 
$\underline{k}$ the weight function $\left\langle \cdot \right\rangle ^{-1}k$%
, i.e. $\underline{k}=\left\langle \cdot \right\rangle ^{-1}k\in \mathcal{K}%
_{pol}\left( \mathbb{R}^{n}\right) $. Since 
\begin{equation*}
k^{2}\left( \xi \right) =\left\langle \xi \right\rangle ^{2}\underline{k}%
^{2}\left( \xi \right) =\left( 1+\sum_{j=1}^{n}\xi _{j}^{2}\right) 
\underline{k}^{2}\left( \xi \right)
\end{equation*}%
we have 
\begin{equation*}
\left\Vert u\right\Vert _{\mathcal{B}_{k}}^{2}=\left\Vert u\right\Vert _{%
\mathcal{B}_{\underline{k}}}^{2}+\sum_{j=1}^{n}\left\Vert \partial
_{j}u\right\Vert _{\mathcal{B}_{\underline{k}}}^{2}
\end{equation*}%
hence $u\in \mathcal{B}_{k}\left( \mathbb{R}^{n}\right) $ if and only if $%
u,\partial _{1}u,...,\partial _{n}u\in \mathcal{B}_{\underline{k}}\left( 
\mathbb{R}^{n}\right) $.

The same is true for the spaces $\mathcal{B}_{k}^{p}\left( \mathbb{R}%
^{n}\right) $. Assume that $u\in \mathcal{B}_{k}^{p}\left( \mathbb{R}%
^{n}\right) $. Let $\chi ,\chi _{0}\in \mathcal{C}_{0}^{\infty }\left( 
\mathbb{R}^{n}\right) \smallsetminus 0$ be such that $\chi _{0}=1$ on a
neighborhood of \texttt{supp}$\chi $. Then 
\begin{equation*}
\left( \tau _{y}\chi \right) \partial _{j}u=\left( \tau _{y}\chi \right)
\partial _{j}\left( \left( \tau _{y}\chi _{0}\right) u\right)
\end{equation*}%
and%
\begin{eqnarray*}
\left\Vert \left( \tau _{y}\chi \right) \partial _{j}u\right\Vert _{\mathcal{%
B}_{\underline{k}}} &=&\left\Vert \left( \tau _{y}\chi \right) \partial
_{j}\left( \left( \tau _{y}\chi _{0}\right) u\right) \right\Vert _{\mathcal{B%
}_{\underline{k}}} \\
&\leq &Cst\left\Vert \chi \right\Vert _{\mathcal{BC}^{m_{k}+1}}\left\Vert
\partial _{j}\left( \left( \tau _{y}\chi _{0}\right) u\right) \right\Vert _{%
\mathcal{B}_{\underline{k}}} \\
&\leq &Cst\left\Vert \chi \right\Vert _{\mathcal{BC}^{m_{k}+1}}\left\Vert
\left( \tau _{y}\chi _{0}\right) u\right\Vert _{\mathcal{B}_{k}}
\end{eqnarray*}%
since $m_{\underline{k}}=m_{k}+1$. It follows that 
\begin{equation*}
\left\Vert \partial _{j}u\right\Vert _{\underline{k},p,\chi }\leq
Cst\left\Vert \chi \right\Vert _{\mathcal{BC}^{m_{k}+1}}\left\Vert
u\right\Vert _{k,p,\chi _{0}}.
\end{equation*}%
Suppose now that $u,\partial _{1}u,...,\partial _{n}u\in \mathcal{B}_{%
\underline{k}}^{p}\left( \mathbb{R}^{n}\right) $. Then 
\begin{eqnarray*}
\left\Vert \left( \tau _{y}\chi \right) u\right\Vert _{\mathcal{B}_{k}}^{2}
&=&\left\Vert \left( \tau _{y}\chi \right) u\right\Vert _{\mathcal{B}_{%
\underline{k}}}^{2}+\sum_{j=1}^{n}\left\Vert \partial _{j}\left( \left( \tau
_{y}\chi \right) u\right) \right\Vert _{\mathcal{B}_{\underline{k}}}^{2} \\
&=&\left\Vert \left( \tau _{y}\chi \right) u\right\Vert _{\mathcal{B}_{%
\underline{k}}}^{2}+\sum_{j=1}^{n}\left\Vert \left( \tau _{y}\chi \right)
\partial _{j}u+\left( \tau _{y}\partial _{j}\chi \right) u\right\Vert _{%
\mathcal{B}_{\underline{k}}}^{2} \\
&\leq &\left\Vert \left( \tau _{y}\chi \right) u\right\Vert _{\mathcal{B}_{%
\underline{k}}}^{2}+2\sum_{j=1}^{n}\left( \left\Vert \left( \tau _{y}\chi
\right) \partial _{j}u\right\Vert _{\mathcal{B}_{\underline{k}%
}}^{2}+\left\Vert \left( \tau _{y}\partial _{j}\chi \right) u\right\Vert _{%
\mathcal{B}_{\underline{k}}}^{2}\right)
\end{eqnarray*}%
and 
\begin{equation*}
\left\Vert \left( \tau _{y}\chi \right) u\right\Vert _{\mathcal{B}_{k}}\leq
\left\Vert \left( \tau _{y}\chi \right) u\right\Vert _{\mathcal{B}_{%
\underline{k}}}+\sqrt{2}\sum_{j=1}^{n}\left( \left\Vert \left( \tau _{y}\chi
\right) \partial _{j}u\right\Vert _{\mathcal{B}_{\underline{k}}}+\left\Vert
\left( \tau _{y}\partial _{j}\chi \right) u\right\Vert _{\mathcal{B}_{%
\underline{k}}}\right) .
\end{equation*}%
It follows that 
\begin{equation*}
\left\Vert u\right\Vert _{k,p,\chi }\leq \left\Vert u\right\Vert _{%
\underline{k},p,\chi }+\sqrt{2}\sum_{j=1}^{n}\left( \left\Vert \partial
_{j}u\right\Vert _{\underline{k},p,\chi }+\left\Vert u\right\Vert _{%
\underline{k},p,\partial _{j}\chi }\right) .
\end{equation*}

\noindent Thus we have proved the following result.

\begin{lemma}
\label{kh9}Let $k\in \mathcal{K}_{pol}\left( \mathbb{R}^{n}\right) $ and $%
\underline{k}=\left\langle \cdot \right\rangle ^{-1}k\in \mathcal{K}%
_{pol}\left( \mathbb{R}^{n}\right) $. Let $u\in \mathcal{D}^{\prime }\left( 
\mathbb{R}^{n}\right) $. Then $u\in \mathcal{B}_{k}^{p}\left( \mathbb{R}%
^{n}\right) $ if and only if $u,\partial _{1}u,...,\partial _{n}u\in 
\mathcal{B}_{\underline{k}}^{p}\left( \mathbb{R}^{n}\right) $. Moreover, if $%
\chi \in \mathcal{C}_{0}^{\infty }\left( \mathbb{R}^{n}\right)
\smallsetminus 0$, then there is $C\geq 1$ such that 
\begin{equation*}
C^{-1}\left\Vert u\right\Vert _{k,p,\chi }\leq \left\Vert u\right\Vert _{%
\underline{k},p,\chi }+\sum_{j=1}^{n}\left\Vert \partial _{j}u\right\Vert _{%
\underline{k},p,\chi }\leq C\left\Vert u\right\Vert _{k,p,\chi }.
\end{equation*}
\end{lemma}

\begin{lemma}
If $k\in \mathcal{K}_{pol}\left( \mathbb{R}^{n}\right) \cap L^{p}\left( 
\mathbb{R}^{n}\right) $, then $\lim_{\left\vert \xi \right\vert \rightarrow
\infty }k\left( \xi \right) =0$.
\end{lemma}

\begin{proof}
We shall prove the result by contradiction. Suppose the statement is not
true. Then one can find a number $\varepsilon _{0}>0$ such that for each $%
R>0 $ there is a vector $\xi _{R}\in \mathbb{R}^{n}$ such that $\left\vert
\xi _{R}\right\vert \geq R$ and $k\left( \xi _{R}\right) \geq \varepsilon
_{0}$. Then, by recurrence, one can find a sequence $\left\{ \xi
_{m}\right\} $ in $\mathbb{R}^{n}$ such that $k\left( \xi _{m}\right) \geq
\varepsilon _{0}$ and $\left\vert \xi _{m}\right\vert \geq \left\vert \xi
_{m-1}\right\vert +3$. We have 
\begin{equation*}
k\left( \xi +\xi _{m}\right) \geq k\left( \xi _{m}\right) C^{-1}\left\langle
\xi \right\rangle ^{-N}\geq \varepsilon _{0}C^{-1}\left\langle \xi
\right\rangle ^{-N}
\end{equation*}%
The inequality $\left\vert \xi _{m}\right\vert \geq \left\vert \xi
_{m-1}\right\vert +3$ implies that $\overline{B\left( \xi _{m};1\right) }%
\cap $ $\overline{B\left( \xi _{m-1};1\right) }=\emptyset $. It follows that 
\begin{eqnarray*}
\int_{\left\vert \xi -\xi _{m}\right\vert \leq 1}k^{p}\left( \xi \right) 
\mathtt{d}\xi &=&\int_{\left\vert \xi \right\vert \leq 1}k^{p}\left( \xi
+\xi _{m}\right) \mathtt{d}\xi \geq \left( \varepsilon _{0}/C\right)
^{p}\int_{\left\vert \xi \right\vert \leq 1}\left\langle \xi \right\rangle
^{-pN}\mathtt{d}\xi \\
&=&\left( \varepsilon _{0}/C\right) ^{p}\left\Vert \left\langle \xi
\right\rangle ^{-N}\right\Vert _{L^{p}\left( B\right) }^{p}.
\end{eqnarray*}%
Then for each $l=1,2,...$we have 
\begin{equation*}
\left\Vert k\right\Vert _{L^{p}}^{p}\geq \int_{\bigcup_{1}^{l}B\left( \xi
_{m};1\right) }k^{p}\left( \xi \right) \mathtt{d}\xi \geq l\left(
\varepsilon _{0}/C\right) ^{p}\left\Vert \left\langle \xi \right\rangle
^{-N}\right\Vert _{L^{p}\left( B\right) }^{p}.
\end{equation*}%
This contradicts the assumption that $k\in L^{p}\left( \mathbb{R}^{n}\right) 
$.
\end{proof}

\begin{corollary}
\label{kh10}If $k\in \mathcal{K}_{pol}\left( \mathbb{R}^{n}\right) $ and $%
k^{-1}\in L^{p}\left( \mathbb{R}^{n}\right) $, then $\lim_{\left\vert \xi
\right\vert \rightarrow \infty }k\left( \xi \right) =\infty $.
\end{corollary}

\begin{lemma}
\label{kh7}Let $k,k_{1},k_{2}\in \mathcal{K}_{pol}\left( \mathbb{R}%
^{n}\right) $ and $t_{0}<t_{1}$. Suppose that there are $C_{0},C_{1}>0$ such
that%
\begin{equation*}
k_{1}^{t_{j}}\ast k_{2}^{t_{j}}\leq C_{j}k^{t_{j}},\quad j=0,1.
\end{equation*}%
Then, for each $t\in \left[ t_{0},t_{1}\right] $, we have%
\begin{equation*}
k_{1}^{t}\ast k_{2}^{t}\leq C_{t}k^{t}
\end{equation*}%
where $C_{t}=C_{0}^{\frac{t_{1}-t}{t_{1}-t_{0}}},C_{1}^{\frac{t-t_{0}}{%
t_{1}-t_{0}}}\leq \max \left\{ C_{0},C_{1}\right\} $.
\end{lemma}

\begin{proof}
Let us note that 
\begin{equation*}
t=\frac{t_{1}-t}{t_{1}-t_{0}}t_{0}+\frac{t-t_{0}}{t_{1}-t_{0}}t_{1}=\frac{%
t_{0}}{p_{0}}+\frac{t_{1}}{p_{1}}
\end{equation*}%
where 
\begin{equation*}
p_{0}=\frac{t_{1}-t_{0}}{t_{1}-t},\quad p_{1}=\frac{t_{1}-t_{0}}{t-t_{0}}%
,\quad \frac{1}{p_{0}}+\frac{1}{p_{1}}=1.
\end{equation*}%
Now, by applying the Holder inequality one obtains the result%
\begin{eqnarray*}
k_{1}^{t}\ast k_{2}^{t} &=&\left( k_{1}^{\frac{t_{0}}{p_{0}}}\cdot k_{1}^{%
\frac{t_{1}}{p_{1}}}\right) \ast \left( k_{2}^{\frac{t_{0}}{p_{0}}}\cdot
k_{2}^{\frac{t_{1}}{p_{1}}}\right) \\
&\leq &\left( k_{1}^{t_{0}}\ast k_{2}^{t_{0}}\right) ^{\frac{1}{p_{0}}}\cdot
\left( k_{1}^{t_{1}}\ast k_{2}^{t_{1}}\right) ^{\frac{1}{p_{1}}} \\
&\leq &C_{0}^{\frac{1}{p_{0}}}C_{1}^{\frac{1}{p_{1}}}k^{\frac{t_{0}}{p_{0}}+%
\frac{t_{1}}{p_{1}}}=C_{t}k^{t}.
\end{eqnarray*}
\end{proof}

\begin{definition}
\label{kh6}Let $k\in \mathcal{K}_{a}\left( \mathbb{R}^{n}\right) $. We say
that $k$ satisfies the WL-condition if there is $\delta =\delta _{k}\in
\left( 0,1\right) $ such that $k^{\delta }\in \mathcal{K}_{a}\left( \mathbb{R%
}^{n}\right) $, $k^{-\delta }\in L^{2}\left( \mathbb{R}^{n}\right) $ and 
\begin{equation*}
\frac{1}{k^{2\delta }}\ast \frac{1}{\underline{k}^{2}}\leq \frac{C_{k,\delta
}^{2}}{\underline{k}^{2}},
\end{equation*}%
where $\underline{k}=\left\langle \cdot \right\rangle ^{-1}k\in \mathcal{K}%
_{pol}\left( \mathbb{R}^{n}\right) $. The set of all such functions $k$ will
be denoted by $\mathcal{K}_{WL}\left( \mathbb{R}^{n}\right) $.
\end{definition}

\begin{lemma}
$\left( \mathtt{a}\right) $ If $k\in \mathcal{K}_{pol}\left( \mathbb{R}%
^{n}\right) $ satisfies $\lim_{\left\vert \xi \right\vert \rightarrow \infty
}k\left( \xi \right) =\infty $, then $c_{k}=\inf_{\xi }k\left( \xi \right)
>0 $ and 
\begin{equation*}
t\geq s\Longrightarrow k^{-t}\leq c_{k}^{s-t}k^{-s}.
\end{equation*}

$\left( \mathtt{b}\right) $ If $k\in \mathcal{K}_{WL}\left( \mathbb{R}%
^{n}\right) $ and $\delta =\delta _{k}\in \left( 0,1\right) $ is as in
Definition \ref{kh6}, then for each $t\in \left[ \delta ,1\right] $ we have $%
k^{t}\in \mathcal{K}_{a}\left( \mathbb{R}^{n}\right) $, $k^{-t}\in
L^{2}\left( \mathbb{R}^{n}\right) $ and%
\begin{equation*}
\frac{1}{k^{2t}}\ast \frac{1}{\underline{k}^{2}}\leq c_{k}^{2\left( \delta
-t\right) }\frac{C_{k,\delta }^{2}}{\underline{k}^{2}}.
\end{equation*}

$\left( \mathtt{c}\right) $ If $\delta $ is sufficiently close to $1$, then
there is $C>0$ such that%
\begin{equation*}
\underline{k}=\left\langle \cdot \right\rangle ^{-1}k\leq Ck^{\delta }.
\end{equation*}
\end{lemma}

\begin{proof}
$\left( \mathtt{a}\right) $ The first part is an easy consequence of the
hypothesis and Lemma \ref{kh5}. Then 
\begin{equation*}
k\geq c_{k}\Longrightarrow k^{t-s}\geq c_{k}^{t-s}\Longrightarrow
c_{k}^{s-t}k^{-s}\geq k^{-t}.
\end{equation*}

$\left( \mathtt{b}\right) $ $k^{t}\in \mathcal{K}_{a}\left( \mathbb{R}%
^{n}\right) $ is an immediate consequence of Lemma \ref{kh7}, $k^{-t}\leq
c_{k}^{\delta -t}k^{-\delta }\in L^{2}\left( \mathbb{R}^{n}\right) $ implies
that $k^{-t}\in L^{2}\left( \mathbb{R}^{n}\right) $ and%
\begin{equation*}
\frac{1}{k^{2t}}\ast \frac{1}{\underline{k}^{2}}\leq c_{k}^{2\left( \delta
-t\right) }\frac{1}{k^{2\delta }}\ast \frac{1}{\underline{k}^{2}}\leq
c_{k}^{2\left( \delta -t\right) }\frac{C_{k,\delta }^{2}}{\underline{k}^{2}}.
\end{equation*}

$\left( \mathtt{c}\right) $ We know that 
\begin{equation*}
C^{-1}k\left( 0\right) \left\langle \xi \right\rangle ^{-N}\leq k\left( \xi
\right) \leq Ck\left( 0\right) \left\langle \xi \right\rangle ^{N},\quad \xi
\in \mathbb{R}^{n}.
\end{equation*}%
Now if we take $\delta >\frac{N-1}{N}$, then $N<\frac{1}{1-\delta }$ and 
\begin{equation*}
k\left( \xi \right) \leq Ck\left( 0\right) \left\langle \xi \right\rangle
^{N}\leq Ck\left( 0\right) \left\langle \xi \right\rangle ^{\frac{1}{%
1-\delta }},\quad \xi \in \mathbb{R}^{n}
\end{equation*}%
which implies 
\begin{equation*}
\underline{k}\left( \xi \right) =\left\langle \xi \right\rangle ^{-1}k\left(
\xi \right) \leq \left( Ck\left( 0\right) \right) ^{1-\delta }k^{\delta
}\left( \xi \right) ,\quad \xi \in \mathbb{R}^{n}.
\end{equation*}
\end{proof}

\begin{theorem}[Wiener-L\'{e}vy for $\mathcal{B}_{k}^{\infty }$]
Let $k\in \mathcal{K}_{WL}\left( \mathbb{R}^{n}\right) $, $\Omega =\mathring{%
\Omega}\subset \mathbb{C}^{d}$ and $\Phi :\Omega \rightarrow \mathbb{C}$ a
holomorphic function. If $u=\left( u_{1},...,u_{d}\right) \in \mathcal{B}%
_{k}^{\infty }\left( \mathbb{R}^{n}\right) ^{d}$ satisfies the condition $%
\overline{u\left( \mathbb{R}^{n}\right) }\subset \Omega $, then 
\begin{equation*}
\Phi \circ u\equiv \Phi \left( u\right) \in \mathcal{B}_{k}^{\infty }\left( 
\mathbb{R}^{n}\right) .
\end{equation*}
\end{theorem}

\begin{proof}
Let $\delta \in \left( \frac{N-1}{N},1\right) $ be such that $k^{\delta }\in 
\mathcal{K}_{a}\left( \mathbb{R}^{n}\right) $, $k^{-\delta }\in L^{2}\left( 
\mathbb{R}^{n}\right) $ and $k^{-2\delta }\ast \underline{k}^{-2}\leq
C_{k,\delta }^{2}\underline{k}^{-2}$. Since $\delta >\frac{N-1}{N}$ we have
also $\underline{k}=\left\langle \cdot \right\rangle ^{-1}k\leq Ck^{\delta }$
which implies $\mathcal{B}_{k^{\delta }}^{\infty }\left( \mathbb{R}%
^{n}\right) \subset \mathcal{B}_{\underline{k}}^{\infty }\left( \mathbb{R}%
^{n}\right) $.

We consider the family $\left\{ u_{\varepsilon }\right\} _{0<\varepsilon
\leq 1}$ 
\begin{equation*}
u_{\varepsilon }=\varphi _{\varepsilon }\ast u=\left( \varphi _{\varepsilon
}\ast u_{1},...,\varphi _{\varepsilon }\ast u_{d}\right) \in \mathcal{B}%
_{k}^{\infty }\left( \mathbb{R}^{n}\right) ^{d}.
\end{equation*}%
We have%
\begin{equation*}
\partial _{j}\mathit{\Phi }\left( u_{\varepsilon }\right) =\sum_{k=1}^{d}%
\frac{\partial \mathit{\Phi }}{\partial z_{k}}\left( u_{\varepsilon }\right)
\cdot \partial _{j}u_{\varepsilon k},\quad \text{\textit{in} }\mathcal{C}%
^{\infty }\left( 
\mathbb{R}
^{n}\right) ,\quad j=1,...,n.
\end{equation*}%
Let $\chi \in \mathcal{C}_{0}^{\infty }\left( \mathbb{R}^{n}\right)
\smallsetminus 0$. Then Lemma \ref{kh8} and Lemma \ref{kh9} imply that $%
\partial _{j}u_{\varepsilon }=\varphi _{\varepsilon }\ast \partial _{j}u\in 
\mathcal{B}_{\underline{k}}^{\infty }\left( \mathbb{R}^{n}\right) ^{d}$ and 
\begin{equation*}
\left\Vert \partial _{j}u_{\varepsilon k}\right\Vert _{\underline{k},\infty
,\chi }\leq \left\Vert \partial _{j}u_{k}\right\Vert _{\underline{k},\infty
,\chi },\quad j=1,...,n,\quad k=1,...,d.
\end{equation*}%
Now we observe that the pair of weight functions $\left( k^{\delta
},k\right) \in \mathcal{K}_{pol}\left( \mathbb{R}^{n}\right) \times \mathcal{%
K}_{a}\left( \mathbb{R}^{n}\right) $ satisfies the condition 
\begin{equation*}
\frac{k^{\delta }\left( \xi \right) }{k\left( \xi \right) }\rightarrow
0,\quad \xi \rightarrow \infty
\end{equation*}%
(Corollary \ref{kh10}). So we may apply Theorem \ref{ks11} $\left( \mathtt{b}%
\right) $ to obtain that%
\begin{eqnarray*}
\mathit{\Phi }\left( u_{\varepsilon }\right) &\rightarrow &\mathit{\Phi }%
\left( u\right) \quad in\text{ }\mathcal{B}_{k^{\delta }}^{\infty }\left( 
\mathbb{R}^{n}\right) \subset \mathcal{BC}\left( \mathbb{R}^{n}\right)
\subset \mathcal{D}^{\prime }\left( 
\mathbb{R}
^{n}\right) \text{ }as\text{ }\varepsilon \rightarrow 0, \\
\frac{\partial \mathit{\Phi }}{\partial z_{k}}\left( u_{\varepsilon }\right)
&\rightarrow &\frac{\partial \mathit{\Phi }}{\partial z_{k}}\left( u\right)
\quad in\text{ }\mathcal{B}_{k^{\delta }}^{\infty }\left( \mathbb{R}%
^{n}\right) \text{ }as\text{ }\varepsilon \rightarrow 0.
\end{eqnarray*}%
Hence $\partial _{j}\mathit{\Phi }\left( u_{\varepsilon }\right) \rightarrow
\partial _{j}\mathit{\Phi }\left( u\right) $ in $\mathcal{D}^{\prime }\left( 
\mathbb{R}
^{n}\right) $ as $\varepsilon \rightarrow 0$ and there is $\varepsilon
_{0}>0 $ such that for any $\varepsilon \in \left( 0,\varepsilon _{0}\right] 
$ we have%
\begin{equation*}
\left\Vert \frac{\partial \mathit{\Phi }}{\partial z_{k}}\left(
u_{\varepsilon }\right) \right\Vert _{k^{\delta },\infty ,\chi }\leq
2\left\Vert \frac{\partial \mathit{\Phi }}{\partial z_{k}}\left( u\right)
\right\Vert _{k^{\delta },\infty ,\chi },\quad k=1,...,d.
\end{equation*}

Let $\psi \in \mathcal{C}_{0}^{\infty }\left( \mathbb{R}^{n}\right) $. Then
using the above estimates and Lemma \ref{kh11} we can evaluate $\left\vert
\left\langle \left( \tau _{y}\chi ^{2}\right) \partial _{j}\mathit{\Phi }%
\left( u_{\varepsilon }\right) ,\psi \right\rangle \right\vert $ uniformly
with respect to $\varepsilon \in \left( 0,\varepsilon _{0}\right] $, then by
using Lemma \ref{ks18} $\left( \mathtt{d}\right) $ we evaluate $\left\Vert
\partial _{j}\mathit{\Phi }\left( u\right) \right\Vert _{\underline{k}%
,\infty ,\chi ^{2}}$.%
\begin{multline*}
\left\vert \left\langle \left( \tau _{y}\chi ^{2}\right) \partial _{j}%
\mathit{\Phi }\left( u_{\varepsilon }\right) ,\psi \right\rangle \right\vert
\leq \sum_{k=1}^{d}\left\vert \left\langle \left( \tau _{y}\chi \right) 
\frac{\partial \mathit{\Phi }}{\partial z_{k}}\left( u_{\varepsilon }\right)
\cdot \left( \tau _{y}\chi \right) \partial _{j}u_{\varepsilon k},\psi
\right\rangle \right\vert \\
\leq \left( 2\pi \right) ^{-n}\sum_{k=1}^{d}\left\Vert \left( \tau _{y}\chi
\right) \frac{\partial \mathit{\Phi }}{\partial z_{k}}\left( u_{\varepsilon
}\right) \cdot \left( \tau _{y}\chi \right) \partial _{j}u_{\varepsilon
k}\right\Vert _{\mathcal{B}_{\underline{k}}}\left\Vert \psi \right\Vert _{%
\mathcal{B}_{1/\underline{\check{k}}}} \\
\leq \left( 2\pi \right) ^{-2n}C_{k,\delta }\sum_{k=1}^{d}\left\Vert \left(
\tau _{y}\chi \right) \frac{\partial \mathit{\Phi }}{\partial z_{k}}\left(
u_{\varepsilon }\right) \right\Vert _{\mathcal{B}_{k^{\delta }}}\left\Vert
\left( \tau _{y}\chi \right) \partial _{j}u_{\varepsilon k}\right\Vert _{%
\mathcal{B}_{\underline{k}}}\left\Vert \psi \right\Vert _{\mathcal{B}_{1/%
\underline{\check{k}}}} \\
\leq \left( 2\pi \right) ^{-2n}C_{k,\delta }\sum_{k=1}^{d}\left\Vert \frac{%
\partial \mathit{\Phi }}{\partial z_{k}}\left( u_{\varepsilon }\right)
\right\Vert _{k^{\delta },\infty ,\chi }\left\Vert \partial
_{j}u_{\varepsilon k}\right\Vert _{\underline{k},\infty ,\chi }\left\Vert
\psi \right\Vert _{\mathcal{B}_{1/\underline{\check{k}}}} \\
\leq 2\left( 2\pi \right) ^{-2n}C_{k,\delta }\left( \sum_{k=1}^{d}\left\Vert 
\frac{\partial \mathit{\Phi }}{\partial z_{k}}\left( u\right) \right\Vert
_{k^{\delta },\infty ,\chi }\left\Vert \partial _{j}u_{k}\right\Vert _{%
\underline{k},\infty ,\chi }\right) \left\Vert \psi \right\Vert _{\mathcal{B}%
_{1/\underline{\check{k}}}}.
\end{multline*}%
Since $\partial _{j}\mathit{\Phi }\left( u_{\varepsilon }\right) \rightarrow
\partial _{j}\mathit{\Phi }\left( u\right) $ in $\mathcal{D}^{\prime }\left( 
\mathbb{R}
^{n}\right) $ as $\varepsilon \rightarrow 0$ it follows that 
\begin{multline*}
\left\vert \left\langle \left( \tau _{y}\chi ^{2}\right) \partial _{j}%
\mathit{\Phi }\left( u\right) ,\psi \right\rangle \right\vert \\
\leq 2\left( 2\pi \right) ^{-2n}C_{k,\delta }\left( \sum_{k=1}^{d}\left\Vert 
\frac{\partial \mathit{\Phi }}{\partial z_{k}}\left( u\right) \right\Vert
_{k^{\delta },\infty ,\chi }\left\Vert \partial _{j}u_{k}\right\Vert _{%
\underline{k},\infty ,\chi }\right) \left\Vert \psi \right\Vert _{\mathcal{B}%
_{1/\underline{\check{k}}}}.
\end{multline*}%
Hence $\left( \tau _{y}\chi ^{2}\right) \partial _{j}\mathit{\Phi }\left(
u\right) \in \mathcal{B}_{\underline{k}}\left( \mathbb{R}^{n}\right) $ for
every $y\in \mathbb{R}^{n}$ and 
\begin{equation*}
\left\Vert \left( \tau _{y}\chi ^{2}\right) \partial _{j}\mathit{\Phi }%
\left( u\right) \right\Vert _{\mathcal{B}_{\underline{k}}}\leq 2\left( 2\pi
\right) ^{-2n}C_{k,\delta }\left( \sum_{k=1}^{d}\left\Vert \frac{\partial 
\mathit{\Phi }}{\partial z_{k}}\left( u\right) \right\Vert _{k^{\delta
},\infty ,\chi }\left\Vert \partial _{j}u_{k}\right\Vert _{\underline{k}%
,\infty ,\chi }\right)
\end{equation*}%
which gives%
\begin{equation*}
\left\Vert \partial _{j}\mathit{\Phi }\left( u\right) \right\Vert _{%
\underline{k},\infty ,\chi ^{2}}\leq 2\left( 2\pi \right) ^{-2n}C_{k,\delta
}\left( \sum_{k=1}^{d}\left\Vert \frac{\partial \mathit{\Phi }}{\partial
z_{k}}\left( u\right) \right\Vert _{k^{\delta },\infty ,\chi }\left\Vert
\partial _{j}u_{k}\right\Vert _{\underline{k},\infty ,\chi }\right) ,\
j=1,...,n.
\end{equation*}%
Since $\mathit{\Phi }\left( u\right) \in \mathcal{B}_{k^{\delta }}^{\infty
}\left( \mathbb{R}^{n}\right) \subset \mathcal{B}_{\underline{k}}^{\infty
}\left( \mathbb{R}^{n}\right) $ and $\partial _{j}\mathit{\Phi }\left(
u\right) \in \mathcal{B}_{\underline{k}}^{\infty }\left( \mathbb{R}%
^{n}\right) $, $j=1,...,n$ it follows from Lemma \ref{kh9} that $\mathit{%
\Phi }\left( u\right) \in \mathcal{B}_{k}^{\infty }\left( \mathbb{R}%
^{n}\right) $ which ends the proof.
\end{proof}

\begin{corollary}[Kato]
Suppose that $k\in \mathcal{K}_{a}\left( \mathbb{R}^{n}\right) $ satisfies
the WL-condition.

$\left( \mathtt{a}\right) $ If $u\in \mathcal{B}_{k}^{\infty }\left( \mathbb{%
R}^{n}\right) $ satisfies the condition 
\begin{equation*}
\left\vert u\left( x\right) \right\vert \geq c>0,\quad x\in 
\mathbb{R}
^{n},
\end{equation*}%
then 
\begin{equation*}
\frac{1}{u}\in \mathcal{B}_{k}^{\infty }\left( \mathbb{R}^{n}\right) .
\end{equation*}

$\left( \mathtt{b}\right) $ If $u\in \mathcal{B}_{k}^{\infty }\left( \mathbb{%
R}^{n}\right) $, then $\overline{u\left( 
\mathbb{R}
^{n}\right) }$ is the spectrum of the element $u$.
\end{corollary}

\begin{corollary}
Suppose that $k\in \mathcal{K}_{a}\left( \mathbb{R}^{n}\right) $ satisfies
the WL-condition. If $u=\left( u_{1},...,u_{d}\right) \in \mathcal{B}%
_{k}^{\infty }\left( \mathbb{R}^{n}\right) ^{d}$, then 
\begin{equation*}
\sigma _{\mathcal{B}_{k}^{\infty }}\left( u_{1},...,u_{d}\right) =\overline{%
u\left( 
\mathbb{R}
^{n}\right) },
\end{equation*}%
where $\sigma _{\mathcal{B}_{k}^{\infty }}\left( u_{1},...,u_{d}\right) $ is
the joint spectrum of the elements.
\end{corollary}

\begin{proof}
Since 
\begin{equation*}
\delta _{x}:\mathcal{B}_{k}^{\infty }\left( \mathbb{R}^{n}\right) \subset 
\mathcal{BC}\left( 
\mathbb{R}
^{n}\right) \rightarrow 
\mathbb{C}
,\quad w\rightarrow w\left( x\right) ,
\end{equation*}%
is a multiplicative linear functional, Theorem 3.1.14 of \cite{Hormander 2}
implies the inclusion $\overline{u\left( 
\mathbb{R}
^{n}\right) }\subset \sigma _{\mathcal{B}_{k}^{\infty }}\left(
u_{1},...,u_{d}\right) $. On the other hand, if $\lambda =\left( \lambda
_{1},...,\lambda _{d}\right) \notin \overline{u\left( 
\mathbb{R}
^{n}\right) }$, then 
\begin{equation*}
u_{\lambda }=\overline{\left( u_{1}-\lambda _{1}\right) }\left(
u_{1}-\lambda _{1}\right) +...+\overline{\left( u_{d}-\lambda _{d}\right) }%
\left( u_{d}-\lambda _{d}\right) \in \mathcal{B}_{k}^{\infty }\left( \mathbb{%
R}^{n}\right)
\end{equation*}%
satisfies the condition 
\begin{equation*}
u_{\lambda }\left( x\right) \geq c>0,\quad x\in 
\mathbb{R}
^{n}.
\end{equation*}%
It follows that 
\begin{equation*}
\frac{1}{u_{\lambda }}\in \mathcal{B}_{k}^{\infty }\left( \mathbb{R}%
^{n}\right)
\end{equation*}%
and%
\begin{equation*}
v_{1}\left( u_{1}-\lambda _{1}\right) +...+v_{d}\left( u_{d}-\lambda
_{d}\right) =1
\end{equation*}%
with $v_{1}=\overline{\left( u_{1}-\lambda _{1}\right) }/u_{\lambda
},...,v_{d}=\overline{\left( u_{d}-\lambda _{d}\right) }/u_{\lambda }\in 
\mathcal{B}_{k}^{\infty }\left( \mathbb{R}^{n}\right) $. The last equality
expresses precisely that $\lambda \notin \sigma _{\mathcal{B}_{k}^{\infty
}}\left( u_{1},...,u_{d}\right) $.
\end{proof}

\begin{corollary}
Suppose that $k\in \mathcal{K}_{a}\left( \mathbb{R}^{n}\right) $ satisfies
the WL-condition. Let $\Omega =\mathring{\Omega}\subset 
\mathbb{C}
^{d}$ and $\Phi :\Omega \rightarrow 
\mathbb{C}
$ a holomorphic function.

$\left( \mathtt{a}\right) $ Let $1\leq p<\infty $. If $u=\left(
u_{1},...,u_{d}\right) \in \mathcal{B}_{k}^{p}\left( 
\mathbb{R}
^{n}\right) ^{d}$ satisfies the condition $\overline{u_{1}\left( 
\mathbb{R}
^{n}\right) }\times ...\times \overline{u_{d}\left( 
\mathbb{R}
^{n}\right) }\subset \Omega $ and if $\Phi \left( 0\right) =0$, then $\Phi
\left( u\right) \in \mathcal{B}_{k}^{p}\left( 
\mathbb{R}
^{n}\right) $.

$\left( \mathtt{b}\right) $ If $u=\left( u_{1},...,u_{d}\right) \in
B_{2,k}\left( \mathbb{R}^{n}\right) ^{d}$ satisfies the condition $\overline{%
u_{1}\left( 
\mathbb{R}
^{n}\right) }\times ...\times \overline{u_{d}\left( 
\mathbb{R}
^{n}\right) }\subset \Omega $ and if $\Phi \left( 0\right) =0$, then $\Phi
\left( u\right) \in \mathcal{B}_{k}\left( \mathbb{R}^{n}\right)
=B_{2,k}\left( \mathbb{R}^{n}\right) $.
\end{corollary}

\begin{proof}
$\left( \mathtt{a}\right) $ Since $\mathcal{B}_{k}^{p}\left( \mathbb{R}%
^{n}\right) $ is an ideal in the algebra $\mathcal{B}_{k}^{\infty }\left( 
\mathbb{R}^{n}\right) $, it follows that $0$ belongs to the spectrum of any
element of $\mathcal{B}_{k}^{p}\left( 
\mathbb{R}
^{n}\right) $. Hence $0\in \overline{u_{1}\left( 
\mathbb{R}
^{n}\right) }\times ...\times \overline{u_{d}\left( 
\mathbb{R}
^{n}\right) }\subset \mathit{\Omega }$. Shrinking $\mathit{\Omega }$ if
necessary, we can assume that $\mathit{\Omega }=\mathit{\Omega }_{1}\times
...\times \mathit{\Omega }_{d}$ with $\overline{u_{k}\left( 
\mathbb{R}
^{n}\right) }\subset \mathit{\Omega }_{k}$, $k=1,...,d$. Now we continue by
induction on $d$.

Let $F:\mathit{\Omega }\rightarrow 
\mathbb{C}
$ be the holomorphic function defined by 
\begin{equation*}
F\left( z_{1},...,z_{d}\right) =\left\{ 
\begin{array}{ccc}
\frac{\mathit{\Phi }\left( z_{1},...,z_{d}\right) -\mathit{\Phi }\left(
0,...,z_{d}\right) }{z_{1}} & if & z_{1}\neq 0, \\ 
\frac{\partial \mathit{\Phi }}{\partial z_{1}}\left( 0,...,z_{d}\right) & if
& z_{1}=0.%
\end{array}%
\right.
\end{equation*}%
Then $\mathit{\Phi }\left( z_{1},...,z_{d}\right) =z_{1}F\left(
z_{1},...,z_{d}\right) +\mathit{\Phi }\left( 0,...,z_{d}\right) $, so 
\begin{equation*}
\mathit{\Phi }\left( u\right) =u_{1}F\left( u\right) +\mathit{\Phi }\left(
0,...,u_{d}\right) \in \mathcal{B}_{k}^{p}\left( 
\mathbb{R}
^{n}\right)
\end{equation*}%
because $u_{1}F\left( u\right) \in \mathcal{B}_{k}^{p}\left( 
\mathbb{R}
^{n}\right) \cdot \mathcal{B}_{k}^{\infty }\left( \mathbb{R}^{n}\right)
\subset \mathcal{B}_{k}^{p}\left( 
\mathbb{R}
^{n}\right) $ and $\mathit{\Phi }\left( 0,...,u_{d}\right) \in \mathcal{B}%
_{k}^{p}\left( 
\mathbb{R}
^{n}\right) $ by inductive hypothesis.

$\left( \mathtt{b}\right) $ is a consequence of $\left( \mathtt{a}\right) $
and of localization principle.
\end{proof}

\section{The spaces $\mathcal{B}_{k}^{p}$ and Schatten-von Neumann class
properties for pseudo-differential operators}

We begin this section with some interpolation results of $\mathcal{B}%
_{k}^{p} $ spaces.

For $k\in \mathcal{K}_{pol}\left( \mathbb{R}^{n}\right) $, if 
\begin{gather*}
F_{k}=\left\{ v\in \mathcal{S}^{\prime }\left( \mathbb{R}^{n}\right) :kv\in
L^{2}\left( \mathbb{R}^{n}\right) \right\} , \\
\left\Vert v\right\Vert _{F_{k}}=\left\Vert kv\right\Vert _{L^{2}},\quad
v\in F_{k},
\end{gather*}%
then the Fourier transform $\mathcal{F}$ is an isometry (up to a constant
factor) from $\mathcal{B}_{k}\left( \mathbb{R}^{n}\right) $ onto $F_{k}$ and
the inverse Fourier transform $\mathcal{F}^{-1}$ is an isometry (up to a
constant factor) from $F_{k}$ onto $\mathcal{B}_{k}\left( \mathbb{R}%
^{n}\right) $. The interpolation property implies then that $\mathcal{F}$
maps continuously $\left[ \mathcal{B}_{k_{0}}\left( \mathbb{R}^{n}\right) ,%
\mathcal{B}_{k_{1}}\left( \mathbb{R}^{n}\right) \right] _{\theta }$ into $%
\left[ F_{k_{0}},F_{k_{1}}\right] _{\theta }$ and $\mathcal{F}^{-1}$ maps
continuously $\left[ F_{k_{0}},F_{k_{1}}\right] _{\theta }$ into $\left[ 
\mathcal{B}_{k_{0}}\left( \mathbb{R}^{n}\right) ,\mathcal{B}_{k_{1}}\left( 
\mathbb{R}^{n}\right) \right] _{\theta }$, so that $\left[ \mathcal{B}%
_{k_{0}}\left( \mathbb{R}^{n}\right) ,\mathcal{B}_{k_{1}}\left( \mathbb{R}%
^{n}\right) \right] _{\theta }$ coincides with the tempered distributions
whose Fourier transform belongs to $\left[ F_{k_{0}},F_{k_{1}}\right]
_{\theta }$ (and one deduces in the same way that it is an isometry if one
uses the corresponding norms). Identifying interpolation spaces between
spaces $\mathcal{B}_{k}\left( \mathbb{R}^{n}\right) $ is then the same
question as interpolating between some $L^{2}$ spaces with weights. The
following lemma is a consequence of this remark and Theorem 1.18.5 in \cite%
{Triebel}.

\begin{lemma}
If $k_{0}$, $k_{1}\in \mathcal{K}_{pol}\left( \mathbb{R}^{n}\right) $, $%
0<\theta <1$ and $k=k_{0}^{1-\theta }\cdot k_{1}^{\theta }\in \mathcal{K}%
_{pol}\left( \mathbb{R}^{n}\right) $, then 
\begin{equation*}
\left[ \mathcal{B}_{k_{0}}\left( \mathbb{R}^{n}\right) ,\mathcal{B}%
_{k_{1}}\left( \mathbb{R}^{n}\right) \right] _{\theta }=\mathcal{B}%
_{k}\left( \mathbb{R}^{n}\right) .
\end{equation*}
\end{lemma}

Using the results of \cite{Triebel} Subsection 1.18.1 we obtain the
following corollary.

\begin{corollary}
Let $k_{0}$, $k_{1}\in \mathcal{K}_{pol}\left( \mathbb{R}^{n}\right) $, $%
1\leq p_{0}<\infty $, $1\leq p_{1}\leq \infty $, $0<\theta <1$ and 
\begin{equation*}
\frac{1}{p}=\frac{1-\theta }{p_{0}}+\frac{\theta }{p_{1}},\quad
k=k_{0}^{1-\theta }\cdot k_{1}^{\theta }\in \mathcal{K}\left( \mathbb{R}%
^{n}\right) .
\end{equation*}%
Then 
\begin{equation*}
\left[ l^{p_{0}}\left( \mathbb{Z}^{n},\mathcal{B}_{k_{0}}\left( \mathbb{R}%
^{n}\right) \right) ,l^{p_{1}}\left( \mathbb{Z}^{n},\mathcal{B}%
_{k_{1}}\left( \mathbb{R}^{n}\right) \right) \right] _{\theta }=l^{p}\left( 
\mathbb{Z}^{n},\mathcal{B}_{k}\left( \mathbb{R}^{n}\right) \right) .
\end{equation*}
\end{corollary}

We pass now to the Kato-H\"{o}rmander spaces $\mathcal{B}_{k}^{p}$. We
choose $\chi _{\mathbb{Z}^{n}}\in \mathcal{C}_{0}^{\infty }\left( \mathbb{R}%
^{n}\right) $ so that $\sum_{\gamma \in \mathbb{Z}^{n}}\chi _{\mathbb{Z}%
^{n}}\left( \cdot -\gamma \right) =1$. For $\gamma \in \mathbb{Z}^{n}$ we
define the operator%
\begin{equation*}
S_{\gamma }:\mathcal{D}^{\prime }\left( \mathbb{R}^{n}\right) \rightarrow 
\mathcal{D}^{\prime }\left( \mathbb{R}^{n}\right) ,\quad S_{\gamma }u=\left(
\tau _{\gamma }\chi _{\mathbb{Z}^{n}}\right) u.
\end{equation*}%
Now from the definition of $\mathcal{B}_{k}^{p}$ it follows that the linear
operator 
\begin{equation*}
S:\mathcal{B}_{k}^{p}\left( \mathbb{R}^{n}\right) \rightarrow l^{p}\left( 
\mathbb{Z}^{n},\mathcal{B}_{k}\left( \mathbb{R}^{n}\right) \right) ,\quad
Su=\left( S_{\gamma }u\right) _{\gamma \in \mathbb{Z}^{n}}
\end{equation*}%
is well defined and continuous.

On the other hand, for any $\chi \in \mathcal{C}_{0}^{\infty }$ the operator 
\begin{gather*}
R_{\chi }:l^{p}\left( \mathbb{Z}^{n},\mathcal{B}_{k}\left( \mathbb{R}%
^{n}\right) \right) \rightarrow \mathcal{B}_{k}^{p}\left( \mathbb{R}%
^{n}\right) , \\
R_{\chi }\left( \left( u_{\gamma }\right) _{\gamma \in \mathbb{Z}%
^{n}}\right) =\sum_{\gamma \in \mathbb{Z}^{n}}\left( \tau _{\gamma }\chi
\right) u_{\gamma }
\end{gather*}%
is well defined and continuous.

Let $\underline{\mathbf{u}}=\left( u_{\gamma }\right) _{\gamma \in \mathbb{Z}%
^{n}}\in l^{p}\left( \mathbb{Z}^{n},\mathcal{B}_{k}\left( \mathbb{R}%
^{n}\right) \right) $. Using Proposition \ref{ks4} we get 
\begin{equation*}
\left\Vert \left( \tau _{\gamma ^{\prime }}\chi _{\mathbb{Z}^{n}}\right)
\left( \tau _{\gamma }\chi \right) u_{\gamma }\right\Vert _{\mathcal{B}%
_{k}}\leq Cst\sup_{\left\vert \alpha +\beta \right\vert \leq
m_{k}}\left\vert \left( \left( \tau _{\gamma ^{\prime }}\partial ^{\alpha
}\chi _{\mathbb{Z}^{n}}\right) \left( \tau _{\gamma }\partial ^{\beta }\chi
\right) \right) \right\vert \left\Vert u_{\gamma }\right\Vert _{\mathcal{B}%
_{k}}.
\end{equation*}%
where $m_{k}=\left[ N+\frac{n+1}{2}\right] +1$. Now for some continuous
seminorm $p=p_{n,k}$ on $\mathcal{S}\left( \mathbb{R}^{n}\right) $ we have 
\begin{eqnarray*}
\left\vert \left( \left( \tau _{\gamma ^{\prime }}\partial ^{\alpha }\chi _{%
\mathbb{Z}^{n}}\right) \left( \tau _{\gamma }\partial ^{\beta }\chi \right)
\right) \left( x\right) \right\vert &\leq &p\left( \chi _{\mathbb{Z}%
^{n}}\right) p\left( \chi \right) \left\langle x-\gamma ^{\prime
}\right\rangle ^{-2\left( n+1\right) }\left\langle x-\gamma \right\rangle
^{-2\left( n+1\right) } \\
&\leq &2^{n+1}p\left( \chi _{\mathbb{Z}^{n}}\right) p\left( \chi \right)
\left\langle 2x-\gamma ^{\prime }-\gamma \right\rangle ^{-n-1}\left\langle
\gamma ^{\prime }-\gamma \right\rangle ^{-n-1} \\
&\leq &2^{n+1}p\left( \chi _{\mathbb{Z}^{n}}\right) p\left( \chi \right)
\left\langle \gamma ^{\prime }-\gamma \right\rangle ^{-n-1},\quad \left\vert
\alpha +\beta \right\vert \leq m_{k}.
\end{eqnarray*}%
Hence 
\begin{gather*}
\sup_{\left\vert \alpha +\beta \right\vert \leq m_{k}}\left\vert \left(
\left( \tau _{\gamma ^{\prime }}\partial ^{\alpha }\chi _{\mathbb{Z}%
^{n}}\right) \left( \tau _{\gamma }\partial ^{\beta }\chi \right) \right)
\right\vert \leq 2^{n+1}p\left( \chi _{\mathbb{Z}^{n}}\right) p\left( \chi
\right) \left\langle \gamma ^{\prime }-\gamma \right\rangle ^{-n-1}, \\
\left\Vert \left( \tau _{\gamma ^{\prime }}\chi _{\mathbb{Z}^{n}}\right)
\left( \tau _{\gamma }\chi \right) u_{\gamma }\right\Vert _{\mathcal{B}%
_{k}}\leq C\left( n,k\mathbf{,}\chi _{\mathbb{Z}^{n}},\chi \right)
\left\langle \gamma ^{\prime }-\gamma \right\rangle ^{-n-1}\left\Vert
u_{\gamma }\right\Vert _{\mathcal{B}_{k}}.
\end{gather*}%
The last estimate implies that 
\begin{equation*}
\left\Vert \left( \tau _{\gamma ^{\prime }}\chi _{\mathbb{Z}^{n}}\right)
R_{\chi }\left( \underline{\mathbf{u}}\right) \right\Vert _{\mathcal{B}%
_{k}}\leq C\left( n,k\mathbf{,}\chi _{\mathbb{Z}^{n}},\chi \right)
\sum_{\gamma \in \mathbb{Z}^{n}}\left\langle \gamma ^{\prime }-\gamma
\right\rangle ^{-n-1}\left\Vert u_{\gamma }\right\Vert _{\mathcal{B}_{k}}.
\end{equation*}%
Now Schur's lemma implies the result%
\begin{equation*}
\left( \sum_{\gamma ^{\prime }\in \mathbb{Z}^{n}}\left\Vert \left( \tau
_{\gamma ^{\prime }}\chi _{\mathbb{Z}^{n}}\right) R_{\chi }\left( \underline{%
\mathbf{u}}\right) \right\Vert _{\mathcal{B}_{k}}^{p}\right) ^{\frac{1}{p}%
}\leq C^{\prime }\left( n,k\mathbf{,}\chi _{\mathbb{Z}^{n}},\chi \right)
\left\Vert \left\langle \cdot \right\rangle ^{-n-1}\right\Vert
_{L^{1}}\left( \sum_{\gamma \in \mathbb{Z}^{n}}\left\Vert u_{\gamma
}\right\Vert _{\mathcal{B}_{k}}^{p}\right) ^{\frac{1}{p}}.
\end{equation*}

If $\chi =1$ on a neighborhood of $\mathtt{supp}\chi _{\mathbb{Z}^{n}}$,
then $\chi \chi _{\mathbb{Z}^{n}}=\chi _{\mathbb{Z}^{n}}$ and as a
consequence $R_{\chi }S=\mathtt{Id}_{\mathcal{B}_{k}^{p}\left( \mathbb{R}%
^{n}\right) }$:%
\begin{eqnarray*}
R_{\chi }Su &=&\sum_{\gamma \in \mathbb{Z}^{n}}\left( \tau _{\gamma }\chi
\right) S_{\gamma }u=\sum_{\gamma \in \mathbb{Z}^{n}}\left( \tau _{\gamma
}\chi \right) \left( \tau _{\gamma }\chi _{\mathbb{Z}^{n}}\right) u \\
&=&\sum_{\gamma \in \mathbb{Z}^{n}}\left( \tau _{\gamma }\chi _{\mathbb{Z}%
^{n}}\right) u=u.
\end{eqnarray*}%
Thus we proved the following result.

\begin{proposition}
Under the above conditions, the operator $R_{\chi }:l^{p}\left( \mathbb{Z}%
^{n},\mathcal{B}_{k}\left( \mathbb{R}^{n}\right) \right) \rightarrow 
\mathcal{B}_{k}^{p}\left( \mathbb{R}^{n}\right) $ is a retraction and the
operator $S:\mathcal{B}_{k}^{p}\left( \mathbb{R}^{n}\right) \rightarrow
l^{p}\left( \mathbb{Z}^{n},\mathcal{B}_{k}\left( \mathbb{R}^{n}\right)
\right) $ is a coretraction.
\end{proposition}

\begin{corollary}
\label{ks17}Let $k_{0},k_{1}\in \mathcal{K}_{pol}\left( \mathbb{R}%
^{n}\right) $, $1\leq p_{0}<\infty $, $1\leq p_{1}\leq \infty $, $0<\theta
<1 $ and 
\begin{equation*}
\frac{1}{p}=\frac{1-\theta }{p_{0}}+\frac{\theta }{p_{1}},\quad
k=k_{0}^{1-\theta }\cdot k_{1}^{\theta }\in \mathcal{K}\left( \mathbb{R}%
^{n}\right) .
\end{equation*}%
Then 
\begin{equation*}
\left[ \mathcal{B}_{k_{0}}^{p_{0}}\left( \mathbb{R}^{n}\right) ,\mathcal{B}%
_{k_{1}}^{p_{1}}\left( \mathbb{R}^{n}\right) \right] _{\theta }=\mathcal{B}%
_{k}^{p}\left( \mathbb{R}^{n}\right) .
\end{equation*}
\end{corollary}

\begin{proof}
The last part of Theorem \ref{ks15} $\left( \mathtt{d}\right) $ shows that $%
\left\{ \mathcal{B}_{k_{0}}^{p_{0}}\left( \mathbb{R}^{n}\right) ,\mathcal{B}%
_{k_{1}}^{p_{1}}\left( \mathbb{R}^{n}\right) \right\} $ is an interpolation
couple (in the sense of the notations of Subsection 1.2.1 of \cite{Triebel}
one can choose $\mathcal{A}=\mathcal{S}^{\prime }\left( \mathbb{R}%
^{n}\right) $). If $F$ is an interpolation functor, then one obtains by
Theorem 1.2.4 of \cite{Triebel} that%
\begin{equation*}
\left\Vert u\right\Vert _{F\left( \left\{ \mathcal{B}_{k_{0}}^{p_{0}}\left( 
\mathbb{R}^{n}\right) ,\mathcal{B}_{k_{1}}^{p_{1}}\left( \mathbb{R}%
^{n}\right) \right\} \right) }\approx \left\Vert \left( S_{\gamma }u\right)
_{\gamma \in \mathbb{Z}^{n}}\right\Vert _{F\left( \left\{ l^{p_{0}}\left( 
\mathbb{Z}^{n},\mathcal{B}_{k_{0}}\right) ,l^{p_{1}}\left( \mathbb{Z}^{n},%
\mathcal{B}_{k_{1}}\right) \right\} \right) }.
\end{equation*}%
By specialization we obtain 
\begin{eqnarray*}
\left\Vert u\right\Vert _{\left[ \mathcal{B}_{k_{0}}^{p_{0}}\left( \mathbb{R}%
^{n}\right) ,\mathcal{B}_{k_{1}}^{p_{1}}\left( \mathbb{R}^{n}\right) \right]
_{\theta }} &\approx &\left\Vert \left( S_{\gamma }u\right) _{\gamma \in 
\mathbb{Z}^{n}}\right\Vert _{\left[ l^{p_{0}}\left( \mathbb{Z}^{n},\mathcal{B%
}_{k_{0}}\left( \mathbb{R}^{n}\right) \right) ,l^{p_{1}}\left( \mathbb{Z}%
^{n},\mathcal{B}_{k_{1}}\left( \mathbb{R}^{n}\right) \right) \right]
_{\theta }} \\
&\approx &\left\Vert \left( S_{\gamma }u\right) _{\gamma \in \mathbb{Z}%
^{n}}\right\Vert _{l^{p}\left( \mathbb{Z}^{n},\mathcal{B}_{k}\left( \mathbb{R%
}^{n}\right) \right) } \\
&\approx &\left\Vert u\right\Vert _{\mathcal{B}_{k}^{p}\left( \mathbb{R}%
^{n}\right) }.
\end{eqnarray*}
\end{proof}

In addition to the above interpolation results we need an embedding theorem
which we shall prove below. First we shall recall the definition of spaces
that appear in this theorem.

\begin{definition}
Let $1\leq p\leq \infty $. We say that a distribution $u\in \mathcal{D}%
^{\prime }\left( \mathbb{R}^{n}\right) $ belongs to $S_{w}^{p}\left( \mathbb{%
R}^{n}\right) $ if there is $\chi \in \mathcal{C}_{0}^{\infty }\left( 
\mathbb{R}^{n}\right) \smallsetminus 0$ such that the measurable function%
\begin{gather*}
U_{\chi ,p}:\mathbb{R}^{n}\rightarrow \left[ 0,+\infty \right) , \\
U_{\chi ,p}\left( \xi \right) =\left\{ 
\begin{array}{ccc}
\sup_{y\in \mathbb{R}^{n}}\left\vert \widehat{u\tau _{y}\chi }\left( \xi
\right) \right\vert & \text{\textit{if}} & p=\infty \\ 
\left( \int \left\vert \widehat{u\tau _{y}\chi }\left( \xi \right)
\right\vert ^{p}\mathtt{d}y\right) ^{1/p} & \text{\textit{if}} & 1\leq
p<\infty%
\end{array}%
\right. , \\
\widehat{u\tau _{y}\chi }\left( \xi \right) =\left\langle u,\mathtt{e}^{-%
\mathtt{i}\left\langle \cdot ,\xi \right\rangle }\chi \left( \cdot -y\right)
\right\rangle ,
\end{gather*}%
belongs to $L^{1}\left( \mathbb{R}^{n}\right) $.
\end{definition}

These spaces are special cases of modulation spaces which were introduced by
Hans Georg Feichtinger in 1983. They were used by many authors (Boulkhemair,
Gr\"{o}chenig, Heil, Sj\"{o}strand, Toft ...) in the analysis of
pseudo-differential operators defined by symbols more general than usual.

Now we give some properties of these spaces.

\begin{proposition}
$(\mathtt{a})$ Let $u\in S_{w}^{p}\left( \mathbb{R}^{n}\right) $ and let $%
\chi \in \mathcal{C}_{0}^{\infty }\left( \mathbb{R}^{n}\right) $. Then the
measurable function%
\begin{gather*}
U_{\chi ,p}:\mathbb{R}^{n}\rightarrow \left[ 0,+\infty \right) , \\
U_{\chi ,p}\left( \xi \right) =\left\{ 
\begin{array}{ccc}
\sup_{y\in \mathbb{R}^{n}}\left\vert \widehat{u\tau _{y}\chi }\left( \xi
\right) \right\vert & \text{\textit{if}} & p=\infty \\ 
\left( \int \left\vert \widehat{u\tau _{y}\chi }\left( \xi \right)
\right\vert ^{p}\mathtt{d}y\right) ^{1/p} & \text{\textit{if}} & 1\leq
p<\infty%
\end{array}%
\right. , \\
\widehat{u\tau _{y}\chi }\left( \xi \right) =\left\langle u,\mathtt{e}^{-%
\mathtt{i}\left\langle \cdot ,\xi \right\rangle }\chi \left( \cdot -y\right)
\right\rangle ,
\end{gather*}%
belongs to $L^{1}\left( \mathbb{R}^{n}\right) $.

$(\mathtt{b})$ If we fix $\chi \in \mathcal{C}_{0}^{\infty }\left( \mathbb{R}%
^{n}\right) \smallsetminus 0$ and if we put%
\begin{equation*}
\left\Vert u\right\Vert _{S_{w}^{p},\chi }=\int U_{\chi ,p}\left( \xi
\right) \mathtt{d}\xi =\left\Vert U_{\chi ,p}\right\Vert _{L^{1}},\quad u\in
S_{w}\left( \mathbb{R}^{n}\right) ,
\end{equation*}%
then $\left\Vert \cdot \right\Vert _{S_{w}^{p},\chi }$ is a norm on $%
S_{w}^{p}\left( \mathbb{R}^{n}\right) $ and the topology that defines does
not depend on the choice of the function $\chi \in \mathcal{C}_{0}^{\infty
}\left( \mathbb{R}^{n}\right) \smallsetminus 0$.

$(\mathtt{c})$ Let $1\leq p\leq q\leq \infty $. Then 
\begin{equation*}
S_{w}^{1}\left( \mathbb{R}^{n}\right) \subset S_{w}^{p}\left( \mathbb{R}%
^{n}\right) \subset S_{w}^{q}\left( \mathbb{R}^{n}\right) \subset
S_{w}^{\infty }\left( \mathbb{R}^{n}\right) =S_{w}\left( \mathbb{R}%
^{n}\right) \subset \mathcal{BC}\left( \mathbb{R}^{n}\right) \subset 
\mathcal{S}^{\prime }\left( \mathbb{R}^{n}\right) .
\end{equation*}
\end{proposition}

A proof of this proposition can be found for instance in \cite{Arsu} or \cite%
{Toft1}.

\begin{theorem}
Let $k\in \mathcal{K}_{pol}\left( \mathbb{R}^{n}\right) $ and $1\leq p\leq
\infty $. If $1/k\in L^{1}\left( \mathbb{R}^{n}\right) $, then $\mathcal{B}%
_{k}^{p}\left( \mathbb{R}^{n}\right) \hookrightarrow S_{w}^{p}\left( \mathbb{%
R}^{n}\right) $.
\end{theorem}

\begin{proof}
Let $u\in \mathcal{B}_{k}^{p}\left( \mathbb{R}^{n}\right) $. Let $\chi ,%
\widetilde{\chi }\in \mathcal{C}_{0}^{\infty }\left( \mathbb{R}^{n}\right)
\smallsetminus 0$ be such that $\widetilde{\chi }=1$ on $\mathtt{supp}\chi $%
. For $y\in \mathbb{R}^{n}$ we have 
\begin{equation*}
u\tau _{y}\chi =\left( u\tau _{y}\widetilde{\chi }\right) \left( \tau
_{y}\chi \right) \text{ }\Rightarrow \text{ }\widehat{u\tau _{y}\chi }%
=\left( 2\pi \right) ^{-n}\widehat{u\tau _{y}\widetilde{\chi }}\ast \widehat{%
\tau _{y}\chi }.
\end{equation*}%
Multiplying by $k\left( \xi \right) $ and noting the inequality $k\left( \xi
\right) \leq C\left\langle \xi -\eta \right\rangle ^{N}k\left( \eta \right) $%
, we obtain 
\begin{eqnarray*}
k\left( \xi \right) \left\vert \widehat{u\tau _{y}\chi }\left( \xi \right)
\right\vert &\leq &\left( 2\pi \right) ^{-n}C\int k\left( \eta \right)
\left\vert \widehat{u\tau _{y}\widetilde{\chi }}\left( \eta \right)
\right\vert \left\langle \xi -\eta \right\rangle ^{N}\left\vert \widehat{%
\tau _{y}\chi }\left( \xi -\eta \right) \right\vert \mathtt{d}\xi \\
&\leq &\left( 2\pi \right) ^{-n}C\left\Vert k\widehat{u\tau _{y}\widetilde{%
\chi }}\right\Vert _{L^{2}}\left\Vert \left\langle \cdot \right\rangle ^{N}%
\widehat{\tau _{y}\chi }\right\Vert _{L^{2}} \\
&=&\left( 2\pi \right) ^{-n}C\left\Vert u\tau _{y}\widetilde{\chi }%
\right\Vert _{\mathcal{B}_{k}}\left\Vert \chi \right\Vert _{H^{N}},
\end{eqnarray*}%
hence 
\begin{equation*}
\left\vert \widehat{u\tau _{y}\chi }\left( \xi \right) \right\vert \leq
\left( 2\pi \right) ^{-n}C\left\Vert u\tau _{y}\widetilde{\chi }\right\Vert
_{\mathcal{B}_{k}}\left\Vert \chi \right\Vert _{H^{N}}\frac{1}{k\left( \xi
\right) }.
\end{equation*}%
It follows that%
\begin{gather*}
U_{\chi ,p}\left( \xi \right) \leq \left( 2\pi \right) ^{-n}C\left\Vert \chi
\right\Vert _{H^{N}}\frac{1}{k\left( \xi \right) }\left\{ 
\begin{array}{ccc}
\left\Vert u\right\Vert _{k,\infty ,\widetilde{\chi }} & \text{\textit{if}}
& p=\infty \\ 
\left\Vert u\right\Vert _{k,p,\widetilde{\chi }} & \text{\textit{if}} & 
1\leq p<\infty%
\end{array}%
\right. \\
\leq \left( 2\pi \right) ^{-n}C\left\Vert \chi \right\Vert
_{H^{N}}\left\Vert u\right\Vert _{k,p,\widetilde{\chi }}\frac{1}{k\left( \xi
\right) },
\end{gather*}%
which implies that%
\begin{equation*}
\left\Vert u\right\Vert _{S_{w}^{p},\chi }=\left\Vert U_{\chi ,p}\right\Vert
_{L^{1}}\leq \left( 2\pi \right) ^{-n}C\left\Vert \chi \right\Vert
_{H^{N}}\left\Vert 1/k\right\Vert _{L^{1}}\left\Vert u\right\Vert _{\mathbf{s%
},p,\widetilde{\chi }},\quad u\in \mathcal{B}_{k}^{p}.
\end{equation*}
\end{proof}

This embedding theorem allows us to deal with Schatten-von Neumann class
properties of pseudo-differential operators.

Let $\tau \in \mathtt{End}_{\mathbb{R}}\left( \mathbb{R}^{n}\right) \equiv
M_{n\times n}\left( \mathbb{R}\right) $, $a\in \mathcal{S}\left( \mathbb{R}%
^{n}\times \mathbb{R}^{n}\right) $, $v\in \mathcal{S}\left( \mathbb{R}%
^{n}\right) $. We define 
\begin{eqnarray*}
\mathtt{Op}_{\tau }\left( a\right) v\left( x\right) &=&a^{\tau }\left(
X,D\right) v\left( x\right) \\
&=&\left( 2\pi \right) ^{-n}\iint \mathtt{e}^{\mathtt{i}\left\langle
x-y,\eta \right\rangle }a\left( \left( 1-\tau \right) x+\tau y,\eta \right)
v\left( y\right) \mathtt{d}y\mathtt{d}\eta .
\end{eqnarray*}%
If $u,v\in \mathcal{S}\left( \mathbb{R}^{n}\right) $, then%
\begin{eqnarray*}
\left\langle \mathtt{Op}_{\tau }\left( a\right) v,u\right\rangle &=&\left(
2\pi \right) ^{-n}\iiint \mathtt{e}^{\mathtt{i}\left\langle x-y,\eta
\right\rangle }a\left( \left( 1-\tau \right) x+\tau y,\eta \right) u\left(
x\right) v\left( y\right) \mathtt{d}x\mathtt{d}y\mathtt{d}\eta \\
&=&\left\langle \left( \left( 1\otimes \mathcal{F}^{-1}\right) a\right)
\circ \mathtt{C}_{\tau },u\otimes v\right\rangle ,
\end{eqnarray*}%
where 
\begin{equation*}
\mathtt{C}_{\tau }:\mathbb{R}^{n}\times \mathbb{R}^{n}\rightarrow \mathbb{R}%
^{n}\times \mathbb{R}^{n},\quad \mathtt{C}_{\tau }\left( x,y\right) =\left(
\left( 1-\tau \right) x+\tau y,x-y\right) .
\end{equation*}%
We can define $\mathtt{Op}_{\tau }\left( a\right) $ as an operator in $%
\mathcal{B}\left( \mathcal{S}\left( \mathbb{R}^{n}\right) ,\mathcal{S}%
^{\prime }\left( \mathbb{R}^{n}\right) \right) $ for any $a\in \mathcal{S}%
^{\prime }\left( \mathbb{R}^{n}\times \mathbb{R}^{n}\right) $%
\begin{eqnarray*}
\left\langle \mathtt{Op}_{\tau }\left( a\right) v,u\right\rangle _{\mathcal{S%
},\mathcal{S}^{\prime }} &=&\left\langle \mathcal{K}_{\mathtt{Op}_{\tau
}\left( a\right) },u\otimes v\right\rangle , \\
\mathcal{K}_{\mathtt{Op}_{\tau }\left( a\right) } &=&\left( \left( 1\otimes 
\mathcal{F}^{-1}\right) a\right) \circ \mathtt{C}_{\tau }.
\end{eqnarray*}

\begin{theorem}
Let $k\in \mathcal{K}_{pol}\left( \mathbb{R}^{n}\times \mathbb{R}^{n}\right) 
$ be such that $1/k\in L^{1}\left( \mathbb{R}^{n}\times \mathbb{R}%
^{n}\right) $.

$(\mathtt{a})$ Let $1\leq p<\infty $, $\tau \in \mathtt{End}_{\mathbb{R}%
}\left( \mathbb{R}^{n}\right) \equiv M_{n\times n}\left( \mathbb{R}\right) $
and $a\in \mathcal{B}_{k}^{p}\left( \mathbb{R}^{n}\times \mathbb{R}%
^{n}\right) $. Then 
\begin{equation*}
\mathtt{Op}_{\tau }\left( a\right) =a^{\tau }\left( X,D\right) \in \mathcal{B%
}_{p}\left( L^{2}\left( \mathbb{R}^{n}\right) \right) ,
\end{equation*}%
where $\mathcal{B}_{p}\left( L^{2}\left( \mathbb{R}^{n}\right) \right) $
denote the Schatten ideal of compact operators whose singular values lie in $%
l^{p}$. We have 
\begin{equation*}
\left\Vert \mathtt{Op}_{\tau }\left( a\right) \right\Vert _{\mathcal{B}%
_{p}\left( L^{2}\left( \mathbb{R}^{n}\right) \right) }\leq Cst\left\Vert
a\right\Vert _{\mathcal{B}_{k}^{p}}.
\end{equation*}%
Moreover, the mapping%
\begin{equation*}
\mathtt{End}_{\mathbb{R}}\left( \mathbb{R}^{n}\right) \ni \tau \rightarrow 
\mathtt{Op}_{\tau }\left( a\right) =a^{\tau }\left( X,D\right) \in \mathcal{B%
}_{p}\left( L^{2}\left( \mathbb{R}^{n}\right) \right)
\end{equation*}%
is continuous.

$(\mathtt{b})$ Let $\tau \in \mathtt{End}_{\mathbb{R}}\left( \mathbb{R}%
^{n}\right) \equiv M_{n\times n}\left( \mathbb{R}\right) $ and $a\in 
\mathcal{B}_{k}^{\infty }\left( \mathbb{R}^{n}\times \mathbb{R}^{n}\right) $%
. Then 
\begin{equation*}
\mathtt{Op}_{\tau }\left( a\right) =a^{\tau }\left( X,D\right) \in \mathcal{B%
}\left( L^{2}\left( \mathbb{R}^{n}\right) \right) .
\end{equation*}%
We have 
\begin{equation*}
\left\Vert \mathtt{Op}_{\tau }\left( a\right) \right\Vert _{\mathcal{B}%
\left( L^{2}\left( \mathbb{R}^{n}\right) \right) }\leq Cst\left\Vert
a\right\Vert _{\mathcal{B}_{k}^{\infty }}.
\end{equation*}%
Moreover, the mapping%
\begin{equation*}
\mathtt{End}_{\mathbb{R}}\left( \mathbb{R}^{n}\right) \ni \tau \rightarrow 
\mathtt{Op}_{\tau }\left( a\right) =a^{\tau }\left( X,D\right) \in \mathcal{B%
}\left( L^{2}\left( \mathbb{R}^{n}\right) \right)
\end{equation*}%
is continuous.
\end{theorem}

\begin{proof}
This theorem is a consequence of the previous theorem and the fact that it
is true for pseudo-differential operators with symbols in $S_{w}^{p}\left( 
\mathbb{R}^{n}\times \mathbb{R}^{n}\right) $ (see for instance \cite{Arsu}
or \cite{Toft1} for $1\leq p<\infty $ and \cite{Boulkhemair 2} for $p=\infty 
$).
\end{proof}

\begin{theorem}
Let $k\in \mathcal{K}_{pol}\left( \mathbb{R}^{n}\times \mathbb{R}^{n}\right) 
$ be such that $1/k\in L^{1}\left( \mathbb{R}^{n}\times \mathbb{R}%
^{n}\right) $ and $1\leq p<\infty $. If $\tau \in \mathtt{End}_{\mathbb{R}%
}\left( \mathbb{R}^{n}\right) \equiv M_{n\times n}\left( \mathbb{R}\right) $
and $a\in \mathcal{B}_{k^{\left\vert 1-2/p\right\vert }}^{p}\left( \mathbb{R}%
^{n}\times \mathbb{R}^{n}\right) $ then 
\begin{equation*}
\mathtt{Op}_{\tau }\left( a\right) =a^{\tau }\left( X,D\right) \in \mathcal{B%
}_{p}\left( L^{2}\left( \mathbb{R}^{n}\right) \right) .
\end{equation*}%
Moreover, the mapping%
\begin{equation*}
\mathtt{End}_{\mathbb{R}}\left( \mathbb{R}^{n}\right) \ni \tau \rightarrow 
\mathtt{Op}_{\tau }\left( a\right) =a^{\tau }\left( X,D\right) \in \mathcal{B%
}_{p}\left( L^{2}\left( \mathbb{R}^{n}\right) \right)
\end{equation*}%
is continuous.
\end{theorem}

\begin{proof}
The Schwartz kernel of the operator $\mathtt{Op}_{\tau }\left( a\right) $ is 
$\left( \left( 1\otimes \mathcal{F}^{-1}\right) a\right) \circ \mathtt{C}%
_{\tau }$. Therefore, $a\in \mathcal{B}_{1}^{2}\left( \mathbb{R}^{n}\times 
\mathbb{R}^{n}\right) \equiv L^{2}\left( \mathbb{R}^{n}\times \mathbb{R}%
^{n}\right) $ implies that $\mathtt{Op}_{\tau }\left( a\right) \in \mathcal{B%
}_{2}\left( L^{2}\left( \mathbb{R}^{n}\right) \right) $. Next we use the
interpolation properties of Kato-H\"{o}rmander spaces $\mathcal{B}_{k}^{p}$
and of the Schatten ideals $\mathcal{B}_{p}\left( L^{2}\left( \mathbb{R}%
^{n}\right) \right) $ to finish the theorem.%
\begin{equation*}
\begin{array}{c}
\left[ \mathcal{B}_{1}^{2}\left( \mathbb{R}^{n}\times \mathbb{R}^{n}\right) ,%
\mathcal{B}_{k}^{1}\left( \mathbb{R}^{n}\times \mathbb{R}^{n}\right) \right]
_{\frac{2}{p}-1}=\mathcal{B}_{k^{2/p-1}}^{p}\left( \mathbb{R}^{n}\times 
\mathbb{R}^{n}\right) \\ 
\left[ \mathcal{B}_{2}\left( L^{2}\left( \mathbb{R}^{n}\right) \right) ,%
\mathcal{B}_{1}\left( L^{2}\left( \mathbb{R}^{n}\right) \right) \right] _{%
\frac{2}{p}-1}=\mathcal{B}_{p}\left( L^{2}\left( \mathbb{R}^{n}\right)
\right)%
\end{array}%
,\quad 1\leq p\leq 2,
\end{equation*}%
\begin{equation*}
\begin{array}{c}
\left[ \mathcal{B}_{1}^{2}\left( \mathbb{R}^{n}\times \mathbb{R}^{n}\right) ,%
\mathcal{B}_{k}^{\infty }\left( \mathbb{R}^{n}\times \mathbb{R}^{n}\right) %
\right] _{1-\frac{2}{p}}=\mathcal{B}_{k^{1-2/p}}^{p}\left( \mathbb{R}%
^{n}\times \mathbb{R}^{n}\right) \\ 
\left[ \mathcal{B}_{2}\left( L^{2}\left( \mathbb{R}^{n}\right) \right) ,%
\mathcal{B}\left( L^{2}\left( \mathbb{R}^{n}\right) \right) \right] _{1-%
\frac{2}{p}}=\mathcal{B}_{p}\left( L^{2}\left( \mathbb{R}^{n}\right) \right)%
\end{array}%
,\quad 2\leq p<\infty .
\end{equation*}
\end{proof}

We shall end this section by considering pseudo-differential operators
defined by 
\begin{equation}
\mathtt{Op}\left( a\right) v\left( x\right) =\iint \mathtt{e}^{\mathtt{i}%
\left\langle x-y,\theta \right\rangle }a\left( x,y,\theta \right) v\left(
y\right) \mathtt{d}y\mathtt{d}\theta ,\quad v\in \mathcal{S}\left( \mathbb{R}%
^{n}\right) ,  \label{kh12}
\end{equation}%
with $a\in \mathcal{B}_{k}^{p}\left( 
\mathbb{R}
^{3n}\right) $, $k\in \mathcal{K}_{pol}\left( 
\mathbb{R}
^{3n}\right) $, $1/k\in L^{1}\left( 
\mathbb{R}
^{3n}\right) $ and $1\leq p\leq \infty $. For $a\in S_{w}^{\infty }\left( 
\mathbb{R}
^{3n}\right) $ such operators and Fourier integral operators were studied by
A. Boulkhemair in \cite{Boulkhemair 2}. In \cite{Boulkhemair 2}, the author
give a meaning to the above integral and proves $L^{2}$-boundedness of
global non-degenerate Fourier integral operators related to the Sj\"{o}%
strand class $S_{w}=S_{w}^{\infty }$. Therefore, by taking into account the
embedding theorem it follows that the above integral defines a bounded
operator in $L^{2}\left( \mathbb{R}^{n}\right) $. Now if we use Proposition
4.6. in \cite{Toft1} or Proposition 5.4. in \cite{Arsu} we obtain the
following result.

\begin{proposition}
Let $k\in \mathcal{K}_{pol}\left( 
\mathbb{R}
^{3n}\right) $ be such that $1/k\in L^{1}\left( 
\mathbb{R}
^{3n}\right) $, $1\leq p<\infty $ and $a\in \mathcal{B}_{k}^{p}\left( 
\mathbb{R}
^{3n}\right) $. If $\mathtt{Op}\left( a\right) $ is the operator defined by $%
(\ref{kh12})$, then $\mathtt{Op}\left( a\right) \in \mathcal{B}_{p}\left(
L^{2}\left( 
\mathbb{R}
^{n}\right) \right) $ and 
\begin{equation*}
\left\Vert \mathtt{Op}\left( a\right) \right\Vert _{\mathcal{B}_{p}\left(
L^{2}\left( \mathbb{R}^{n}\right) \right) }\leq Cst\left\Vert a\right\Vert _{%
\mathcal{B}_{k}^{p}}.
\end{equation*}
\end{proposition}

\end{document}